\documentclass{amsart}
\pdfoutput=1
\usepackage[utf8]{inputenc} 

\usepackage[T1]{fontenc}    

\usepackage{url}

\usepackage{amsmath}
\usepackage{amsfonts}
\usepackage{amssymb}
\usepackage{amsthm} %
\usepackage{mathrsfs} %
\usepackage{enumerate}

\usepackage{tikz}
\usepackage{pgfplots}
\usetikzlibrary{calc,shapes,intersections,patterns,arrows}
\usepackage{subfigure}

\usepackage{graphicx}
\usepackage{caption}
\usepackage[hyperfootnotes,colorlinks=true,citecolor=cyan,backref=page]{hyperref}

\theoremstyle{plain} 
\newtheorem{thm}{Theorem}[section]
\newtheorem{cor}[thm]{Corollary}
\newtheorem{prop}[thm]{Proposition}
\newtheorem{lem}[thm]{Lemma}

\theoremstyle{definition}
\newtheorem{defi}[thm]{Definition}

\newtheorem{remark}[thm]{Remark}

\newtheorem{question}{Question}
\newtheorem{ex}[thm]{Example}

\newcommand{\dom}{{\operatorname{dom}}}
\newcommand{\dep}{\operatorname{dp}}

\newcommand{\cone}{\operatorname{cone}}
\newcommand{\conv}{\operatorname{conv}}
\newcommand{\shi}{{\operatorname{Shi}}}

\newcommand{\op}{{\operatorname{op}}}
\newcommand{\precinv}{\mathrel{\ooalign{$\dot\prec$\cr
  \hidewidth\cr}}}

\author[M.~Dyer]{Matthew~Dyer}
\address{Department of Mathematics 
\\ 255 Hurley Building\\ University of Notre Dame \\
Notre Dame, Indiana 46556, U.S.A.}
\email{dyer.1@nd.edu}

\author[S. Fishel]{Susanna~Fishel}
\address[Susanna Fishel]{School of Mathematical and Statistical Sceinces\\
Arizona State University\\
Tempe AZ 85287, U.S.A.}
\email{sfishel1@asu.edu}

\author[C. Hohlweg]{Christophe~Hohlweg}
\address[Christophe Hohlweg]{Universit\'e du Qu\'ebec \`a Montr\'eal\\
LaCIM et D\'epartement de Math\'ematiques\\ CP 8888 Succ. Centre-Ville\\
Montr\'eal, Qu\'ebec, H3C 3P8\\ Canada}
\email{hohlweg.christophe@uqam.ca}
\urladdr{http://hohlweg.math.uqam.ca}

\author[A. Mark]{Alice~Mark}
\address[Alice Mark]{Vanderbilt University\\
Department of mathematics\\ Nashville, TN 37240, U.S.A.}
\email{alice.h.mark@vanderbilt.edu}

\title{Shi arrangements and low elements in Coxeter groups}

\keywords{Coxeter groups, low elements, Shi arrangements, Garside shadows}
\thanks{This work was partially supported by the Simons collaboration grants for Mathematicians \#359602 2$\backslash 709671$ held by Fisher and by the NSERC  grant {\em algebraic and geometric combinatorics of Coxeter groups} held by Hohlweg.}
\subjclass[2020]{Primary 20F55; secondary 17B22; 05E16; 51F15}

\begin{document}
\date{\today}
\maketitle

\begin{abstract}  Given an arbitrary Coxeter system $(W,S)$ and a nonnegative integer $m$, the $m$-Shi arrangement of $(W,S)$ is a subarrangement of the  Coxeter hyperplane arrangement of $(W,S)$. The classical Shi arrangement ($m=0$) was introduced in the case of affine Weyl groups by Shi to study Kazhdan-Lusztig cells for $W$. As two key results, Shi showed that each region of the Shi arrangement contains exactly one element of minimal length in $W$ and that the union of their inverses form a convex subset of the Coxeter complex. The set of $m$-low elements in $W$ were introduced to study the word problem of the corresponding Artin-Tits (braid) group and they turn out to produce automata to study the combinatorics of reduced words in $W$.

In this article, we generalize and extend Shi's results to any Coxeter system for any $m$: (1) the set of minimal length elements of the regions in a $m$-Shi arrangement is precisely the set of $m$-low elements, settling a conjecture of the first and third authors in this case; (2) the union of the inverses of the ($0$-)low elements form a convex subset in the Coxeter complex, settling a conjecture by the third author, Nadeau and Williams. 
\end{abstract}

\setcounter{tocdepth}{1}
\tableofcontents

\section{Introduction}   Let $(W,S)$ be a Coxeter system with length function $\ell:W\to \mathbb N$ and set of reflections 
$$
T=\bigcup_{w\in W}wSw^{-1}=\{s_\alpha\mid \alpha\in \Phi^+\},
$$
where $\Phi^+$ is a set of positive roots in a root system $\Phi$ for $(W,S)$.  As a discrete reflection group, $W$ acts on the {\em Tits cone} $\mathcal U(W,S)$, in which each reflection $s_\alpha\in T$ is uniquely determined by a hyperplane $H_\alpha$. The collection of hyperplanes 
$$
\mathcal A(W,S)=\{H_\alpha\mid \alpha\in \Phi^+\}
$$ 
is called the {\em Coxeter arrangement of $(W,S)$}.  The connected components of the complement of $\mathcal A(W,S)$ in $\mathcal U(W,S)$ are called {\em chambers} and the map $w\mapsto C_w$ is a bijection between  $W$ and the set of chambers; see for instance~\cite{Hu90} for more details. 

\smallskip
Let $m\in \mathbb N$. A positive root $\beta\in\Phi^+$ is {\em $m$-small\footnote{Small roots are called elementary roots in \cite{BrHo93,Fu12}}} if there are  at most $m$ parallel, or ultraparallel, hyperplanes separating $H_\beta$  from the fundamental chamber $C_e$ (not counting $H_\beta$). Denote by $\Sigma_m$ the set of $m$-small roots. For $m=0$, we simply denote $\Sigma=\Sigma_0$ the set of small roots. Small roots were introduced by Brink and Howlett to prove that any finitely generated Coxeter system is automatic~\cite{BrHo93}; a key and remarkable result in their article was to prove that $\Sigma$ is a finite set. Later,  Fu~\cite{Fu12} proved that $\Sigma_m$ is finite for all $m\in\mathbb N$. The sets of $m$-small roots are the building blocks of a family of regular automata that recognize the language of reduced words in $(W,S)$; see~\cite{Ed09,HoNaWi16,PaYa21}.

The {\em $m$-Shi arrangement $\shi_m(W,S)$ of $(W,S)$} is the hyperplane subarrangement of $\mathcal A(W,S)$:
 $$
 \shi_m(W,S) = \{H_\alpha\mid \alpha\in \Sigma_m \} .
 $$
 The regions of $\shi_m(W,S)$ are a  union of chambers and define therefore an equivalence relation $\sim_{m}$  on $W$. It is conjectured in \cite[Conjecture 2]{DyHo16} that each equivalence class under $\sim_{m}$ contains a unique minimal length element and that the set of these minimal length elements is the set of $m$-low elements. An element $w\in W$ is  {\em $m$-low} if the inversion set $\Phi(w)$ of $w$ is spanned by the $m$-small roots it contains. The set $L_m$  of $m$-low elements turns out to be a finite Garside shadow~\cite{DyHo16,Dy21}, which corresponds to a finite Garside family in the corresponding Artin-Tits group; see for instance~\cite{DDH14} and the references therein. 

 The following two theorems are the main results of this article: the first theorem settles \cite[Conjecture~2]{DyHo16} and the second settles~\cite[Conjecture~3]{HoNaWi16}.

\begin{thm}\label{thm:Main1} Let $(W,S)$ be a Coxeter system and $m\in \mathbb N$.
\begin{enumerate}
\item  Each region of $\shi_m(W,S)$ contains a unique element of minimal length. 
\item The set of the minimal length elements of $\shi_m(W,S)$ is equal to the set~$L_m$ of $m$-low elements.
\end{enumerate}
\end{thm}

An noticeable consequence of the above theorem and of the fact that $L_m$ is a Garside shadow is that if the join $z$ (under the right weak order) of two minimal elements of $\shi_m(W,S)$ exists, then $z$ is also the minimal element of a region of $\shi_m(W,S)$. 

\begin{thm}\label{thm:Main2} Let $(W,S)$ be a Coxeter system.  The union of the chambers $C_w$ for $w^{-1}\in L_0$ is a convex set. 
\end{thm}

These theorems are illustrated in Figures~\ref{fig:Aff1},~\ref{fig:Aff2}~and~\ref{fig:Hyp1}.  The proofs of these theorems are consequence of a {\em sandwich property} of {\em short inversion posets},  proven in~\S\ref{se:ShortInv}.  The first author showed in~\cite{Dy19} that the {\em inversion set $\Phi(w)$ of $w\in W$} is spanned by its {\em set of short inversions $\Phi^1(w)$}. We endow $\Phi^1(w)$ with a poset structure arising from the configuration of maximal dihedral reflection subgroups: $\alpha\precinv_w \beta$ if $\beta$ is not a simple root of the maximal dihedral reflection subgroup containing $\alpha,\beta\in \Phi^1(w)$, see \S\ref{se:PosetBasisInversion}. Then we prove that any short inversion $\beta\in \Phi^1(w)$ is {\em sandwiched} between a left-descent root and a right-descent root (that are naturally defined from the left and right descent sets of $w$); this is Theorem~\ref{thm:BasisInversion}, which is the core result of this article. 


Coxeter arrangements and their subarrangements such as $m$-Shi arrangements  may be naturally realized in various ambient spaces such as  {\em Tits cone} (see~\cite[\S5.13]{Hu90}), {\em Coxeter complexes} and  {\em Davis complexes} (see~\cite[\S3 and \S12.3]{AbBr08}), Euclidean and hyperbolic spaces,   {\em imaginary cones}~\cite{Dy19-1,DyHoRi16}, etc.  In order to facilitate a uniform treatment of conditions under which our results apply, we define  chambered sets in~\S\ref{se:CS}.   

A {\em chambered set} for a Coxeter system $(W,S)$ is defined  to be a pair $(U,C)$ of a $W$-set $U$ and a distinguished set $C$ of $W$-orbit representatives on $U$, called the fundamental chamber; the stabilizers of points of $U$ are required to be standard parabolic subgroups of $(W,S)$. In the setting of chambered sets,  one may define chambers, facets, hyperplanes, half-spaces, convex subsets etc,   and isolate simple conditions under which various favorable standard properties of these notions (such as that $W$ acts simply transitively on the chambers, or that the fundamental chamber is not contained in the union of the other chambers, or that the hyperplanes are in natural bijection with $T$, etc) hold.

 For certain examples, such as arrangements in the Tits cone,   Euclidean space or  hyperbolic space in~\S\ref{se:LowShi}  the properties we need for our results on $m$-Shi arrangements to hold are very well known and reference to chambered sets is unnecessary while reading \S\ref{se:LowShi}. For others, it is a  more delicate matter; the necessary properties may or may not hold in  Coxeter complexes, depending on whether they realized as simplicial complexes or chamber complexes, for example,  and they  hold in imaginary cones of irreducible, infinite,  finite rank Coxeter groups but not in general.  By stating our results in the context of chambered sets, we are able to supply simple uniform conditions whose validity determines whether the results apply in any particular situation without having to discuss the examples individually. In any case, a natural category of  chambered sets forms part of  rich  structure observed by the first author, involving various natural  monoidal categories and their actions, bicategories etc associated to Coxeter groups, and it is natural to present their definition initially in the context which motivated it and to which that structure has further applications.


In \S\ref{se:LowShi}, we introduce extended Shi arrangements  and we prove Theorem~\ref{thm:Main1} and  Theorem~\ref{thm:Main2}. Combinatorics of roots and reduced words are surveyed in \S\ref{se:Combin} while $m$-small roots and $m$-low elements are discussed in~\S\ref{se:Low}.

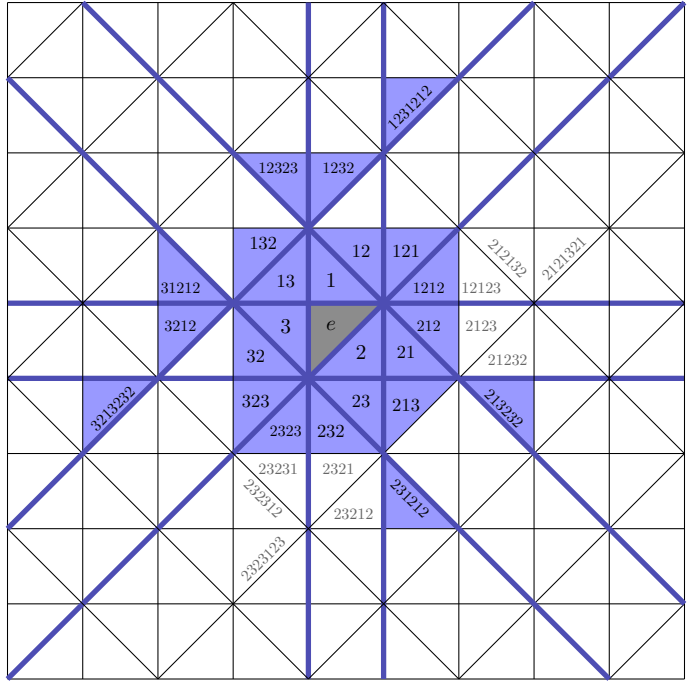
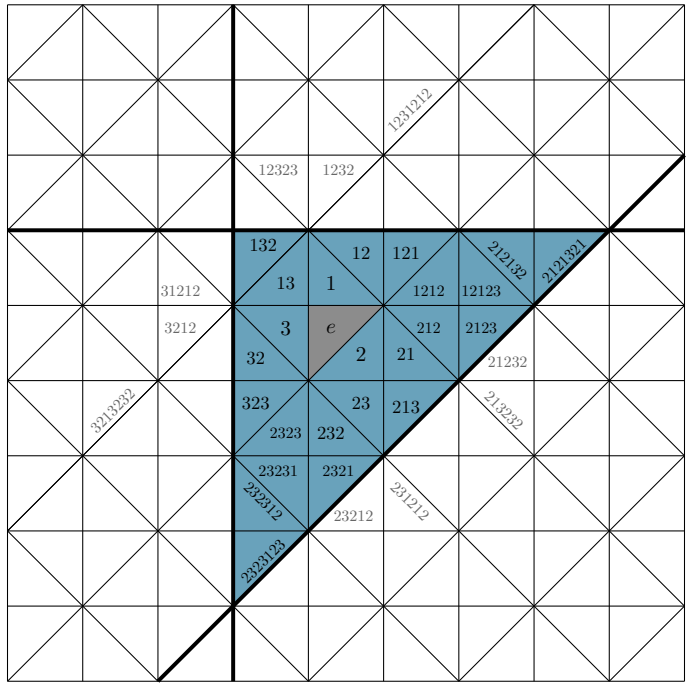
\begin{figure}
\centering
\subfigure[The classical ($m=0$) Shi arrangement. The chambers indexed by the $0$-low elements are shaded.]{
\begin{tikzpicture}[scale=1]
\path [fill=gray!90] (0,0) -- (1,1) --  (0,1) --  (0,0);
\path [fill=blue!40] (-1,1)--(0,2)--(2,2)--(2,0)--(1,-1)--(0,-1)--(0,0)--(1,1)--(-1,1);
\path [fill=blue!40] (0,2)--(-1,3)--(1,3)--(0,2);
\path [fill=blue!40] (1,3)--(1,4)--(2,4)--(1,3);
\path [fill=blue!40] (-1,1)--(0,1)--(0,-1)--(-1,-1)--(-1,1);
\path [fill=blue!40] (-1,1)--(-2,0)--(-2,2)--(-1,1);
\path [fill=blue!40] (-2,0)--(-3,0)--(-3,-1)--(-2,0);
\path [fill=blue!40] (1,-1)--(1,-2)--(2,-2)--(1,-1);
\path [fill=blue!40] (2,0)--(3,0)--(3,-1)--(2,0);
\path [fill=blue!40] (-1,1)--(-1,2)--(0,2)--(-1,1);
\node[scale=0.8] at (0.3,0.7) {$e$};
\node[scale=0.8]  at (0.3,1.3) {$1$};
\node[scale=0.8] at (0.7,0.35) {$2$};
\node[scale=0.8]  at (-0.3,0.7) {$3$};
\node[scale=0.7] at (-0.3,1.3) {$13$};
\node[scale=0.7]  at (-0.6,1.8) {$132$};
\node[scale=0.7] at (0.3,-0.7) {$232$};
\node[scale=0.7] at (0.7,-0.3) {$23$};
\node[scale=0.6] at (-0.3,-0.7) {$2323$};
\node[scale=0.7] at (-0.7,0.3) {$32$};
\node[scale=0.7] at (-0.7,-0.3) {$323$};
\node[scale=0.7] at (1.3,0.35) {$21$};
\node[scale=0.7] at (1.3,-0.35) {$213$};
\node[scale=0.6,rotate=-45] at (1.35,-1.6) {$231212$};
\node[scale=0.7]  at (0.7,1.7) {$12$};
\node[scale=0.7]  at (1.3,1.7) {$121$};
\node[scale=0.6]  at (1.6,1.2) {$1212$};
\node[scale=0.6]  at (1.6,0.7) {$212$};
\node[scale=0.6,rotate=-45] at (2.6,-0.4) {$213232$};
\node[scale=0.6]  at (-1.7,0.7) {$3212$};
\node[scale=0.6]  at (-1.7,1.2) {$31212$};
\node[scale=0.6] at (0.4,2.8) {$1232$};
\node[scale=0.6] at (-0.4,2.8) {$12323$};
\node[scale=0.6,rotate=45] at (1.35,3.6) {$1231212$};
\node[scale=0.6,rotate=45] at (-2.6,-0.4) {$3213232$};
\node[scale=0.6,color=black!60]  at (2.3,1.2) {$12123$};
\node[scale=0.6,color=black!60,rotate = -45]  at (2.65,1.6) {$212132$};
\node[scale=0.6,color=black!60,rotate = 45]  at (3.4,1.6) {$2121321$};
\node[scale=0.6,color=black!60]  at (2.3,0.7)  {$2123$};
\node[scale=0.6,color=black!60]  at (2.65,0.25)  {$21232$};
\node[scale=0.6,color=black!60]  at (0.4,-1.2) {$2321$};
\node[scale=0.6,color=black!60]  at (-0.4,-1.2) {$23231$};
\node[scale=0.6,color=black!60,rotate=-45]  at (-0.6,-1.6) {$232312$};
\node[scale=0.6,color=black!60]  at (0.6,-1.8) {$23212$};
\node[scale=0.6,color=black!60,rotate=45]  at (-0.6,-2.4) {$2323123$};
\draw (-4,-2) -- (5,-2); 
\draw (-4,-1) -- (5,-1);
\draw [line width=2pt,color=blue!40!gray](-4,0) -- (5,0);
\draw [line width=2pt,color=blue!40!gray](-4,1) -- (5,1);
\draw (-4,2) -- (5,2);
\draw (-4,3) -- (5,3); 
\draw (-4,4) -- (5,4);
\draw (-4,5) -- (5,5);
\draw (-4,-3) -- (5,-3);
\draw (-4,-4) -- (5,-4);
\draw (-2,-4) -- (-2,5);
\draw (-1,-4) -- (-1,5);
\draw [line width=2pt,color=blue!40!gray](0,-4) -- (0,5);
\draw [line width=2pt,color=blue!40!gray](1,-4) -- (1,5);
\draw (2,-4) -- (2,5);
\draw (-4,-4) -- (-4,5);
\draw (-3,-4) -- (-3,5);
\draw (3,-4) -- (3,5);
\draw (4,-4) -- (4,5);
\draw (5,-4) -- (5,5);
\draw (-4,4)--(-3,5);
\draw (-4,2)--(-1,5);
\draw (-4,0) -- (1,5);
\draw (-4,-2) -- (3,5);
\draw [line width=2pt,color=blue!40!gray](-4,-4) -- (5,5);
\draw (-2,-4) -- (5,3);
\draw (0,-4) -- (5,1);
\draw (2,-4)--(5,-1);
\draw (4,-4)--(5,-3);
\draw (3,5)--(5,3);
\draw (1,5)--(5,1);
\draw (-1,5)--(5,-1);
\draw[line width=2pt,color=blue!40!gray] (-3,5)--(5,-3);
\draw[line width=2pt,color=blue!40!gray](-4,4)--(4,-4);
\draw (-4,2)--(2,-4);
\draw (-4,0)--(0,-4);
\draw (-4,-2)--(-2,-4);
\draw[line width=2pt,color=blue!40!gray] (-4,-2)--(3,5);
\end{tikzpicture}
}\quad\qquad
\subfigure[The $0$-Shi polyhedron]{
\begin{tikzpicture}[scale=1]
\path [fill=cyan!40!gray] (-1,2)--(4,2)--(-1,-3)--(-1,2);
\path [fill=gray!90] (0,0) -- (1,1) --  (0,1) --  (0,0);
\node[scale=0.8] at (0.3,0.7) {$e$};
\node[scale=0.8]  at (0.3,1.3) {$1$};
\node[scale=0.8] at (0.7,0.35) {$2$};
\node[scale=0.8]  at (-0.3,0.7) {$3$};
\node[scale=0.7] at (-0.3,1.3) {$13$};
\node[scale=0.7]  at (-0.6,1.8) {$132$};
\node[scale=0.7] at (0.3,-0.7) {$232$};
\node[scale=0.7] at (0.7,-0.3) {$23$};
\node[scale=0.6] at (-0.3,-0.7) {$2323$};
\node[scale=0.7] at (-0.7,0.3) {$32$};
\node[scale=0.7] at (-0.7,-0.3) {$323$};
\node[scale=0.7] at (1.3,0.35) {$21$};
\node[scale=0.7] at (1.3,-0.35) {$213$};
\node[scale=0.6,rotate=-45,color=black!60] at (1.35,-1.6) {$231212$};
\node[scale=0.7]  at (0.7,1.7) {$12$};
\node[scale=0.7]  at (1.3,1.7) {$121$};
\node[scale=0.6]  at (1.6,1.2) {$1212$};
\node[scale=0.6]  at (1.6,0.7) {$212$};
\node[scale=0.6,rotate=-45,color=black!60] at (2.6,-0.4) {$213232$};
\node[scale=0.6,color=black!60]  at (-1.7,0.7) {$3212$};
\node[scale=0.6,color=black!60]  at (-1.7,1.2) {$31212$};
\node[scale=0.6,color=black!60] at (0.4,2.8) {$1232$};
\node[scale=0.6,color=black!60] at (-0.4,2.8) {$12323$};
\node[scale=0.6,rotate=45,color=black!60] at (1.35,3.6) {$1231212$};
\node[scale=0.6,rotate=45,color=black!60] at (-2.6,-0.4) {$3213232$};
\node[scale=0.6]  at (2.3,1.2) {$12123$};
\node[scale=0.6,rotate = -45]  at (2.65,1.6) {$212132$};
\node[scale=0.6,rotate = 45]  at (3.4,1.6) {$2121321$};
\node[scale=0.6]  at (2.3,0.7)  {$2123$};
\node[scale=0.6,color=black!60]  at (2.65,0.25)  {$21232$};
\node[scale=0.6]  at (0.4,-1.2) {$2321$};
\node[scale=0.6]  at (-0.4,-1.2) {$23231$};
\node[scale=0.6,rotate=-45]  at (-0.6,-1.6) {$232312$};
\node[scale=0.6,color=black!60]  at (0.6,-1.8) {$23212$};
\node[scale=0.6,rotate=45]  at (-0.6,-2.4) {$2323123$};
\draw (-4,-2) -- (5,-2); 
\draw (-4,-1) -- (5,-1);
\draw (-4,0) -- (5,0);
\draw (-4,1) -- (5,1);
\draw[line width=1.5pt] (-4,2) -- (5,2);
\draw (-4,3) -- (5,3); 
\draw (-4,4) -- (5,4);
\draw (-4,5) -- (5,5);
\draw (-4,-3) -- (5,-3);
\draw (-4,-4) -- (5,-4);
\draw (-2,-4) -- (-2,5);
\draw[line width=1.5pt] (-1,-4) -- (-1,5);
\draw (0,-4) -- (0,5);
\draw (1,-4) -- (1,5);
\draw (2,-4) -- (2,5);
\draw (-4,-4) -- (-4,5);
\draw (-3,-4) -- (-3,5);
\draw (3,-4) -- (3,5);
\draw (4,-4) -- (4,5);
\draw (5,-4) -- (5,5);
\draw (-4,4)--(-3,5);
\draw (-4,2)--(-1,5);
\draw (-4,0) -- (1,5);
\draw (-4,-2) -- (3,5);
\draw (-4,-4) -- (5,5);
\draw[line width=1.5pt] (-2,-4) -- (5,3);
\draw (0,-4) -- (5,1);
\draw (2,-4)--(5,-1);
\draw (4,-4)--(5,-3);
\draw (3,5)--(5,3);
\draw (1,5)--(5,1);
\draw (-1,5)--(5,-1);
\draw (-3,5)--(5,-3);
\draw (-4,4)--(4,-4);
\draw (-4,2)--(2,-4);
\draw (-4,0)--(0,-4);
\draw (-4,-2)--(-2,-4);
\draw (-4,-2)--(3,5);
\end{tikzpicture}
}
\caption{The ($0$-)Shi arrangement and the $0$-Shi polyhedron  for $\tilde{B}_2$.}\label{fig:Aff1}
\end{figure}

\begin{figure}
\centering
\subfigure[The $0$ and $1$-Shi arrangements. The chambers indexed by the $1$-low elements are shaded (with the chambers indexed by the  $0$-low elements in a darker shade).]{
\begin{tikzpicture}[scale=0.7]
\path [fill=gray!90] (0,0) -- (1,1) -- (0,1) -- (0,0);
\path [fill=blue!50] (-1,1)--(0,2)--(2,2)--(2,0)--(1,-1)--(0,-1)--(0,0)--(1,1)--(-1,1);
\path [fill=blue!50] (0,2)--(-1,3)--(1,3)--(0,2);
\path [fill=blue!50] (1,3)--(1,4)--(2,4)--(1,3);
\path [fill=blue!50] (-1,1)--(0,1)--(0,-1)--(-1,-1)--(-1,1);
\path [fill=blue!50] (-1,1)--(-2,0)--(-2,2)--(-1,1);
\path [fill=blue!50] (-2,0)--(-3,0)--(-3,-1)--(-2,0);
\path [fill=blue!50] (1,-1)--(1,-2)--(2,-2)--(1,-1);
\path [fill=blue!50] (2,0)--(3,0)--(3,-1)--(2,0);
\path [fill=blue!50] (-1,1)--(-1,2)--(0,2)--(-1,1);
\path [fill=blue!35] (1,-1)--(2,0)--(3,-1)--(2,-2)--(1,-1);
\path [fill=blue!35] (3,-1)--(4,-1)--(4,-2)--(3,-1);
\path [fill=blue!35] (5,-1)--(6,-1)--(6,-2)--(5,-1);
\path [fill=blue!35] (4,0)--(5,0)--(5,-1)--(4,0);
\path [fill=blue!35] (2,0)--(2,3)--(3,3)--(3,0)--(2,0);
\path [fill=blue!35] (4,2)--(5,3)--(5,2)--(4,2);
\path [fill=blue!35] (3,1)--(4,2)--(4,0)--(3,1);
\path [fill=blue!35] (0,2)--(1,3)--(2,3)--(2,2)--(0,2);
\path [fill=blue!35] (-2,-1)--(1,-1)--(1,-2)--(-2,-2)--(-2,-1);
\path [fill=blue!35] (0,-2)--(-1,-3)--(1,-3)--(0,-2);
\path [fill=blue!35] (1,-3)--(1,-4)--(2,-4)--(1,-3);
\path [fill=blue!35] (2,-4)--(2,-5)--(3,-5)--(2,-4);
\path [fill=blue!35] (2,-2)--(2,-3)--(3,-3)--(2,-2);
\path [fill=blue!35] (-1,-3)--(-2,-4)--(-1,-4)--(-1,-3);
\path [fill=blue!35] (-1,-1)--(-2,-1)--(-2,0)--(-1,1)--(-1,-1);
\path [fill=blue!35] (1,5)--(1,6)--(2,6)--(1,5);
\path [fill=blue!35] (2,6)--(2,7)--(3,7)--(2,6);
\path [fill=blue!35] (2,4)--(2,5)--(3,5)--(2,4);
\path [fill=blue!35] (0,4)--(-1,5)--(1,5)--(0,4);
\path [fill=blue!35] (0,4)--(1,4)--(1,3)--(0,4);
\path [fill=blue!35] (-1,3)--(-2,4)--(0,4)--(-1,3);
\path [fill=blue!35] (-1,5)--(-2,6)--(-1,6)--(-1,5);
\path [fill=blue!35] (0,2)--(-1,3)--(-2,3)--(-2,2)--(0,2);
\path [fill=blue!35] (-1,2)--(-1,1)--(-2,2)--(-1,2);
\path [fill=blue!35] (-3,1)--(-4,0)--(-4,2)--(-3,1);
\path [fill=blue!35] (-2,2)--(-3,1)--(-3,3)--(-2,2);
\path [fill=blue!35] (-3,1)--(-2,0)--(-3,0)--(-3,1);
\path [fill=blue!35] (-5,3)--(-4,2)--(-5,2)--(-5,3);
\path [fill=blue!35] (-5,-1)--(-4,0)--(-5,0)--(-5,-1);
\path [fill=blue!35] (-6,-2)--(-5,-1)--(-6,-1)--(-6,-2);
\path [fill=blue!35] (-4,-2)--(-3,-1)--(-4,-1)--(-4,-2);
\node[scale=0.7] at (0.3,0.7) {$e$};
\node[scale=0.7]  at (0.3,1.3) {$1$};
\node[scale=0.7] at (0.7,0.35) {$2$};
\node[scale=0.7]  at (-0.3,0.7) {$3$};
\node[scale=0.6] at (-0.3,1.3) {$13$};
\node[scale=0.6]  at (-0.6,1.8) {$132$};
\node[scale=0.6] at (0.3,-0.7) {$232$};
\node[scale=0.6] at (0.7,-0.3) {$23$};
\node[scale=0.5] at (-0.3,-0.7) {$2323$};
\node[scale=0.6] at (-0.7,0.3) {$32$};
\node[scale=0.6] at (-0.7,-0.3) {$323$};
\node[scale=0.6] at (1.3,0.35) {$21$};
\node[scale=0.6] at (1.3,-0.35) {$213$};
\node[scale=0.5,rotate=-45] at (1.35,-1.6) {$231212$};
\node[scale=0.6]  at (0.7,1.7) {$12$};
\node[scale=0.6]  at (1.3,1.7) {$121$};
\node[scale=0.5]  at (1.6,1.2) {$1212$};
\node[scale=0.5]  at (1.6,0.7) {$212$};
\node[scale=0.5,rotate=-45] at (2.6,-0.4) {$213232$};
\node[scale=0.5]  at (-1.7,0.7) {$3212$};
\node[scale=0.5]  at (-1.7,1.2) {$31212$};
\node[scale=0.5] at (0.4,2.8) {$1232$};
\node[scale=0.5] at (-0.4,2.8) {$12323$};
\node[scale=0.5,rotate=45] at (1.35,3.6) {$1231212$};
\node[scale=0.5,rotate=45] at (-2.6,-0.4) {$3213232$};
\draw (-6,7)--(7,7);
\draw (-6,6)--(7,6);
\draw (-6,6)--(7,6);
\draw (-6,-2) -- (7,-2); 
\draw[line width=2pt,color=blue!20!gray] (-6,-1) -- (7,-1);
\draw [line width=2pt,color=blue!40!gray](-6,0) -- (7,0);
\draw[line width=2pt,color=blue!40!gray](-6,1) -- (7,1);
\draw[line width=2pt,color=blue!20!gray](-6,2) -- (7,2);
\draw (-6,3) -- (7,3); 
\draw (-6,4) -- (7,4);
\draw (-6,5) -- (7,5);
\draw (-6,-3) -- (7,-3);
\draw (-6,-4) -- (7,-4);
\draw (-6,-5)--(7,-5);
\draw (-6,-6)--(7,-6);
\draw (7,-6)--(7,7);
\draw (6,-6)--(6,7);
\draw (-6,-6)--(-6,7);
\draw (-5,-6)--(-5,7);
\draw (-2,-6) -- (-2,7);
\draw [line width=2pt,color=blue!20!gray](-1,-6) -- (-1,7);
\draw [line width=2pt,color=blue!40!gray](0,-6) -- (0,7);
\draw [line width=2pt,color=blue!40!gray](1,-6) -- (1,7);
\draw [line width=2pt,color=blue!20!gray](2,-6) -- (2,7);
\draw (-4,-6) -- (-4,7);
\draw (-3,-6) -- (-3,7);
\draw (3,-6) -- (3,7);
\draw (4,-6) -- (4,7);
\draw (5,-6) -- (5,7);
\draw (-6,6)--(-5,7);
\draw (-6,4)--(-3,7);
\draw (-6,2)--(-1,7);
\draw (-6,0)--(1,7);
\draw[line width=2pt,color=blue!20!gray] (-6,-2) -- (3,7);
\draw[line width=2pt,color=blue!40!gray] (-6,-4)--(5,7);
\draw [line width=2pt,color=blue!40!gray](-6,-6) -- (7,7);
\draw [line width=2pt,color=blue!20!gray] (-4,-6) -- (7,5);
\draw (-2,-6) -- (7,3);
\draw (0,-6)--(7,1);
\draw (2,-6)--(7,-1);
\draw (4,-6)--(7,-3);
\draw (6,-6)--(7,-5);
\draw  (5,7)--(7,5);
\draw  (3,7)--(7,3);
\draw  (1,7)--(7,1);
\draw (-1,7)--(7,-1);
\draw[line width=2pt,color=blue!20!gray] (-3,7)--(7,-3);
\draw[line width=2pt,color=blue!40!gray] (-5,7)--(7,-5);
\draw[line width=2pt,color=blue!40!gray](-6,6)--(6,-6);
\draw[line width=2pt,color=blue!20!gray] (-6,4)--(4,-6);
\draw (-6,2)--(2,-6);
\draw (-6,0)--(0,-6);
\draw (-6,-2)--(-2,-6);
\draw (-6,-4)--(-4,-6);
\end{tikzpicture}
}\quad\qquad
\subfigure[The $1$-Shi polyhedron (with the $0$-Shi polyhedron inside)]{
\begin{tikzpicture}[scale=0.7]
\path [fill=lightgray] (-2,3)--(7,3)--(-2,-6)--(-2,3);
\path [fill=cyan!40!gray] (-1,2)--(4,2)--(-1,-3)--(-1,2);
\path [fill=white] (0,0) -- (1,1) --  (0,1) --  (0,0);
\path [fill=gray!90] (0,0) -- (1,1) --  (0,1) --  (0,0);

\draw (-6,7)--(7,7);
\draw (-6,6)--(7,6);
\draw (-6,-2) -- (7,-2); 
\draw (-6,-1) -- (7,-1);
\draw (-6,0) -- (7,0);
\draw (-6,1) -- (7,1);
\draw (-6,2) -- (7,2);
\draw[line width=1.5pt] (-6,3) -- (7,3); 
\draw (-6,4) -- (7,4);
\draw (-6,5) -- (7,5);
\draw (-6,-3) -- (7,-3);
\draw (-6,-4) -- (7,-4);
\draw (-6,-5)--(7,-5);
\draw (-6,-6)--(7,-6);
\draw (7,-6)--(7,7);
\draw (6,-6)--(6,7);
\draw (-6,-6)--(-6,7);
\draw (-5,-6)--(-5,7);
\draw[line width=1.5pt] (-2,-6) -- (-2,7);
\draw (-1,-6) -- (-1,7);
\draw  (0,-6) -- (0,7);
\draw  (1,-6) -- (1,7);
\draw (2,-6) -- (2,7);
\draw (-4,-6) -- (-4,7);
\draw (-3,-6) -- (-3,7);
\draw (3,-6) -- (3,7);
\draw (4,-6) -- (4,7);
\draw (5,-6) -- (5,7);
\draw (-6,6)--(-5,7);
\draw (-6,4)--(-3,7);
\draw (-6,2)--(-1,7);
\draw (-6,0)--(1,7);
\draw  (-6,-2) -- (3,7);
\draw  (-6,-4)--(5,7);
\draw (-6,-6) -- (7,7);
\draw  (-4,-6) -- (7,5);
\draw[line width=1.5pt] (-2,-6) -- (7,3);
\draw (0,-6)--(7,1);
\draw (2,-6)--(7,-1);
\draw (4,-6)--(7,-3);
\draw (6,-6)--(7,-5);
\draw  (5,7)--(7,5);
\draw  (3,7)--(7,3);
\draw  (1,7)--(7,1);
\draw (-1,7)--(7,-1);
\draw  (-3,7)--(7,-3);
\draw  (-5,7)--(7,-5);
\draw (-6,6)--(6,-6);
\draw (-6,4)--(4,-6);
\draw (-6,2)--(2,-6);
\draw (-6,0)--(0,-6);
\draw (-6,-2)--(-2,-6);
\draw (-6,-4)--(-4,-6);
\end{tikzpicture}
}
\caption{The $1$-Shi arrangement and the $1$-Shi polyhedron  for $\tilde{B}_2$.}\label{fig:Aff2}
\end{figure}
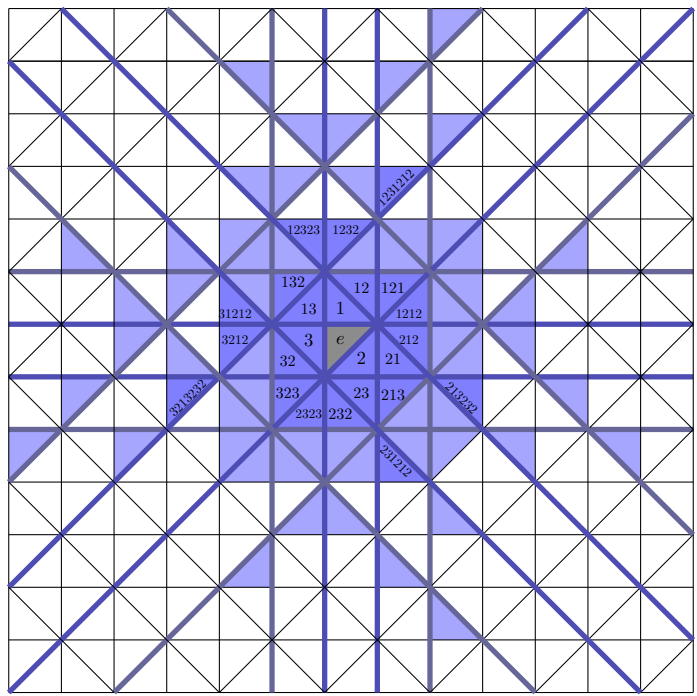
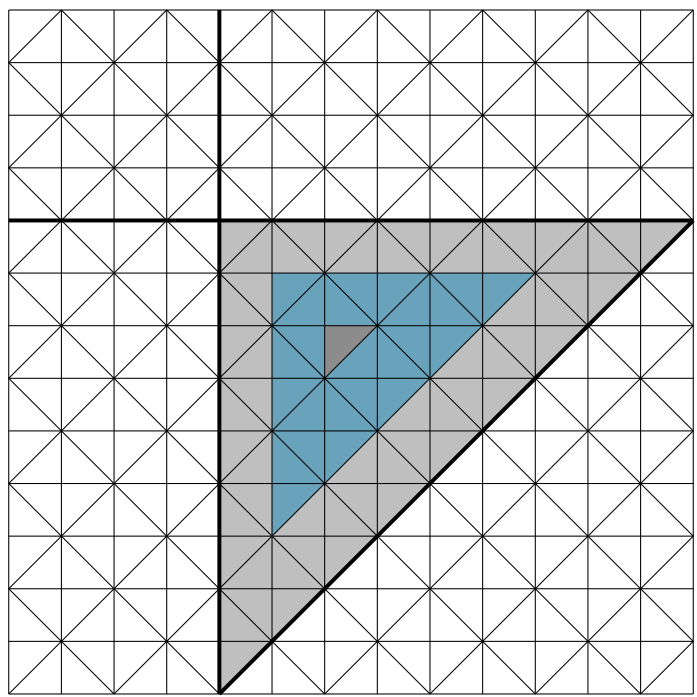


\begin{figure}
\centering
\subfigure[The $0$-Shi arrangement. The chambers indexed by the $0$-low elements are shaded. ]{
\resizebox{0.8\hsize}{!}{
\begin{tikzpicture}
	[scale=1,
	 q/.style={teal,line join=round},
	 racine/.style={blue},
	 racinesimple/.style={blue},
	 racinedih/.style={blue},
	 sommet/.style={inner sep=2pt,circle,draw=black,fill=blue,thick,anchor=base},
	 rotate=0]
 \tikzstyle{every node}=[font=\small]
\def\grosseursimple{0.025}
\node[anchor=center,inner sep=0pt] at (3.05,2.9) {\includegraphics[width=12cm]{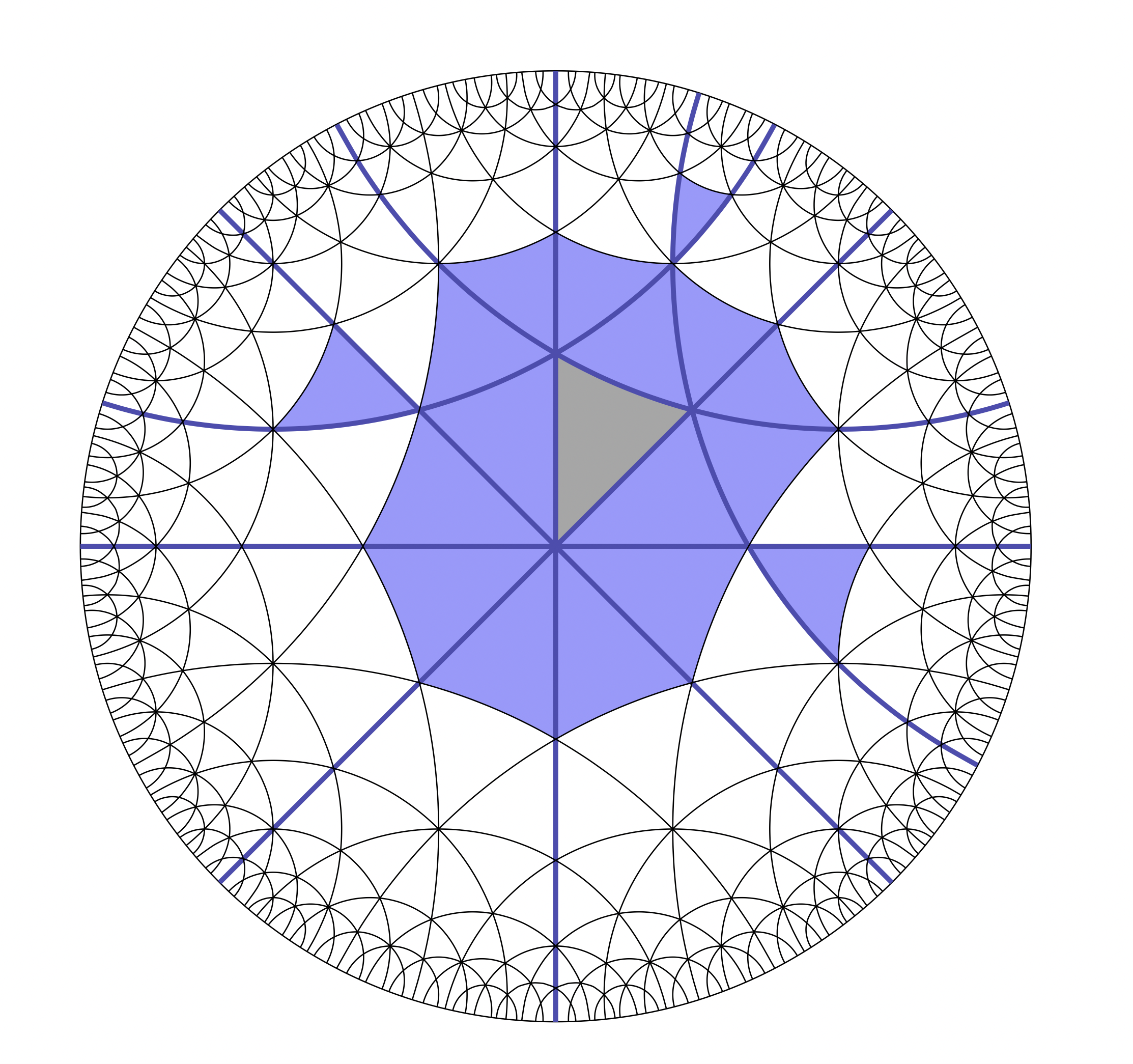}};
\coordinate (ancre) at (-3,6);

\node[sommet,label=above:$1$] (gamma) at ($(ancre)+(0.5,0.83)$) {};
\node[sommet,label=below right :$2$] (beta) at ($(ancre)+(1,0)$) {} edge[thick]  (gamma) ;
\node[sommet,label=below left:$3$] (alpha) at (ancre) {}
 edge[thick] node[auto,swap] {$4$} (beta)  edge[thick] (gamma)  ;

\node[scale=1] at (3.5,4) {$e$};
\node[scale=1]  at (3.8,5) {$1$};
\node[scale=1]  at (4.2,3.5) {$2$};
\node[scale=1] at (2.5,4) {$3$};
\node[scale=0.8] at (4.6,5) {$12$};
\node[scale=0.8] at (5.2,4.5) {$121$};
\node[scale=0.8] at (5,3.6) {$21$};
\node[scale=0.8] at (2.2,5) {$31$};
\node[scale=0.8] at (2.6,5.5) {$313$};
\node[scale=0.8] at (3.4,5.5) {$13$};
\node[scale=0.6,rotate=60] at (4.45,6.25) {$13232$};
\node[scale=0.8] at (2,3.3) {$32$};
\node[scale=0.8] at (2,2.5) {$323$};
\node[scale=0.8] at (2.5,1.75) {$3232$};
\node[scale=0.8] at (3.5,1.75) {$232$};
\node[scale=0.8] at (4,2.5) {$23$};
\node[scale=0.8] at (5.7,2.5) {$2131$};
\node[scale=0.8] at (0.8,4.5) {$3121$};
\node[scale=0.6,color=black!60]   at (2.25,0.75) {$32321$};
\node[scale=0.6,color=black!60] at (2.4,6.1) {$1312$};
\node[scale=0.6,color=black!60] at (5.7,4.8) {$1213$};


\end{tikzpicture}}}
\subfigure[The $0$-Shi polyhedron]{
\resizebox{0.8\hsize}{!}{
\begin{tikzpicture}
	[scale=1,
	 q/.style={teal,line join=round},
	 racine/.style={blue},
	 racinesimple/.style={blue},
	 racinedih/.style={blue},
	 sommet/.style={inner sep=2pt,circle,draw=black,fill=blue,thick,anchor=base},
	 rotate=0]
 \tikzstyle{every node}=[font=\small]
\def\grosseursimple{0.025}
\node[anchor=center,inner sep=0pt] at (2.9,3) {\includegraphics[width=12cm]{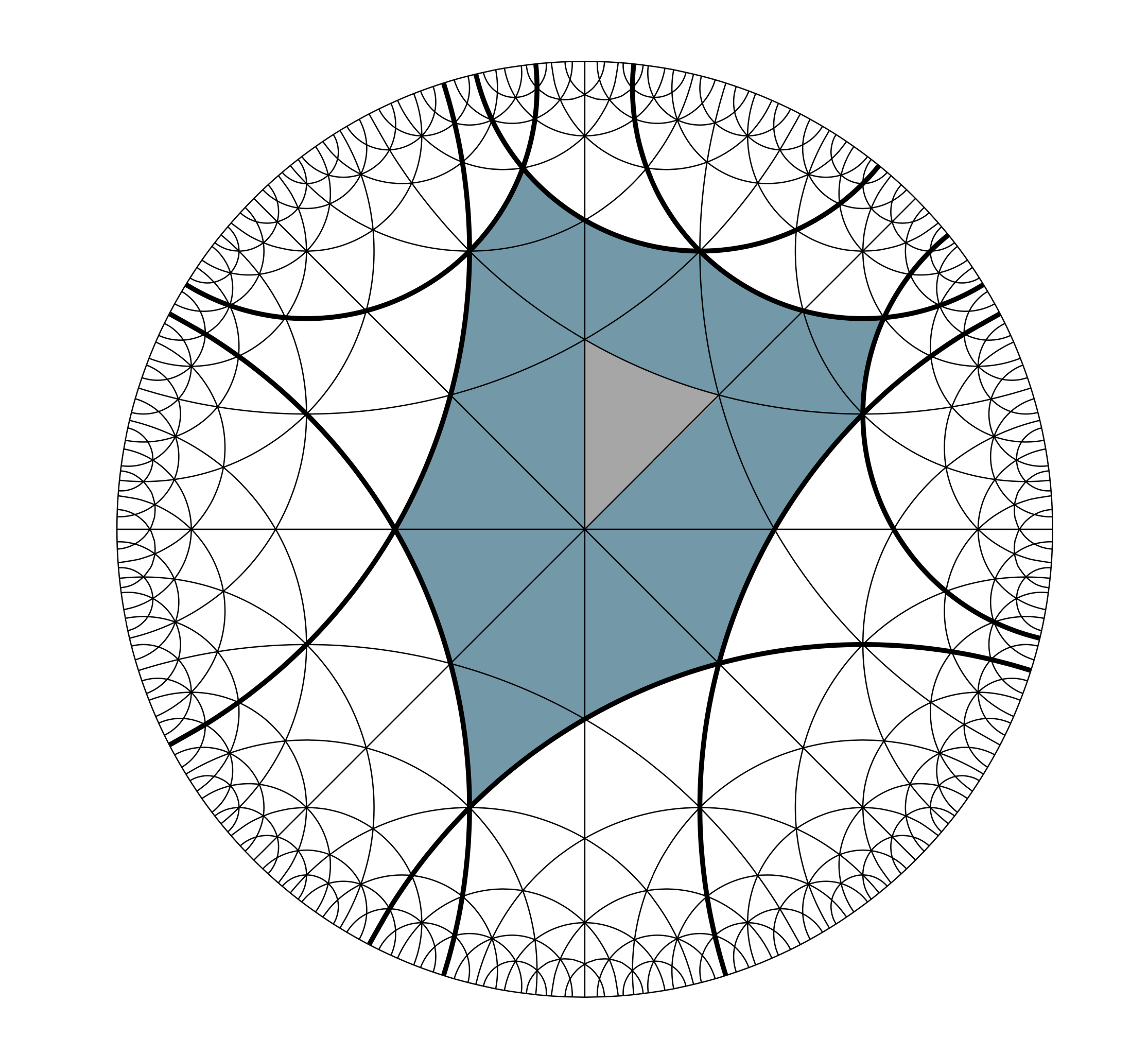}};

\node[scale=1] at (3.5,4) {$e$};
\node[scale=1]  at (3.8,5) {$1$};
\node[scale=1]  at (4.2,3.5) {$2$};
\node[scale=1] at (2.5,4) {$3$};
\node[scale=0.8] at (4.6,5) {$12$};
\node[scale=0.8] at (5.2,4.5) {$121$};
\node[scale=0.8] at (5,3.6) {$21$};
\node[scale=0.8] at (2.2,5) {$31$};
\node[scale=0.8] at (2.6,5.5) {$313$};
\node[scale=0.8] at (3.4,5.5) {$13$};
\node[scale=0.6,rotate=60,color=black!60] at (4.45,6.25) {$13232$};
\node[scale=0.8] at (2,3.3) {$32$};
\node[scale=0.8] at (2,2.5) {$323$};
\node[scale=0.8] at (2.5,1.75) {$3232$};
\node[scale=0.8] at (3.5,1.75) {$232$};
\node[scale=0.8] at (4,2.5) {$23$};
\node[scale=0.8,color=black!60] at (5.7,2.5) {$2131$};
\node[scale=0.8,color=black!60] at (0.8,4.5) {$3121$};
\node[scale=0.6]   at (2.25,0.75) {$32321$};
\node[scale=0.6] at (2.4,6.1) {$1312$};
\node[scale=0.6] at (5.7,4.8) {$1213$};


\end{tikzpicture}}}
\caption{The $0$-Shi arrangement and $0$-Shi polyhedron of the Coxeter system (a triangle group) with Coxeter graph given in the upper left.}
\label{fig:Hyp1}
\end{figure}

\medskip

Let us discuss a bit of history about the $m$-Shi arrangement.  In~\cite{Shi86,Shi88}, Shi introduced the  Shi arrangement $\shi(W,S)=\shi_0(W,S)$ in the case of irreducible affine Weyl groups to study Kazhdan-Lusztig cells for $W$.  Several surprising connections with Shi arrangements for affine Weyl groups have been studied, from {\em ad}-nilpotent ideals of Borel subalgebras~\cite{CePa02} to  Catalan arrangements~\cite{At04,ArRh11}; see also \cite{Fi19} and the references within. In~\cite{Shi88}, Shi  proves a conjecture by Carter on the number of sign-types of an affine Weyl group. In order to prove that conjecture, Shi enumerates the number of regions in $\shi_0(W,S)$, which turns out to amount to $(h+1\big)^n$, where $h$ is the Coxeter number of the underlying finite Weyl group $W_0$ of $(W,S)$ and $n$ is the rank of $W_0$.  In order to do so,  Shi proves that each region of the Shi arrangement contains a unique minimal element and that the union of the chambers corresponding to the inverses of those minimal elements is a convex subset of the Euclidean space. Theorems~\ref{thm:Main1}~and~\ref{thm:Main2} are a generalization of both results to arbitrary Coxeter systems.  Notice that in the case of affine Coxeter systems and for $m=0$, Theorem~\ref{thm:Main1} was proven by Chapelier-Laget and the second author in~\cite{ChHo22}, while for rank $3$ and $m=0$ it was proven by  Charles in~\cite{Chb20}. 

After the first version of this manuscript was published on the ArXiv, Przytycki and Yau informed us that they found, in the case $m=0$, a proof of Theorem~\ref{thm:Main1}(1) using the techniques Przytycki developed with Osajda to prove the biautomaticity of Coxeter groups~\cite{OsPr22}; they also prove in~\cite{PrYa23} that this set of minimal length elements for $\shi_0(W,S)$ form a Garside shadow (but not that it is equal to $L_0$).  

\smallskip
As far as we know, the {\em $m$-(extended) Shi arrangements} were defined for affine Coxeter systems by Athanasiadis~\cite[Theorem 4.3]{At00}. In that article, Athanasiadis proved for Coxeter systems of types $A$, $B$, $C$, $D$, $BC$ that the number of regions in $\shi_m(W,S)$ is  $\big((m+1)h+1\big)^n$ (the  {\em generalized Catalan numbers}). Extended Shi arrangements were further studied in~\cite{At04,At05,Yo04,Arm09}. In his thesis~\cite[Ch.~4]{Th15}, Thiel gives a direct proof by extending Shi's result to any $m$ in the case of an arbitrary affine Coxeter system; see also \cite[\S11]{Th16}. 

Theorem~\ref{thm:Main1} shows that Thiel's minimal elements for $\shi_m(W,S)$ are precisely the $m$-low elements. Moreover, our approach allow us to recover Thiel's results in \S\ref{se:Convex} as a direct consequence of the proof of Theorem~\ref{thm:Main2}.

\begin{thm}[Thiel~\cite{Th16}]\label{thm:Main3}  If $(W,S)$ is of affine type, then the union of the chambers $C_w$ for $w^{-1}\in L_m$ is a convex set. 
\end{thm}

This result is not true for an indefinite Coxeter system, i.e., neither affine nor finite; for a counterexample see Figure~\ref{fig:Hyp2}. It could be interesting to classify all indefinite Coxeter systems for which  the union of the chambers $C_w$ for $w^{-1}\in L_m$ is a convex set for any $m\in\mathbb N$.
\smallskip

The $m$-Shi arrangement in affine Coxeter systems also appears in the work of Gunnells~\cite{Gu10} in relation to the study of Kazdan-Lusztig cells. It would be interesting to investigate the $m$-Shi arrangement in indefinite type in relation to Belolipetsky and Gunnells' subsequent work~\cite{Gu15}, and in doing so, to go back to the root of Shi's original motivation to introduce Shi arrangements.

\subsection*{Acknowledgements} The authors are grateful to Christian Stump for pointing out to us the reference \cite{Th15}, and many thanks to Vic Reiner and Nathan Williams for useful discussion on $m$-Shi arrangement in affine types. We would also like to thank the two anonymous referees for their valuable comments that have helped improve the final version of this article. 

The authors also warmly thanks James Parkinson for providing the TikZ files from \cite{PaYa21} that are used for the affine and hyperbolic Coxeter complex  in Figures~\ref{fig:Aff1},~\ref{fig:Aff2} and~\ref{fig:Hyp1}. This article would not have been illustrated as it is without his contribution. We also wish to thank Franco Saliola for helping us with the Python code for Figure~\ref{fig:Hyp2}, which is derived from Python code obtained on Wikipedia: {\tt https://en.wikipedia.org/wiki/Triangle\_group}.

\section{Coxeter combinatorics of words and roots}\label{se:Combin}

Fix a Coxeter system $(W,S)$ with length function $\ell:W\to\mathbb N$; the {\em rank of $(W,S)$} is the cardinality of $S$. We assume the reader is familiar with the basics of the theory of Coxeter groups; see for instance \cite{Hu90,BjBr05,AbBr08}. 

\subsection{Combinatorics of reduced words}

We say that a word $s_1\cdots s_k$ ($s_i\in S$) is a {\em reduced word for $w\in W$} if $w=s_1\cdots s_k$ and
$k = \ell(w)$. For $u,v,w\in W$, we adopt the following terminology:
\begin{itemize}
\item {\em $u$ is a prefix of $w$} if a reduced word for $u$ can be obtained as a prefix of a reduced word for $w$;
\item {\em $v$ is a suffix of $w$} if a reduced word for $u$ can be obtained as a suffix of a reduced word for $w$;
\item $w=uv$ is a {\em reduced product}  if $\ell(w)=\ell(u)+\ell(v)$.
\end{itemize}
Observe that for a product $w=uv$, with $w,u,v\in W$, we have the following equivalences: $w=uv$ is a reduced product if and only if $u$ is a prefix of $w$, if and only if $v$ is a suffix of $w$.  In this case, the concatenation of a reduced word for $u$ with a reduced word for $v$ gives a reduced word for $w$.  More generally, we say that:
\begin{itemize}
\item $w=u_1\cdots u_k$ is  a {\em reduced product} if $\ell(w)=\ell(u_1)+\cdots +\ell(u_k)$,~$u_i,w\in W$.
\end{itemize}
The empty product (the identity $e$) and the product of one element ($k=1$ above) are said to be reduced products. Obviously, a reduced product $w=u_1\cdots u_k$ is a reduced word if and only if $u_i\in S$ for all $1\leq i\leq k$. 

\subsection{Weak and Bruhat orders} This suffix/prefix terminology is best embodied by the {\em weak order}.   The {\em right weak order} is the poset $(W,\leq_R)$ defined by $u\leq_R w$ if $u$ is a prefix of~$w$. The right weak order gives a natural orientation of the right Cayley graph of $(W,S)$: for $w\in W$ and $s\in S$, we orient the edge $w\to ws$  if $w\leq_R ws$. The example of the infinite dihedral Coxeter group is given in Figure~\ref{fig:InfiniteDi}. 

Similarly, the {\em left weak order} is the poset $(W,\leq_L)$ defined by $v\leq_L w$ if $v$ is a suffix of~$w$, which induces a natural orientation of the left Cayley graph of $(W,S)$.

Finally, recall that the {\em Bruhat order} is the poset $(W,\leq)$ defined as follows: $u\leq w$ if and only if a word for $u$ can be obtained as a subword of a reduced word for $w$. The example of finite dihedral groups is given in Figure~\ref{fig:BruhatDi}.

\subsection{Root system and reflections} 
Recall that a quadratic space $(V,B)$ is a real vector space $V$, together with a symmetric bilinear form $B$. The group $O_B(V)$ is the group consisting of all invertible linear maps that preserve $B$. For any non-isotropic vector $\alpha\in V$, i.e., $B(\alpha,\alpha)\not = 0$, we associate the $B$-reflection $s_\alpha$ given by the formula 
$$
s_\alpha(v)=v-2\frac{B(\alpha,v)}{B(\alpha,\alpha)}\alpha,
$$ 
for all $v\in V$. A {\em geometric representation of $(W,S)$} is a faithful representation of $W$ as a subgroup of  $O_B(V)$, where $S$ is mapped into the set of $B$-reflections associated to a {\em simple system} $\Delta=\{\alpha_s\,|\, s\in S\}$ ($s=s_{\alpha_s}$). In a geometric representation of $(W,S)$, the $W$-orbit $\Phi=W(\Delta)$ is a {\em root system} with {\em positive roots} $\Phi^+=\cone_\Phi(\Delta)$ and {\em negative roots $\Phi^-=-\Phi^+$}, where $\cone(X)$ is the set of nonnegative linear combination of vectors in $X\subseteq V$ and $\cone_\Phi(X)=\cone(X)\cap \Phi$; see  \cite[\S1]{HoLaRi14} for more details.  For any $w\in W$ and $\beta\in \Phi$, we have $ws_\beta w^{-1}=s_{w(\beta)}$, which induces a bijection between $\Phi^+$ and the  {\em set of reflections} in $W$:
$$
T=\{s_\beta\mid \beta\in \Phi^+\}=\bigcup_{w\in W}wSw^{-1}.
$$

\subsection{Depth of positive roots}  The {\em depth statistic on $\Phi^+$} is the function $\dep:\Phi^+\to \mathbb N$ defined by:
$$
\dep(\beta)=\min\{\ell(g)\mid g(\beta)\in \Delta\}.
$$ 
Observe that $\dep(\alpha_s)=0$ for all $s\in S$. It is well-known, see for instance \cite[Lemma 4.6.2]{BjBr05}, that for $\beta\in \Phi^+$ and $s\in S$ we have:
\begin{equation}\label{eq:Depth}
\dep(s(\beta))=
\left\{
\begin{array}{ll}
\dep(\beta)-1&\textrm{if }B(\alpha_s,\beta)>0\\
\dep(\beta)&\textrm{if }B(\alpha_s,\beta)=0\\
\dep(\beta)+1&\textrm{if }B(\alpha_s,\beta)<0
\end{array}
\right.
\end{equation}

\begin{remark} (1) The depth statistic was originally defined in \cite{BrHo93}. In~\cite[\S4.6]{BjBr05}, the depth of simple roots $\Delta$ has value $1$, but it is more natural for the arguments given in this article to set the depth of a simple root to $0$ at $\Delta$; for instance in Proposition~\ref{prop:Palind0} below it allows us to consider $g\in W$ with $\ell(g)=\dep(\beta)$ (instead of depth minus $1$). 

\noindent (2) The depth of a positive root $\beta$ may be seen as a distance measuring how far $\beta$ is from $\Delta$ in the orbit $\Phi=W(\Delta)$. There are many  different depth statistics and they are not equivalent. In this article we also consider the $\infty$-depth. For more information on depth statistics, lengths and weak orders on root systems in general, see \cite[\S5.1]{DyHo16}
\end{remark}

\subsection{Inversion sets} The {\em  inversion set~$\Phi(w)$ of $w\in W$} is defined\footnote{The inversion set of $w\in W$ is sometimes denoted in the literature by $N(w)$ or inv$(w)$.} by:
$$
\Phi(w)=\Phi^+\cap w(\Phi^-)=\{\beta\in \Phi^{+}\, | \, \ell(s_{\beta}w)<\ell(w)\}. 
$$
Its cardinality is well-known to be $\ell(w)$.  Inversion sets allow a useful geometric interpretation of the right weak order; see  Figure~\ref{fig:InfiniteDi} for an illustration. 

\begin{prop}\label{prop:InvBasics} Let $w,u,v\in W$.
\begin{enumerate} 
\item If $w=s_1\cdots s_k$ is a reduced word for $w\in W$, then
$$
\Phi(w)=\{\alpha_{s_1},s_1(\alpha_{s_2}),\dots , s_1\cdots s_{k-1}(\alpha_{s_k})\}.
$$

\item If $w=uv$ is a reduced product, then $\Phi(w)=\Phi(u)\sqcup u(\Phi(v))$. 
\item The map $w\mapsto \Phi(w)$, from $(W,\leq_R)$ to ${{\{\Phi(w)\mid w\in W\}}}$ ordered by inclusion, is an order isomorphism.

\end{enumerate}
\end{prop}
For a proof of this classical proposition see, for instance,~\cite[Proposition 2.1]{HoLa16}. 

\begin{figure}[h!]
\resizebox{.75\hsize}{!}{
\begin{tikzpicture}
	[scale=1,
	 pointille/.style={dashed},
	 axe/.style={color=black, very thick},
	 sommet/.style={inner sep=2pt,circle,draw=blue!75!black,fill=blue!40,thick,anchor=west}]
	 
\node[sommet]  (id)    [label=below:{\small{$e;\  \emptyset$}}]          at (0,0)    {};
\node[sommet]  (1)    [label=left:{\small{$\{\alpha_s\};\ s$}}]         at (-1,1)   {} edge[thick,<-] (id);
\node[sommet]  (2)    [label=right:{\small{$t;\ \{\alpha_t\}$}}]        at (1,1)    {} edge[thick,<-] (id);
\node[sommet]  (12)   [label=left:{\small{$\{\alpha_s, s(\alpha_t)\};\ st$}}]      at (-1,2)   {} edge[thick,<-] (1);
\node[sommet]  (21)   [label=right:{\small{$ts;\ \{\alpha_t, t(\alpha_s)\}$}}]     at (1,2)    {} edge[thick,<-] (2);
\node[sommet]  (121)  [label=left:{\small{$\{\alpha_s\}\sqcup s(\Phi(ts));\ sts$}}]    at (-1,3)    {} edge[thick,<-] (12);
\node[sommet]  (212)  [label=right:{\small{$tst;\ \{\alpha_t\}\sqcup t(\Phi(st))$}}]    at (1,3)    {} edge[thick,<-] (21);

\draw[pointille] (121) -- +(0,1);
\draw[pointille] (212) -- +(0,1);
\end{tikzpicture}}
\caption{\cite[Figure~1]{DyHo16} The right weak order on  the infinite dihedral group $\mathcal D_\infty$: the vertices are labelled by $w\in \mathcal D_\infty$  and its corresponding inversion set $\Phi(w)$.}
\label{fig:InfiniteDi}
\end{figure}
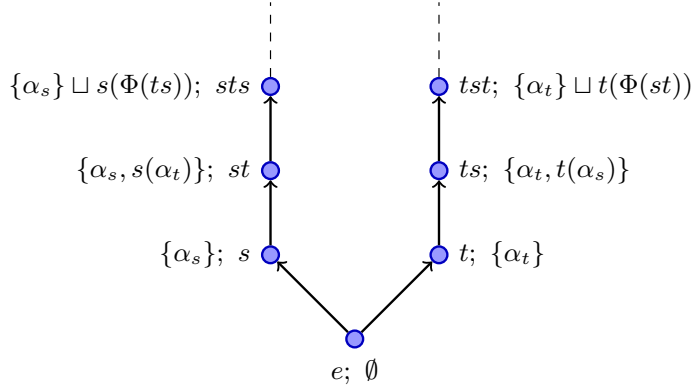
%

\subsection{Reflections and the depth of positive roots} We recall two results that provide palindromic reduced words for a reflection in $T$ in relation with the depth of the associated positive root. The following first proposition goes back to Springer~\cite[Proposition~1]{Sp82}, see also~\cite[Lemma 2.7]{Dy90}; we provide a proof for convenience.

\begin{prop}\label{prop:Palind-00} Let $t\in T$ and chose a reduced word $t=s_1\cdots s_m$. Then there is an integer $k$ such that $\ell(t)=m=2k+1$. Furthermore,  $t=s_1\cdots s_{k}s_{k+1}s_k\cdots s_1$ is a (palindromic) reduced word for $t$.
\end{prop}
\begin{proof} Let $\beta\in \Phi^+$ such that $t=s_\beta$. Since $\beta =w(\alpha_s)$ for some $w\in W$ and $s\in S$, we have $t=s_\beta=wsw^{-1}$. Therefore $\epsilon(t)=-1$, where $\epsilon:W\to \{\pm 1\}$ is the signature of $(W,S)$. Therefore $\ell(t)$ is odd since $\epsilon(t)=(-1)^{\ell(t)}$. Then there is $0\leq k<m$ such that $m=\ell(t)=2k+1$. 

First, set $x=s_1\cdots s_{k+1} \leq_R t$. Since the word $s_1\cdots s_{k+1}$  is reduced  (as a prefix of the given reduced word for $t$), we have:
\begin{eqnarray*}
\ell(tx)=\ell(t^{-1}x)&=&\ell(s_m\cdots s_1 s_1\cdots s_{k+1})=\ell(s_m\cdots s_{k+2})\\
&\leq &m-1-k = 2k+1-1-k=k \\
&<&k+1=\ell(x).
\end{eqnarray*}
Therefore $\beta\in \Phi(x)$. Now,  by Proposition~\ref{prop:InvBasics}~(1), there is $1\leq i\leq k+1$ such that $\beta = s_1 \cdots s_{i-1} (\alpha_{s_i})$. Therefore 
$
t=s_\beta = s_1 \cdots s_{i-1}s_is_{i-1}\cdots s_1,
$
implying $2k+1=\ell(t)\leq 2(i-1)+1\leq 2k+1$. So $i=k+1$ and $t=s_\beta = s_1 \cdots s_{k}s_{k+1}s_{k}\cdots s_1$ is a reduced word for $t=s_\beta$.
\end{proof}

The next consequence relates the length of a reflection with the depth of its associated positive root. 

\begin{prop}\label{prop:Palind0} Let $\beta\in \Phi^+$ and $g\in W$ such that $g(\beta)=\alpha_r\in \Delta$ with $r\in S$ and $\ell(g)=\dep(\beta)$.  Then $\ell(s_\beta)=2\dep(\beta)+1$ and the product $s_\beta=g^{-1}rg$ is reduced. 
\end{prop}
\begin{proof} On  the one hand, choose a reduced word $s_\beta = s_1\cdots s_k s_{k+1} s_k\cdots s_1$,  by Proposition~\ref{prop:Palind-00}. Therefore the prefix   $s_1\cdots s_k s_{k+1}$ is also a reduced word, hence $s_1\cdots s_k(\alpha_{s_{k+1}})\in \Phi^+$.  Since $s_\beta = s_{s_1\cdots s_k(\alpha_{s_{k+1}})}$, we have $\beta = s_1\cdots s_k(\alpha_{s_{k+1}})$ and $\dep(\beta)\leq k$. On the other hand, by assumption, we have $s_{\beta}=s_{g^{-1}(\alpha_r)}=g^{-1}rg$. So:
$$
2k+1=\ell(s_\beta)\leq 2\ell(g)+1=2\dep(\beta)+1\leq 2k+1
$$ 
Therefore we have $\ell(s_\beta)=2\dep(\beta)+1$ and the product $s_\beta=g^{-1}rg$ is reduced (and $\ell(g)=\dep(\beta)=k$).
\end{proof}

\begin{cor}\label{cor:Palind} Let $\beta\in \Phi^+$ and $g\in W$ such that $g(\beta)=\alpha_r\in \Delta$ with $r\in S$ and $\ell(g)=\dep(\beta)$.  If $b\in W$ and $s\in S$ are  such that $b<_L sb\leq_Lg$, then $\alpha_s \in \Phi(s_{b(\beta)})$ and $\ell(ss_{b(\beta)}s)=\ell(s_{b(\beta)})-2$. 
\end{cor}
\begin{proof} Since $sb\leq_L g$ there is a prefix $a\in W$ of $g$ such that $g=asb$ is a reduced product. Since $g(\beta)=\alpha_r$ we have $b(\beta)=sa^{-1}(\alpha_r)$. Therefore
$$
s_{b(\beta)}= sa^{-1}ras.
$$
Since $b^{-1}sa^{-1}rasb=g^{-1}rg$ and $ g^{-1}rg$ is a reduced product by Proposition~\ref{prop:Palind0}, the product $sa^{-1}ras$ is also reduced. It follows immediately that $\alpha_s \in \Phi(s_{b(\beta)})$ and $\ell(ss_{b(\beta)}s)=\ell(s_{b(\beta)})-2$.
\end{proof}

\subsection{Reflections and the Bruhat order}  Recall that the Bruhat order $(W,\leq)$
arises also as the reflexive, transitive closure of the relation $\to$ on $W$ 
defined by $x\to y$ if there is $t \in T$ such that $y=t x$ with $\ell(x)<\ell(y)$. We say that a pair $x< y$ is a {\em covering in the Bruhat order}, which we denote by $x\lhd y$,  if for any $z\in W$ such that $x\leq z\leq y$ then we have $z=x$ or $z=y$. The chain property of Bruhat order implies that $x\lhd y$  is a {\em covering}  if $x<y$ and  $\ell(y)=\ell(x)+1$, i.e., if there is $t \in T$ such that $y=t x$ and  $\ell(y)=\ell(x)+1$. It is known that $u\leq_R v$ or $u\leq_L v$ implies that $u\leq v$.   We refer the reader to \cite[Chapter~2]{BjBr05} for more details.

\begin{ex}\label{ex:BruhatDi}  Let $W$ be a dihedral group $\mathcal D_m$ ($m\in\mathbb N_{\geq 2}\cup\{\infty\})$ generated by $S=\{s,t\}$ with Coxeter graph:
\begin{center}
\begin{tikzpicture}[sommet/.style={inner sep=2pt,circle,draw=blue!75!black,fill=blue!40,thick}]
	\node[sommet,label=above:$s$] (alpha) at (0,0) {};
	\node[sommet,label=above:$t$] (beta) at (1,0) {} edge[thick] node[auto,swap] {$m$} (alpha);
\end{tikzpicture}
\end{center}
Denote by $[s,t]_k$ the reduced word $st\cdots$ with $k$ letters; this word ends with a $s$ if $k$ is odd and with a $t$ if $k$ is even. Denote for $k\in\mathbb N$:
$$
\alpha_{s,k} =\left\{\begin{array}{ll}%
 [s,t]_k (\alpha_s)&\textrm{if $k$ is even}.\\%
 {}[s,t]_k (\alpha_t)&\textrm{if $k$ is odd}.%
\end{array}
\right.
$$ 
In particular $\alpha_{s,0}=\alpha_s$. By symmetry we can define $\alpha_{t,k}$ using the words $[t,s]_k$. We illustrate in Figure~\ref{fig:BruhatDi} the Hasse diagrams for the Bruhat order on dihedral groups with the edges representing the covering $x\lhd y$ labelled by the positive root $\beta$ such that $y=s_\beta x$.
\begin{figure}[h!]
\resizebox{.6\hsize}{!}{
\begin{tikzpicture}
	[scale=1,
	 pointille/.style={dashed},
	 axe/.style={color=black, very thick},
	 sommet/.style={inner sep=2pt,circle,draw=blue!75!black,fill=blue!40,thick,anchor=west}]
	 
\node[sommet]  (id)    [label=below:{\small{$e$}}]          at (0,0)    {};
\node[sommet]  (1)    [label=left:{\small{$s$}}]         at (-1,1)   {} edge[thick, red] (id) ;
\node[sommet]  (2)    [label=right:{\small{$t$}}]        at (1,1)    {} edge[thick,blue] (id);
\node[sommet]  (12)   [label=left:{\small{$st$}}]      at (-1,2)   {} edge[thick] (1) {} edge[thick,red] (2);
\node[sommet]  (21)   [label=right:{\small{$ts$}}]     at (1,2)    {} edge[thick] (2) {} edge[thick,blue] (1);
\node[sommet]  (121)  [label=left:{\small{$sts$}}]    at (-1,3)    {} edge[thick] (12) {} edge[thick,red] (21);
\node[sommet]  (212)  [label=right:{\small{$tst$}}]    at (1,3)    {} edge[thick] (21){} edge[thick,blue] (12) ;

\draw[pointille] (121) -- +(0,0.6);
\draw[pointille] (212) -- +(0,0.6);
\draw[pointille,blue] (121) -- +(1.2,0.6);
\draw[pointille,red] (212) -- +(-1.2,0.6);

\node[sommet]  (12k)  [label=left:{\small{$[s,t]_{k-1}$}}]    at (-1,5)    {};
\node[sommet]  (21k)  [label=right:{\small{$[t,s]_{k-1}$}}]    at (1,5)    {};
\draw[pointille] (12k) -- +(0,-0.6);
\draw[pointille] (21k) -- +(0,-0.6);
\draw[pointille,red] (12k) -- +(1.2,-0.6);
\draw[pointille,blue] (21k) -- +(-1.2,-0.6);
\node[sommet]  (12kk)  [label=left:{\small{$[s,t]_{k}$}}]    at (-1,6)    {} edge[thick] (12k) {} edge[thick,red] (21k);
\node[sommet]  (21kk)  [label=right:{\small{$[t,s]_{k}$}}]    at (1,6)    {} edge[thick,blue] (12k) {} edge[thick] (21k);
\node[sommet]  (12kkk)  [label=left:{\small{$[s,t]_{k+1}$}}]    at (-1,7)    {} edge[thick] (12kk) {} edge[thick,red] (21kk);
\node[sommet]  (21kkk)  [label=right:{\small{$[t,s]_{k+1}$}}]    at (1,7)    {} edge[thick,blue] (12kk) {} edge[thick] (21kk);

\draw[pointille] (12kkk) -- +(0,0.6);
\draw[pointille] (21kkk) -- +(0,0.6);
\draw[pointille,blue] (12kkk) -- +(1.2,0.6);
\draw[pointille,red] (21kkk) -- +(-1.2,0.6);

\node[sommet]  (12f)  [label=left:{\small{$[s,t]_{m-1}$}}]    at (-1,9)    {};
\node[sommet]  (21f)  [label=right:{\small{$[t,s]_{m-1}$}}]    at (1,9)    {};
\draw[pointille] (12f) -- +(0,-0.6);
\draw[pointille] (21f) -- +(0,-0.6);
\draw[pointille,red] (12f) -- +(1.2,-0.6);
\draw[pointille,blue] (21f) -- +(-1.2,-0.6);
\node[sommet]  (wo)    [label=above:{\small{$w_\circ=[s,t]_m$}}]          at (0,10)     {} edge[thick] (12f)  {} edge[thick] (21f) ;

\draw (-0.4,0.3) node[auto,swap,red] {\tiny{$\alpha_s$}};
\draw (0.7,0.3) node[auto,swap,blue] {{\tiny $\alpha_t$}};

\draw (-1.2,1.5) node[auto,swap] {\tiny{$\alpha_{s,1}$}};
\draw (-1.2,2.5) node[auto,swap] {\tiny{$\alpha_{s,2}$}};
\draw (-1.35,5.5) node[auto,swap] {\tiny{$\alpha_{s,k-1}$}};
\draw (-1.2,6.5) node[auto,swap] {\tiny{$\alpha_{s,k}$}};
\draw (-1,9.5) node[auto,swap] {\tiny{$\alpha_{s,m-1}$}};

\draw (1.5,1.5) node[auto,swap] {\tiny{$\alpha_{t,1}$}};
\draw (1.5,2.5) node[auto,swap] {\tiny{$\alpha_{t,2}$}};
\draw (1.65,5.5) node[auto,swap] {\tiny{$\alpha_{t,k-1}$}};
\draw (1.5,6.5) node[auto,swap] {\tiny{$\alpha_{t,k}$}};
\draw (1.2,9.5) node[auto,swap] {\tiny{$\alpha_{t,m-1}$}};

\end{tikzpicture}}
\caption{\cite[Figure~6]{DyHo16} The Hasse diagram of the Bruhat order on  the finite dihedral group $\mathcal D_m$. For $\mathcal D_\infty$, the Hasse diagram is the diagram with infinite vertices obtained by not considering the top part of the above diagram. The labels for the interior edges are as follows: the `parallel' red  ones corresponds to the label $\alpha_s$ and the `parallel' blue ones to the label  $\alpha_t$.  }
\label{fig:BruhatDi}
\end{figure}
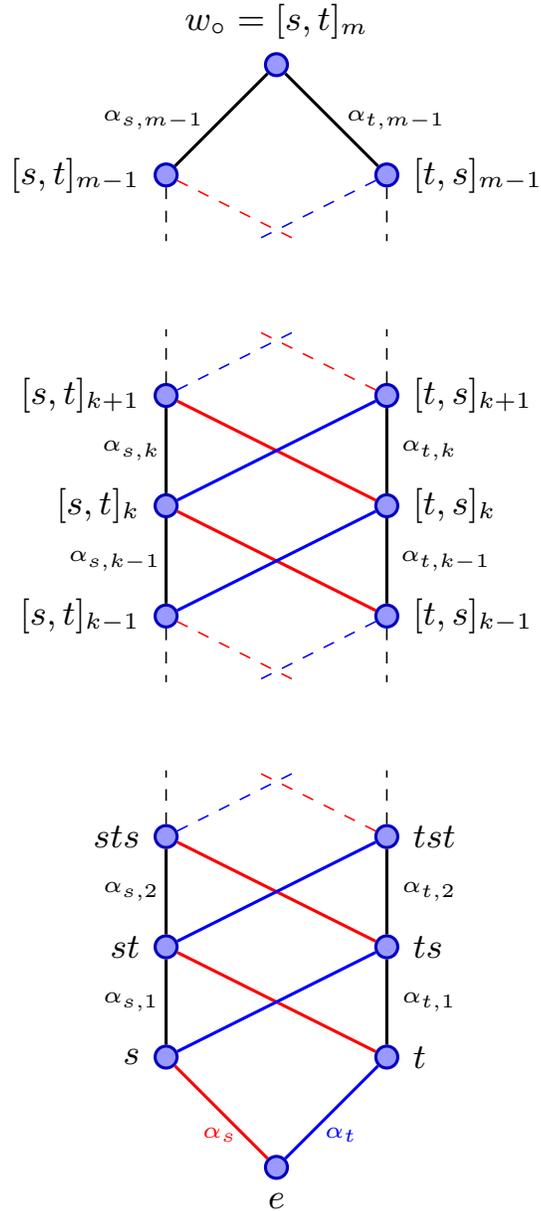
\end{ex}

\subsection{Reflection subgroups and the Bruhat order} We end this section by recalling some useful facts about reflection subgroups and, in particular, about maximal dihedral reflection subgroups. We follow the exposition given in \cite[\S2.8]{DyHo16}

A \emph{reflection subgroup} $W'$ of $W$ is a subgroup $W'=\langle s_\beta\,|\, \beta\in A\rangle$ generated by the reflections associated to the roots in some  $A\subseteq \Phi^+$. Write 
\begin{equation}\label{eq:DRS}
\Phi_{W'}:=\{\beta\in \Phi\,|\, s_{\beta}\in W'\} \ \textrm{ and }\ \Delta_{W'}:=\{\alpha\in \Phi^{+}  \,|\, \Phi(s_\alpha)\cap \Phi_{W'}=\{\alpha\}\}.
\end{equation}
The first author shows in~\cite{Dy90} that  $\Phi_{W'}$ is a  root system in $(V,B)$ with simple system $\Delta_{W'}$, called  the  {\em canonical simple system of $\Phi_{W'}$}.  Therefore the pair  $(W',\chi(W'))$ is a Coxeter system, with {\em  canonical simple  reflections}  
$$
\chi(W'):=\{s_{\alpha}\,|\, \alpha\in \Delta_{W'}\},
$$
and with corresponding positive roots  $\Phi_{W'}^{+}=\Phi_{W'}\cap \Phi^{+}  $;  see also~\cite{Dy10} (both notions depend on $(W,S)$ and not just $W$).   In particular:
\begin{equation}\label{eq:BasisAndDihedral}
B(\alpha,\beta)\leq 0,\ \forall \alpha\not = \beta \in \Delta_{W'} .
\end{equation}
Observe also that  $S\cap W'\subseteq \chi(W')$.

\smallskip

\subsubsection*{Coset representatives} We are now interested in decomposing elements of $W$ with respect to the cosets $W'w$ for $w\in W$. The set of \emph{minimal length right coset representatives} for $W'$ in $W$ is:
\begin{equation*}
X_{W'}:=\{v\in W\,|\, \ell(s_\beta v)>\ell(v),\forall \beta\in \Phi_{W'}  \}=\{v\in W\,|\, \Phi(v)\cap \Phi_{W'}=\emptyset \}.
\end{equation*}  
Any coset $W'w$ in $W$ contains an element of minimal length, which is in~$X_{W'}$. For an element $u\in W'$,  the inversion set $\Phi_{W'}(u)$ of $u$ in $\Phi_{W'}$ satisfies $\Phi_{W'}(u)=\Phi(u)\cap \Phi_{W'}$; the following result is known (see \cite[3.3(ii),3.4]{Dy91}).

\begin{prop} \label{prop:CosetRep} Let $W'$ be a reflection subgroup of $W$ and $w\in W$. Write $w=uv$ with $(u,v)\in  W'\times X_{W'}$. We have:
\begin{enumerate}
\item  $\Phi(w)\cap \Phi_{W'}=\Phi_{W'}(u)$;
\item $\ell(s_\beta w)<\ell(w)$ if and only if $\ell_{W'}(s_\beta u)<\ell_{W'}(u)$ for all $\beta\in \Phi_{W'}$.
\end{enumerate}
\end{prop}

The second statement of the above proposition is referred to  as   {\em ``functoriality of the Bruhat graph''}  (for inclusions of reflection subgroups; see \cite{Dy91,Dy92}). This means that there is an isomorphism of edge-labelled graph between the Bruhat graph of $W'$ and the Bruhat graph of $W$ restricted to $W'w$, for any $w\in W$; see~\cite[Theorem~1.4]{Dy91}. In particular, if $\leq_{W'}$ denotes the Bruhat order on $W'$, we have:

\begin{equation}\label{eq:BruhatG}
u_1\leq_{W'} u_2 \implies u_1v\leq u_2v,\quad \textrm{for all  }u_i \in W',\ v\in  X_{W'}. 
\end{equation}
The converse of this statement is not true in general; a counterexample can be seen in type $I_2(4)$ with simple reflections $s,t$ and $W'=\{s,tst\}$.

This functoriality property implies also that the elements of $X_{W'}$ are precisely  those elements $v$ of $W$ which are of minimum length in their coset $W'v$, and that any element $w\in W$ can be uniquely expressed in the form $w=uv$ where $u\in W'$ and $v\in X_{W'}$; note that this product is not reduced in general (it is a reduced product whenever  $W'$ is a standard parabolic subgroup).  Functoriality of the Bruhat graph also implies the following useful alternative characterization of~$X_{W'}$:   
$$X_{W'}=\{v\in W\,|\, \ell(s_\beta v)>\ell(v),\forall \beta\in \Delta_{W'} \}.$$

\subsection{Maximal dihedral reflection subgroups}  A reflection subgroup $W'$ of rank~$2$ is well-known to be isomorphic to a dihedral group and is so called a  {\em dihedral reflection subgroup}. This following result gives a criterion for comparing depths of roots.

\begin{prop}\label{prop:DiDep} Let $\alpha,\beta\in \Phi^+$. Assume there is a dihedral reflection subgroup~$W'$ such that such that $\alpha\in \Delta_{W'}$ and $\beta\in \Phi^+_{W'}\setminus\Delta_{W'}$, then $\dep(\alpha)<\dep(\beta)$.
\end{prop}
\begin{proof} Since $\alpha\in \Delta_{W'}$, $\beta\in \Phi^+_{W'}\setminus\Delta_{W'}$ and $(W',\chi(W'))$ is a Coxeter system of rank~$2$, it is known that $s_\alpha<_{W'}s_\beta$ in the Bruhat order on $W'$. By Eq.(\ref{eq:BruhatG}) (functoriality of Bruhat graphs), we have $s_\alpha< s_\beta$, since $s_\alpha,s_\beta\in W'$ and $e\in X_{W'}$. In particular, $\ell(s_\alpha)<\ell(s_\beta)$.  By Proposition~\ref{prop:Palind0}, we conclude that $\dep(\alpha)<\dep(\beta)$. 
\end{proof}

A dihedral reflection subgroup $W'$ is a {\em maximal dihedral reflection subgroup} if it is not contained in any other dihedral reflection subgroup but itself. This is equivalent to the following condition on the root subsystem $\Phi_{W'}$:
$$
\Phi_{W'}=\mathbb R\Delta_{W'}\cap \Phi.
$$ 
Recall that for $X\subseteq V$, $\cone_\Phi(X)=\cone(X)\cap \Phi$. The following useful result gives the form of inversion sets in maximal dihedral reflection subgroup.

\begin{prop}\label{prop:InvMDRG} Let $W'$ be a maximal dihedral reflection subgroup. The inversion set of $u\in W'$, $u\not = e$,  is of the form
$\Phi_{W'}(u)=\cone_{\Phi}(\alpha,\beta)$ with $\alpha\in \Delta_{W'}$ and $\beta\in \Phi^+_{W'}$. 
\end{prop}
\begin{proof} It is well-known that inversion sets in dihedral groups are of the form: for $u\in W'$, $u\not = e$,  $\Phi_{W'}(u)=\cone_{\Phi_{W'}}(\alpha,\beta)$, with $\alpha\in \Delta_{W'}$ and $\beta\in \Phi^+_{W'}$. By definition of maximal dihedral reflection subgroups we have:
\begin{eqnarray*}
\Phi_{W'}(u)=\cone_{\Phi_{W'}}(\alpha,\beta)&=&\cone(\alpha,\beta)\cap \Phi_{W'} \\
&=&\cone(\alpha,\beta)\cap \mathbb R\Delta_{W'}\cap \Phi\\
& =& \cone(\alpha,\beta)\cap  \Phi ,\ \textrm{ since $\cone(\alpha,\beta)\subseteq \mathbb R\Delta_{W'}$}\\
&=&\cone_{\Phi}(\alpha,\beta). \qedhere
\end{eqnarray*}
\end{proof}

Any dihedral reflection subgroup is contained in a unique maximal dihedral reflection subgroup. 

\begin{defi} Let $\alpha,\beta\in \Phi$ such that $\mathbb R\alpha\not = \mathbb R\beta$. We denote by $\mathcal M_{\alpha,\beta}$  the unique maximal dihedral reflection subgroup  containing the reflections $s_\alpha$ and $s_\beta$.  For simplicity, if $s=s_\alpha\in T$ and $t=s_\beta\in T$, we write $\mathcal M_{s,t}=\mathcal M_{\alpha,\beta}$.
\end{defi}

\begin{remark} 
\begin{enumerate}
\item  We have $\Phi_{\mathcal M_{\alpha,\beta}}=(\mathbb R\alpha\oplus\mathbb R\beta)\cap \Phi$.
\item The finite maximal dihedral reflection subgroups of $(W,S)$ are precisely the finite parabolic subgroups of rank $2$, that is, the conjugates of the standard parabolic subgroups $W_{s,t}=\langle s,t\rangle$ for $s,t\in S$ distinct such that the order $m_{s,t}$ of $st$ is finite.  Conversely, any conjugate of a rank $2$ finite parabolic subgroup is maximal~\cite[Theorem~3.11(b)]{Dy21}.
\end{enumerate}
\end{remark}

 The next lemma will be useful in \S\ref{se:ShortInv}.

\begin{lem}\label{lem:BruhatKey} Let $z\in W$, $s\in S$ and $t\in T$ such that $z< sz<tsz<sts z$. Assume that $\ell(tsz)=\ell(sz)+1$. Then $\chi(\mathcal M_{s,t}) = \{s,t\}$ and  $\ell(sts)\not = \ell(t)-2$.
\end{lem}
\begin{proof} Write $W'=\mathcal M_{s,t}$ and $z=z' z''$ with $z'\in W'$ and $z''\in X_{W'}$. Since $S\cap W'\subseteq\chi(W')$, we have $s\in \chi(W')$.  Write $\chi(W')=\{s,r\}$. One just has to show that $r=t$. By assumption, we have $z\lhd sz\lhd tsz\lhd sts z$, since $s\in S$ and $\ell(tsz)=\ell(sz)+1$. By functoriality of the Bruhat graph (Proposition~\ref{prop:CosetRep}), we have therefore $z'\lhd sz'\lhd tsz'\lhd sts z'$ in the Coxeter system $(W',\chi(W'))$.   In particular, $\ell_{W'}(tsz')=\ell_{W'}(z')+2$. We use the notation of Example~\ref{ex:BruhatDi}: since $z<sz'$ and $tsz'<stsz'$, we must have $z'=[r,s]_k$ and $tsz'=[r,s]_{k+2}$ for some $k\in\mathbb N$. So 
$$
rs[r,s]_k=[r,s]_{k+2} = tsz=ts[r,s]_k,
$$
implying $r=t$.
\end{proof}

\section{Short inversion posets}\label{se:ShortInv}

As Proposition~\ref{prop:InvBasics} illustrates, inversion sets are useful tools to study elements of~$W$. Among all inversions of an element of $W$, the short inversions span all the others. The key  to prove Theorem~\ref{thm:Main} is to exhibit an order on the short inversions and to show that any short inversion is {\em sandwiched} between a {\em left descent-root} and a  {\em right descent-root}.

\subsection{Short inversions}  We think of $\Phi(w)$ as a polyhedral cone in $\Phi\subseteq V$ since  $\Phi(w)=\cone_\Phi(\Phi(w))$.  The {\em set of short inversions of $\Phi(w)$} is the set
$$
\Phi^1(w)=\{\beta\in \Phi^+ \mid \ell(s_\beta w )=\ell(w)-1\}=\{\beta\in \Phi^+ \mid s_\beta w \lhd w\}\subseteq \Phi(w).
$$
The first author showed in \cite[Proposition 1.4 (c)]{Dy94-1} that $\Phi^1(w)$ is a basis of $\cone(\Phi(w))$:  the set of extreme rays of $\cone_\Phi(\Phi(w))$ is indeed $\{\mathbb R_{\geq 0}\beta\mid \beta \in \Phi^1(w)\}$. 

\begin{ex}\label{ex:H3} Consider $(W,S)$ of type $H_3$. The set of simple reflections is $S=\{1,2,3\}$ as in Figure~\ref{fig:H3}. The classical geometric representation of $(W,S)$ is on a Euclidean linear space $(V,B)$ of dimension $3$. 
In Figure~\ref{fig:H3}, we draw the set of positive roots $\Phi^+$ in the projective space $\mathbb PV$. The polyhedral cone $\cone(\Delta)$ is then represented as a polytope in $\mathbb PV$, here the triangle with vertices $\alpha_1,\alpha_2,\alpha_3$. The advantage of this projective view of root systems is to represent the inversion set $\Phi(w)$ of $w\in W$ as a polytope - an {\em inversion polytope} - with set of vertices $\Phi^1(w)$; for more information on projective root systems, see for instance~\cite{HoLaRi14,HoLa16,DyHoRi16,DyHo16}.

Consider the element $w=312121321$ of length $9$.  We use Proposition~\ref{prop:InvBasics} to compute $\Phi(w)$.  First, observe that the element $12121$ is the longest element of the standard (dihedral) parabolic subgroup $W_{\{1,2\}}$, so $\Phi(12121)=\Phi^+_{\alpha_1,\alpha_2}$. Therefore:
\begin{eqnarray*}
\Phi(w)&=&\Phi(312121)\sqcup 312121(\Phi(321))\\
&=& \{\alpha_3\}\sqcup 3(\Phi(12121))\sqcup 312121(\Phi(321))\\
&=&\{\alpha_3\}\sqcup 3(\Phi^+_{\alpha_1,\alpha_2})\sqcup 312121(\{\alpha_3,\alpha_2+\alpha_3,32(\alpha_1) \})\\
&=&\{\alpha_3\}\sqcup\Phi^+_{\alpha_1,\alpha_2+\alpha_3}\sqcup\{\gamma,\nu,\beta\},
\end{eqnarray*}
where $\gamma=312121(\alpha_3)$, $\nu=312121(\alpha_2+\alpha_3)$ and $\beta=31212132(\alpha_1)$. The set of short inversions of $w$ is 
$$
\Phi^1(w)= \{\alpha_1,\alpha_3, \alpha_2+\alpha_3,\beta\},
$$
and $\Phi(w)=\cone_\Phi(\Phi^1(w))$; see Figure~\ref{fig:H3}. Observe in Figure~\ref{fig:H3} that for any two distinct short inversions $\mu_1,\mu_2$ in $\Phi^1(w)$, at least one of them is an endpoint of the segment $(\mathbb R\mu_1\oplus \mathbb R\mu_2)\cap \Phi$. In other words, at least one of them is a simple root in the maximal dihedral reflection subgroup $\mathcal M_{\mu_1,\mu_2}$. For instance $\alpha_2+\alpha_3$ is a simple root in $\mathcal M_{\alpha_2+\alpha_3,\beta}$, but $\beta$ is not since $\Delta_{\mathcal M_{\alpha_2+\alpha_3,\beta}}=\{1(\alpha_2),\alpha_2+\alpha_3\}$.

\begin{figure}[h!]
\resizebox{\hsize}{!}{
\begin{tikzpicture}
	[scale=2,
	 q/.style={teal,line join=round},
	 racine/.style={blue},
	 racinesimple/.style={blue},
	 racinedih/.style={blue},
	 sommet/.style={inner sep=2pt,circle,draw=black,fill=blue,thick,anchor=base},
	 rotate=0]
 \tikzstyle{every node}=[font=\small]
\def\grosseursimple{0.025}
\node[anchor=south west,inner sep=0pt] at (0,0) {\includegraphics[width=8cm]{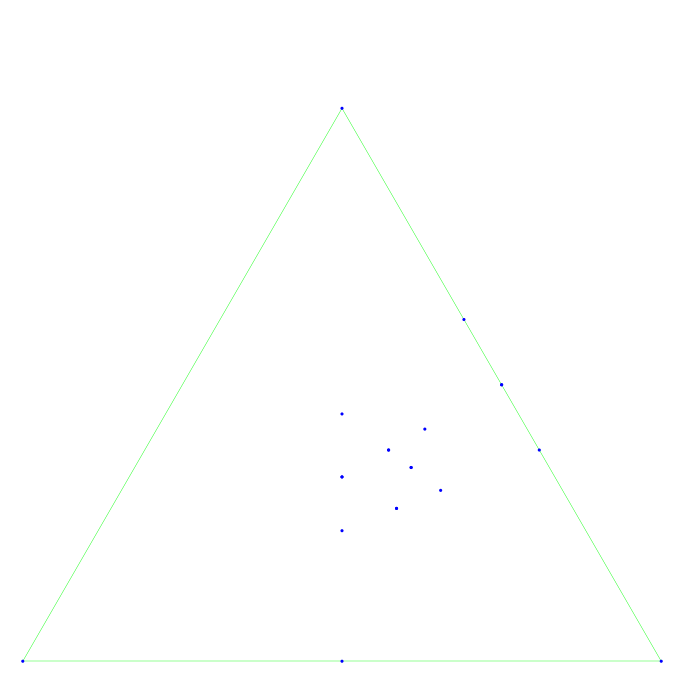}};
\coordinate (ancre) at (0,3);

\node[sommet,label=above:$1$] (gamma) at ($(ancre)+(0.25,0.43)$) {};
\node[sommet,label=below right :$2$] (beta) at ($(ancre)+(0.5,0)$) {} edge[thick] node[auto,swap] {$5$} (gamma) ;
\node[sommet,label=below left:$3$] (alpha) at (ancre) {}
 edge[thick] (beta)  ;

\draw[q,fill = teal, opacity=0.4] (0.13,0.13) -- (2,0.13) -- (2.41,1.27) -- (2,3.37) -- cycle;
\node[label=left :{$\alpha_3$}] (a) at (0.13,0.13) {};
\fill[racinesimple] (0.13,0.13) circle (\grosseursimple);
\node[label=right :{$\alpha_2$}] (b) at (3.87,0.13) {};
\fill[racinesimple] (3.87,0.13) circle (\grosseursimple);
\node[label=above :{$\alpha_1$}] (g) at (2,3.37) {};
\fill[racinesimple] (2,3.37) circle (\grosseursimple);
\node[label=below :{\tiny $\alpha_2+\alpha_3$}]  at (2,0.2) {};
\fill[racinedih] (2,0.13) circle (\grosseursimple);
\node[label={[xshift=0.2cm, yshift=-0.3cm] \tiny $\beta$}]  at (2.41,1.27) {};
\fill[racinedih] (2.41,1.27) circle (\grosseursimple);

\node[label={\fontsize{6}{8}\selectfont $1(\alpha_2)$}] at (2.85,2.05) {};
\node[label={\fontsize{6}{8}\selectfont   $2(\alpha_1)$}]  at (3.4,1.2) {};
\node[label={\fontsize{6}{8}\selectfont   $12(\alpha_1)$}]  at (3.15,1.6) {};

\node[label={\tiny  $32(\alpha_1)$}]  at (2,0.6) {};
\node[label={\tiny  $31(\alpha_2)$}]  at (2,1.5) {};
\node[label={\tiny  $312(\alpha_1)$}]  at (1.9,1.15) {};

\node[label={\tiny  $\gamma$}]  at (2.3,1.3) {};
\node[label={\tiny  $\nu$}]  at (2.3,0.95) {};

\node[label=above :{\color{teal} $\Phi(w)$}] at (1.3,0.8) {};


\end{tikzpicture}}
\caption{The projective root system of type $H_3$ with Coxeter graph given in the upper left; the roots are the blue dots. The shaded polygon is the {\em inversion polytope} of the element $w=312121321$. The inversion set $\Phi(w)$ is constituted of the roots contained in the inversion polytope of $w$.  The set of short inversions $\Phi^1(w)$ consists of the roots that are the vertices of that polytope.}
\label{fig:H3}
\end{figure}
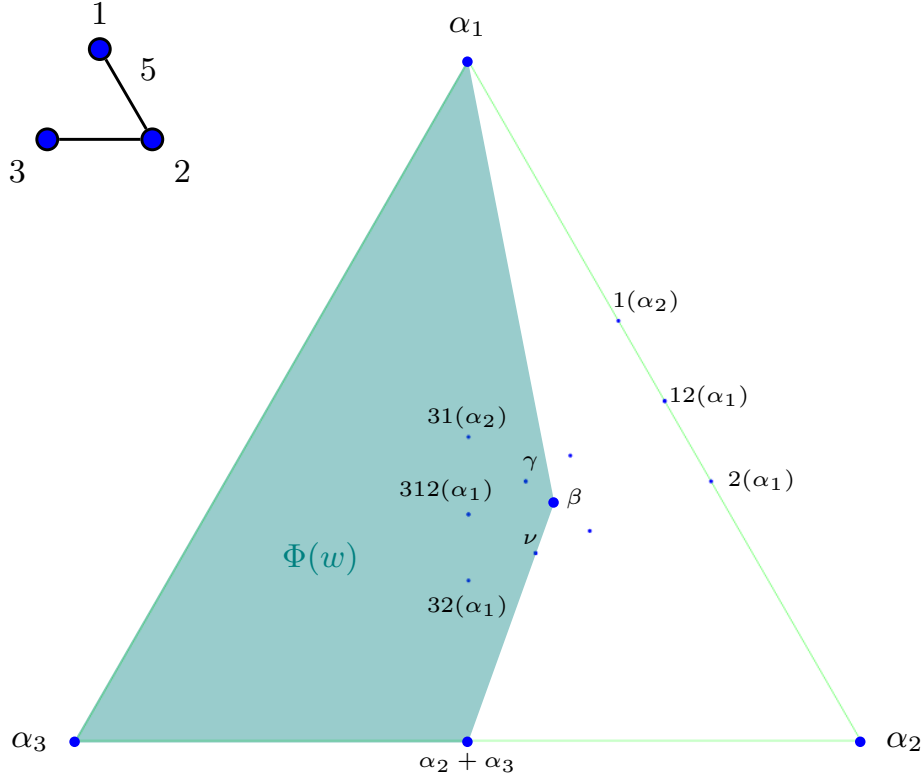
\end{ex}

The following proposition shows that the observation at the end of Example~\ref{ex:H3} is true in general.

\begin{prop}\label{prop:BasisAndDihedral} Let $w\in W$ and $\alpha,\beta\in \Phi^1(w)$ with $\alpha\not = \beta$. Then $\alpha\in \Delta_{\mathcal M_{\alpha,\beta}}$ or $\beta\in \Delta_{\mathcal M_{\alpha,\beta}}$. In particular:
\begin{enumerate}
\item If $\Delta_{\mathcal M_{\alpha,\beta}}=\{\alpha,\alpha'\}$ and $\beta\not = \alpha'$, then $\alpha'\notin \Phi(w)$.
\item If $\Delta_{\mathcal M_{\alpha,\beta}}=\{\alpha,\beta\}$, then $\Phi^+_{\mathcal M_{\alpha,\beta}}\subseteq \Phi(w)$ and  $\mathcal M_{\alpha,\beta}$ is finite.
\end{enumerate}
\end{prop} 
\begin{proof} Write  $w=uv$ with $u\in \mathcal M_{\alpha,\beta}$ and $v \in X_{\mathcal M_{\alpha,\beta}}$. By Proposition~\ref{prop:CosetRep}~(1), $\Phi_{\mathcal M_{\alpha,\beta}}(u)=\Phi(w)\cap \Phi_{\mathcal M_{\alpha,\beta}}$ is a non-empty inversion set, since $\alpha,\beta\in\Phi(w)$. By Proposition~\ref{prop:InvMDRG}, we know that $\Phi_{\mathcal M_{\alpha,\beta}}(u)=\cone_\Phi(\alpha',\gamma)$ with $\alpha'\in \Delta_{\mathcal M_{\alpha,\beta}}$ and $\gamma\in \Phi^+_{\mathcal M_{\alpha,\beta}}$. Since $\cone_\Phi(\alpha',\gamma)\subseteq \Phi(w)$ and $\alpha,\beta$ span extreme rays of $\cone(\Phi(w))$, we must have $\alpha=\alpha'\in \Delta_{\mathcal M_{\alpha,\beta}}$ (and $\beta=\gamma$) or $\beta=\alpha'\in \Delta_{\mathcal M_{\alpha,\beta}}$ (and $\alpha=\gamma$). (1) follows easily from this description and (2)  follows easily since $\Phi(w)$ is a finite set.
\end{proof}

\subsection{Descent-roots}  The well-known left and right descent sets of $w\in W$ have their natural counterparts in $\Phi^1(w)$:
\begin{itemize}
\item The {\em left descent set} $D_L(w)=\{s \in S\mid sw \lhd w\}$ is in bijection with the set of {\em left descent-roots:} $\Phi^L(w)=\Phi(w)\cap \Delta$. 
\item The {\em right descent set} $D_R(w)=\{s \in S\mid ws \lhd w\}$ is in bijection with the set of {\em right descent-roots:} 
$
\Phi^R(w)=\{-w(\alpha_s)\mid s\in D_R(w)\}.
$
\end{itemize}

\begin{ex}[Continuation of Example~\ref{ex:H3}] Since $w=31.2121321=13.2121321$, we have $D_L(w)=\{1,3\}$. Hence  $\Phi^L(w)=\{\alpha_1,\alpha_3\}$. We also have $\Phi^R(w)=\{\alpha_1,\beta\}$. Indeed $D_R(w)=\{1,3\}$,  since $w=3.12121.321=3.2121.232.1=3212132.31$. Furthermore, $\beta=31212132(\alpha_1)$ and $\alpha_1=32121321(\alpha_3)$.  Both $\Phi^L(w)$ and $\Phi^R(w)$ are the set of vertices of a face of the inversion polytope of $w$. Furthermore, there is a (unique in this case) hyperplane supporting the face $\Phi^R(w)$ that separates $\Phi(w)$ from $\Phi^+\setminus \Phi(w)$.
\end{ex}

\begin{prop}\label{prop:DescentRoot} Let $w\in W$, then $\Phi^L(w)$ and $\Phi^R(w)$ are subsets of $\Phi^1(w)$. Moreover,  $\beta\in \Phi(w)$ satisfies $\beta\in \Phi^R(w)$ if and only if $\Phi(w)\setminus\{\beta\}$ is the inversion set of some $u\in W$. 
\end{prop}
\begin{proof} Let $\alpha\in \Phi^L(w)$, then there is $s\in S$ such that $\alpha=\alpha_s\in\Phi(w)$. So $\ell(s_\alpha w)=\ell(sw)<\ell(w)$.  Since $\ell(sw)=\ell(w)\pm 1$, we conclude that $s_\alpha w\lhd w$ and  $\alpha\in \Phi^1(w)$. Now let $\beta = -w(\alpha_s)\in \Phi^R(w)$. Since $s\in D_R(w)$, there is $u\in W$ such that $w=us$ is a reduced product. So 
$$
\beta = -w(\alpha_s) = -us(\alpha_s)=u(\alpha_s)\in \Phi^+.
$$
Now, $s_\beta w= s_{u(\alpha_s)}w=usu^{-1}(us)=u\lhd w$. Therefore $\beta\in \Phi^1(w)$. The last statement follows from Proposition~\ref{prop:InvBasics} since the product $w=us$ is reduced:
$$
\Phi(w)=\Phi(u)\sqcup u(\Phi(s))=\Phi(u)\sqcup \{u(\alpha_s)\}=\Phi(u)\sqcup \{\beta\}.\qedhere
$$  
\end{proof}

\begin{remark} The statement that there is a hyperplane supporting the face $\Phi^R(w)$ that separates $\Phi(w)$ from $\Phi^+\setminus \Phi(w)$ for any $w\in W$ is true in general, which is not difficult to prove using Proposition~\ref{prop:DescentRoot} and the characterization of  inversion sets as sets separated by an hyperplane from their complements in $\Phi^+$ (see for instance  \cite[Proposition 2.11]{HoLa16}).  
 \end{remark}

\subsection{Short inversion graphs}\label{se:PosetBasisInversion} 
Let $w\in W$. For $\alpha,\beta\in \Phi^1(w)$, we write $\alpha \precinv_w \beta$ if $\beta \notin \Delta_{\mathcal M_{\alpha,\beta}}$.  By Proposition~\ref{prop:BasisAndDihedral}, this is equivalent to  $\alpha \in \Delta_{\mathcal M_{\alpha,\beta}}$ and $\beta \notin \Delta_{\mathcal M_{\alpha,\beta}}$. 
  
\begin{defi}[Charles~\cite{Chb20}]
The {\em short inversion graph} of $w\in W$ is the digraph defined as follows:
\begin{itemize} 
\item the set of vertices is $\Phi^1(w)$;
\item   $(\alpha,\beta)$ is an edge if $\alpha\precinv_w \beta$. 
\end{itemize}
\end{defi}

\begin{remark} The short inversion graph was previously introduced in \cite{Chb20} in the context of rank 3 Coxeter systems and under the name {\em bipodality graph}. In this article, the proof of \cite[Conjecture~2]{DyHo16} given by the author for rank 3 is heavily dependent of  the fact that hyperplanes are projective lines in this case.
\end{remark}

The following result is a direct consequence of Proposition~\ref{prop:DiDep}.

\begin{prop}\label{prop:SIGDep} Let $w\in W$ and $\alpha,\beta \in \Phi^1(w)$. If $\alpha\precinv \beta$, then $\dep(\alpha)<\dep(\beta)$. 
\end{prop}

\begin{ex}[Continuation of Example~\ref{ex:H3}]\label{ex:H3.2}  With the notation of Example~\ref{ex:H3} and Figure~\ref{fig:H3}, the short inversion graph of $w=312121321$ is:
\smallskip
\begin{center}
\begin{tikzpicture}
 \tikzstyle{every node}=[font=\small]
\def\grosseursimple{0.025}
\node[draw=blue,circle,fill=blue!55,minimum size=2pt,inner sep=2pt,label=above:$\alpha_1$] (a1) at (0,0) {};
\node[draw=blue,circle,fill=blue!55,minimum size=2pt,inner sep=2pt,label=above:$\alpha_3$] (a3) at (2,0) {};
\node[draw=blue,circle,fill=blue!55,minimum size=2pt,inner sep=2pt,label=above:$\alpha_2+\alpha_3$] (a23) at (4,0) {};
\node[draw=blue,circle,fill=blue!55,minimum size=2pt,inner sep=2pt,label=above:$\beta$] (b) at (6,0) {};
\draw[-latex,line width = 0.7,color=blue] (a3) -- (a23);
\draw[-latex,line width = 0.7,color=blue] (a23) -- (b);
\end{tikzpicture}
\end{center}
\smallskip
Indeed, we have (see Figure~\ref{fig:H3}):
\begin{itemize}
\item  $\alpha_3 \precinv_w \alpha_2+\alpha_3$, since $\Delta_{\mathcal M_{\alpha_3,\alpha_2+\alpha_3}} = \{\alpha_2,\alpha_3\}$;
\item $ \alpha_2+\alpha_3\precinv_w \beta$, since $\Delta_{\mathcal M_{\alpha_2+\alpha_3,\beta}} = \{\alpha_2+\alpha_3,1(\alpha_2)\}$;
\item there are no other relations between short inversions of $w$, since 
$$
\Delta_{\mathcal M_{\alpha_1,\alpha_3}} = \{\alpha_1,\alpha_3\},\ \Delta_{\mathcal M_{\alpha_3,\beta}} = \{\alpha_3,\beta\}, \Delta_{\mathcal M_{\alpha_1,\beta}} = \{\alpha_1,\beta\}
$$
and $\Delta_{\mathcal M_{\alpha_1,\alpha_2+\alpha_3}} = \{\alpha_1,\alpha_2+\alpha_3\}$.
\end{itemize}
We observe in particular that $\dep(\alpha_3)=0<\dep(\alpha_2+\alpha_3)=1<\dep(\beta)= 6$, since $\beta=31212132(\alpha_1)=212132(\alpha_1)$.
\end{ex}

\begin{remark}[On computing the short inversion graph] Let $w\in W$. The computation of $\Phi^1(w)$ knowing the inversion set $\Phi(w)$ is easy: we just have identify which of the $\beta\in \Phi(w)$ verifies $\ell(s_\beta w)=\ell(w)-1$. However, computing the short inversion graph of $w$ is more difficult. In this article, we used SageMath~\cite{sage} and the definition to compute the short inversion graph of $w$. More precisely,  let $\alpha,\beta\in \Phi^1(w)$, then:
\begin{enumerate}[(i)]
\item If $|B(\alpha,\beta)|\geq 1$ then it is well-known that $\mathcal M_{\alpha,\beta}$ is infinite (see for instance \cite[Proposition~1.5]{HoLaRi14}). Since $\alpha,\beta\in \Phi(w)$, we have $B(\alpha,\beta)\geq 1$ by \cite[Proposition~2.15]{DyHo16}. Since $\mathcal M_{\alpha,\beta}$ is an infinie dihedral reflection subgroup, we have  either $\alpha\in \Phi(\beta)$ or $\beta\in \Phi(\alpha)$. Then $\beta \notin \Delta_{\mathcal M_{\alpha,\beta}}$ if and only if $\alpha\in \Phi(\beta)$. Finally, we obtain the following criterion in this case: $\alpha\precinv \beta$ if and only if $\ell( s_\alpha s_\beta)<\ell(s_\beta)$.
\item If $|B(\alpha,\beta)|< 1$, then it is well known that $\mathcal M_{\alpha,\beta}$ is finite (see \cite[Proposition~1.5]{HoLaRi14} again). In this case, we compute explicitly $\Phi_{\mathcal M_{\alpha,\beta}}$ and $\Delta_{\mathcal M_{\alpha,\beta}}$, which requires computing $\Phi^+$ up to a certain value of the depth statistic. 
\end{enumerate}
Our algorithm is not efficient and can be improved. In~\S\ref{ss:Edges1} and \S\ref{ss:Edges2} below, we give more in-depth criteria to find the set of edges of a short inversion graphs. 
\end{remark}

\begin{ex}\label{ex:Hypo}  Consider $(W,S)$ with $S=\{1,2,3,4\}$ and Coxeter graph: 
\begin{center}
\begin{tikzpicture}
	[scale=2,
	 q/.style={teal,line join=round},
	 racine/.style={blue},
	 racinesimple/.style={blue},
	 racinedih/.style={blue},
	 sommet/.style={inner sep=2pt,circle,draw=black,fill=blue,thick,anchor=base},
	 rotate=0]
 \tikzstyle{every node}=[font=\small]
\def\grosseursimple{0.025}
\coordinate (ancre) at (0,3);

\node[sommet,label=above left :$1$] (a1) at ($(ancre)+(0,0.5)$) {}  ;
\node[sommet,label=above right:$2$] (a2) at ($(ancre)+(0.5,0.5)$) {} edge[thick] node[auto,swap] {$4$} (a1) ;
\node[sommet,label=below right :$3$] (a3) at ($(ancre)+(0.5,0)$) {} edge[thick] node[auto,swap] {$4$} (a2) ;
\node[sommet,label=below left:$4$] (a4) at (ancre) {} edge[thick] node[auto,swap] {} (a3) edge[thick] node[auto,swap,left] {$4$} (a1)    ;
\end{tikzpicture}
\end{center}
This is an indefinite Coxeter system, i.e., not finite nor affine. 

\smallskip

\noindent Consider $w=412343$, then $\Phi^L(w)= \{\alpha_4\}$, $\Phi^R(w)=\{1(\alpha_4),412(\alpha_3)\}$ and 
$$
\Phi^1(w)=\{\alpha_4,41(\alpha_2),412(\alpha_3),1(\alpha_4)\}.
$$
The short inversion graph of $w$ is then:
\smallskip
\begin{center}
\begin{tikzpicture}
 \tikzstyle{every node}=[font=\small]
\def\grosseursimple{0.025}
\node[draw=blue,circle,fill=blue!55,minimum size=2pt,inner sep=2pt,label=below:$\alpha_4$] (a1) at (0,0) {};
\node[draw=blue,circle,fill=blue!55,minimum size=2pt,inner sep=2pt,label=above right:$41(\alpha_2)$] (a2) at (2,2) {};
\node[draw=blue,circle,fill=blue!55,minimum size=2pt,inner sep=2pt,label=right:$412(\alpha_3)$] (a3) at (2,4) {};
\node[draw=blue,circle,fill=blue!55,minimum size=2pt,inner sep=2pt,label=left:$1(\alpha_4)$] (a4) at (-2,4) {};
\draw[-latex,line width = 0.7,color=blue] (a1) -- (a2);
\draw[-latex,line width = 0.7,color=blue] (a1) -- (a3);
\draw[-latex,line width = 0.7,color=blue] (a1) -- (a4);
\draw[-latex,line width = 0.7,color=blue] (a2) -- (a3);
\end{tikzpicture}
\end{center}
Now, for $w^{-1}=343214$, we have $\Phi^L(w^{-1})= \{\alpha_3,\alpha_4\}$, $\Phi^R(w^{-1})=\{4321(\alpha_4)\}$ and 
$$
\Phi^1(w^{-1})=\{\alpha_3,\alpha_4,43(\alpha_2),4321(\alpha_4)\}.
$$
The short inversion graph of $w^{-1}$ is then:
\smallskip
\begin{center}
\begin{tikzpicture}
 \tikzstyle{every node}=[font=\small]
\def\grosseursimple{0.025}
\node[draw=blue,circle,fill=blue!55,minimum size=2pt,inner sep=2pt,label=below:$4321(\alpha_4)$] (a1) at (0,0) {};
\node[draw=blue,circle,fill=blue!55,minimum size=2pt,inner sep=2pt,label=above right:$43(\alpha_2)$] (a2) at (2,2) {};
\node[draw=blue,circle,fill=blue!55,minimum size=2pt,inner sep=2pt,label=right:$\alpha_4$] (a3) at (2,4) {};
\node[draw=blue,circle,fill=blue!55,minimum size=2pt,inner sep=2pt,label=left:$\alpha_3$] (a4) at (-2,4) {};
\draw[latex-,line width = 0.7,color=blue] (a1) -- (a2);
\draw[latex-,line width = 0.7,color=blue] (a1) -- (a3);
\draw[latex-,line width = 0.7,color=blue] (a1) -- (a4);
\draw[latex-,line width = 0.7,color=blue] (a2) -- (a3);
\end{tikzpicture}
\end{center}
Observe that the short inversion graph of $w^{-1}$ is the opposite graph of the short inversion graph of $w$.
\end{ex}

For a digraph $\Gamma$, we denote by $\Gamma^\op$ the {\em opposite digraph of $\Gamma$}, which is obtained by reversing all the edges of $\Gamma$. 

\begin{prop}\label{prop:Op} Let $w\in W$. The map $\alpha\mapsto -w^{-1}(\alpha)$ induces an isomorphism between the digraph $(\Phi^1 (w),\precinv_w)$ and $(\Phi^1 (w^{-1}),\precinv_{w^{-1}})^\op$. In particular $\alpha\precinv_w \beta$ if and only if $-w^{-1} (\beta)\precinv_{w^{-1}}-w^{-1} (\alpha)$.
\end{prop}
\begin{proof} The map $\alpha\mapsto -w^{-1}(\alpha)$ is a bijection between $\Phi(w)$ and $\Phi(w^{-1})$; indeed:
$$
\Phi(w^{-1})=\Phi^+\cap w^{-1}(\Phi^-) = w^{-1}(w(\Phi^+)\cap \Phi^-)=-w^{-1}(w(\Phi^-)\cap \Phi^+)=-w^{-1}(\Phi(w)).
$$ 
Now $\beta\in \Phi^1(w)$ if and only if $-w^{-1}(\beta)\in \Phi^1(w^{-1})$ since
$$
\ell(s_\beta w)-\ell(w) = \ell(w^{-1}s_\beta) -\ell(w^{-1})=\ell(s_{-w^{-1}(\beta)}w^{-1})-\ell(w^{-1}).
$$

Since the role of $w$ and $w^{-1}$ are interchangeable, it remains only to show that if $\alpha\precinv_w \beta$ then $-w^{-1} (\beta)\precinv_{w^{-1}}-w^{-1} (\alpha)$. Let $\alpha\precinv_w \beta$. Write $W'=\mathcal M_{\alpha,\beta}$, then it is well-known that $\mathcal M_{-w^{-1}(\alpha),-w^{-1}(\beta)}=w^{-1}W'w$.  By definition, we have to show that $-w^{-1}(\alpha)\notin \Delta_{w^{-1}W'w}$. We know that $\alpha\in \Delta_{W'}$ and $\beta\notin\Delta_{W'}$ by definition. Write $\Delta_{W'}=\{\alpha,\gamma\}$, then there is $a,b>0$ such that $\beta=a\alpha+b\gamma$. Therefore $\alpha=1/a \beta + b/a (-\gamma)$. Notice that $\gamma\notin \Phi(w)$ by Proposition~\ref{prop:BasisAndDihedral}~(1). Hence $w^{-1}(\gamma)=-w^{-1}(-\gamma)\in \Phi^+\cap \Phi_{w^{-1}W'w}=\Phi^+_{w^{-1}W'w}$ and 
$$
-w^{-1}(\alpha)=1/a (-w^{-1}(\beta))+ b/a (-w^{-1}(-\gamma)),
$$ 
with $1/a,b/a>0$ and $-w^{-1}(\beta)\in \Phi^+_{w^{-1}W'w}$. Therefore $-w^{-1}(\alpha)$ can not be a simple root of $w^{-1}W'w$; so by definition $-w^{-1} (\beta)\precinv_{w^{-1}}-w^{-1} (\alpha)$.
\end{proof}

\subsection{Short inversion posets}\label{se:PosetBasisInversion} For $w\in W$, we define the relation $\preceq_w$ to be  the transitive and reflexive closure of $\precinv_w$, which turns out to be a partial order on~$\Phi^1(w)$.

\begin{prop}\label{prop:Above} The relation $\preceq_w$ is a partial order on $\Phi^1(w)$. Moreover, for any reduced word  $w=s_1\cdots s_k$ consider  the following  total order $\leq$  on~$\Phi(w)$:
$$
\alpha_{s_1} < s_1(\alpha_{s_2})< \cdots < s_1\cdots s_{k-1}(\alpha_{s_k})
$$
Then $\alpha\preceq_w \beta$ implies $\alpha\leq \beta$ and $\dep(\alpha)\leq\dep(\beta)$ for any $\alpha,\beta\in \Phi^1(w)$.
\end{prop}
\begin{proof} By definition of  $\preceq_w$, one just has to show that $\preceq_w$ is antisymmetric. This is a direct consequence of the `moreover' statement. Notice first that the statement on the depth statistic is a direct consequence of transitivity and Proposition~\ref{prop:SIGDep}. 

Let $w=s_1\cdots s_k$ be a reduced word and define the following  total order on $\Phi(w)$:
$$
\alpha_{s_1} < s_1(\alpha_{s_2})< \cdots < s_1\cdots s_{k-1}(\alpha_{s_k})
$$
Let $\alpha,\beta\in \Phi^1(w)$ such that $\alpha\precinv_w \beta$.  We show that $\alpha<\beta$. By definition $\alpha \in \Delta_{\mathcal M_{\alpha,\beta}}$ and $\beta \notin \Delta_{\mathcal M_{\alpha,\beta}}$. Let $\alpha'\in\Delta_{\mathcal M_{\alpha,\beta}}$ the other simple root. So $\alpha'\notin \Phi(w)$ by Proposition~\ref{prop:BasisAndDihedral}.
 Let $1\leq i\leq k$ such that $\beta = s_1\cdots s_{i-1}(\alpha_{s_i})$. Write $u=s_1\cdots s_i\leq_R w$. Since $\alpha'\notin\Phi(u)$, we must have $\alpha\in \Phi(u)$ by  Proposition~\ref{prop:BasisAndDihedral} again. Therefore there is $j<i$ such that $\alpha = s_1\cdots s_{j-1}(\alpha_{s_j})$ and so $\alpha<\beta$. 
\end{proof}

\begin{remark} (1) The relation $\precinv_w$ is not the cover relation for $\preceq_w$. In fact, if $\beta$ covers $\alpha$ in $(\Phi^1(w),\preceq_w)$, then $\alpha\precinv_w \beta$ but the converse is not necessarily true; see for instance Example~\ref{ex:Cover} below.  

\smallskip
\noindent (2) In the proof of Proposition~\ref{prop:Above}, the total order on $\Phi(w)$ is in fact the restriction of an {\em admissible order} on $\Phi^+$ to $\Phi(w)$. Admissible orders on $\Phi^+$ are in bijection to {\em reflection orders}, which plays a role in Kazhdan-Lusztig theory; see for instance \cite[\S5.2]{BjBr05} or \cite[\S2.1]{Dy19} for more details and references. 
\end{remark}

\begin{ex}\label{ex:Cover} We consider $(W,S)$ as in Example~\ref{ex:Hypo}.  Write $w=1234232314$, then $\Phi^L(w)= \{\alpha_1\}$, $\Phi^R(w)=\{123432321(\alpha_4)\}$ and 
$$
\Phi^1(w)=\{\alpha_1,1(\alpha_2),31(\alpha_2),\gamma=1234232(\alpha_1),1(\alpha_4),\beta=123432321(\alpha_4)\}.
$$
The short inversion graph of $w$ is then:
\smallskip
\begin{center}
\begin{tikzpicture}
 \tikzstyle{every node}=[font=\small]
\def\grosseursimple{0.025}
\node[draw=blue,circle,fill=blue!55,minimum size=2pt,inner sep=2pt,label=below:$\alpha_1$] (a1) at (0,0) {};
\node[draw=blue,circle,fill=blue!55,minimum size=2pt,inner sep=2pt,label=above right:$1(\alpha_2)$] (a2) at (2,2) {};
\node[draw=blue,circle,fill=blue!55,minimum size=2pt,inner sep=2pt,label=right:$31(\alpha_2)$] (a3) at (5,2.5) {};
\node[draw=blue,circle,fill=blue!55,minimum size=2pt,inner sep=2pt,label=above:$\gamma$] (a4) at (2,4) {};
\node[draw=blue,circle,fill=blue!55,minimum size=2pt,inner sep=2pt,label=left:$1(\alpha_4)$] (a5) at (-2,3) {};
\node[draw=blue,circle,fill=blue!55,minimum size=2pt,inner sep=2pt,label=above:$\beta$] (a6) at (0,6) {};
\draw[-latex,line width = 0.7,color=blue] (a1) -- (a2);
\draw[-latex,line width = 0.7,color=blue] (a2) -- (a3);
\draw[-latex,line width = 0.7,color=blue] (a1) -- (a3);
\draw[-latex,line width = 0.7,color=blue] (a2) -- (a4);
\draw[-latex,line width = 0.7,color=blue] (a3) -- (a4);
\draw[-latex,line width = 0.7,color=blue] (a1) -- (a4);
\draw[-latex,line width = 0.7,color=blue] (a1) -- (a5);
\draw[-latex,line width = 0.7,color=blue] (a1) -- (a6);
\draw[-latex,line width = 0.7,color=blue] (a2) -- (a6);
\draw[-latex,line width = 0.7,color=blue] (a3) -- (a6);
\draw[-latex,line width = 0.7,color=blue] (a4) -- (a6);
\draw[-latex,line width = 0.7,color=blue] (a5) -- (a6);
\end{tikzpicture}
\end{center}
Now, the Hasse diagram of the short inversion poset of $w$ is: 
\smallskip
\begin{center}
\begin{tikzpicture}
 \tikzstyle{every node}=[font=\small]
\def\grosseursimple{0.025}
\node[draw=blue,circle,fill=blue!55,minimum size=2pt,inner sep=2pt,label=below:$\alpha_1$] (a1) at (0,0) {};
\node[draw=blue,circle,fill=blue!55,minimum size=2pt,inner sep=2pt,label=above right:$1(\alpha_2)$] (a2) at (2,1.5) {};
\node[draw=blue,circle,fill=blue!55,minimum size=2pt,inner sep=2pt,label=right:$31(\alpha_2)$] (a3) at (2,3) {};
\node[draw=blue,circle,fill=blue!55,minimum size=2pt,inner sep=2pt,label=above:$\gamma$] (a4) at (2,4.5) {};
\node[draw=blue,circle,fill=blue!55,minimum size=2pt,inner sep=2pt,label=left:$1(\alpha_4)$] (a5) at (-2,3) {};
\node[draw=blue,circle,fill=blue!55,minimum size=2pt,inner sep=2pt,label=above:$\beta$] (a6) at (0,6) {};
\draw[-latex,line width = 0.7,color=blue] (a1) -- (a2);
\draw[-latex,line width = 0.7,color=blue] (a2) -- (a3);
\draw[-latex,line width = 0.7,color=blue] (a3) -- (a4);
\draw[-latex,line width = 0.7,color=blue] (a1) -- (a5);
\draw[-latex,line width = 0.7,color=blue] (a4) -- (a6);
\draw[-latex,line width = 0.7,color=blue] (a5) -- (a6);
\end{tikzpicture}
\end{center}
Observe that any short inversion is {\em sandwiched} between a left descent-root and a right descent-root.
\end{ex}

\begin{remark} Let $w\in W$. In $(\Phi^1(w),\preceq_w)$, the elements of $\Phi^L(w)$ are minimal and therefore  incomparable. By Proposition \ref{prop:Op}, the elements of $\Phi^R(w)$ are maximal in $(\Phi^1(w),\preceq_w)$, and hence also incomparable.
\end{remark}

We state now the main result of this section. 

\begin{thm}\label{thm:BasisInversion} Let $w\in W$. For the poset $(\Phi^1(w),\preceq_w)$,  the minimal elements are the left-descent roots in $\Phi^L(w)$ and the maximal elements are the right-descent roots in $\Phi^R(w)$. In other words,  for any $\beta\in \Phi^1(w)$ there is $\alpha\in \Phi^L(w)$ and $\gamma\in \Phi^R(w)$ such that $\alpha\preceq _w\beta\preceq_w \gamma$. 
\end{thm}

We dedicate the remainder of this section to the proof of this theorem. 

\subsection{Short inversions and reduced products}

The following two results are critical to prove Theorem~\ref{thm:BasisInversion}.

\begin{prop}\label{prop:Key1} Let $x,w\in W$ and $\alpha,\gamma\in\Phi^+\setminus\Phi(x)$. 
\begin{enumerate}
\item We have $\alpha\in \Phi(xw)$ if and only if $x^{-1}(\alpha)\in \Phi(w)$.
\item If $xw$ is a reduced product and $\alpha\in \Phi^1(xw)$ then $x^{-1}(\alpha)\in \Phi^1(w)$.
\item If $xw$ is a reduced product and $\alpha\precinv_{xw}\gamma$ then $x^{-1} (\alpha) \precinv_w x^{-1} (\gamma)$.
\end{enumerate}
\end{prop}
\begin{proof} (1). First, notice that  by assumption $x^{-1}(\alpha)\in \Phi^+$ and 
$$
(xw)^{-1}(\alpha)=w^{-1}x^{-1}(\alpha)=w^{-1}(x^{-1}(\alpha)).
$$
Therefore:
$
\alpha\in\Phi(xw) \iff w^{-1}(x^{-1}(\alpha))\in \Phi^- \iff x^{-1}(\alpha)\in \Phi(w).
$ 

\noindent (2). Assume that $\alpha\in \Phi^1(xw)$. By (1) we have  $x^{-1}(\alpha)\in \Phi(w)$. Therefore
$$
\ell(s_{x^{-1}(\alpha)}w)=\ell(x^{-1}s_\alpha xw)<\ell(w),
$$
and $\ell(x^{-1}s_\alpha xw)\leq\ell(w)-1$. Now, by classical properties of the length function we obtain:
$$
\ell(x^{-1}s_\alpha xw)\geq\ell(xw)-\ell(x^{-1}s_\alpha)\geq \ell(xw)-\ell(x^{-1})-1=\ell(w)-1;
$$
the last equality is a consequence of $xw$ being a reduced product. Therefore $\ell(s_{x^{-1}(\alpha)}w)=\ell(x^{-1}s_\alpha xw)=\ell(w)-1$ and $x^{-1}(\alpha)\in \Phi^1(w)$.

\noindent (3). Set $W'=\mathcal M_{\alpha, \gamma}$. So  $\mathcal M_{x^{-1} (\alpha),x^{-1} (\gamma)}=x^{-1}W'x$. Thanks to (2), we just have to prove that $x^{-1} (\gamma)$ is not a simple root of $x^{-1}W'x$. Since $\alpha\precinv_{xw} \gamma$, $\alpha,\gamma\in \Phi^1(xw)$  by definition and $\alpha\in\Delta_{W'}$ by Proposition~\ref{prop:BasisAndDihedral}. Let $\alpha'$ be the other simple root in $\Delta_{W'}$. Since $\alpha\precinv_{xw} \gamma$, $\gamma\notin\Delta_{W'}$ therefore $\alpha'\notin \Phi(xw)$ by Proposition~\ref{prop:BasisAndDihedral}~(1). Since $xw$ is reduced, $\Phi(x)\subseteq\Phi(xw)$ by     Proposition~\ref{prop:InvBasics}, and so $\alpha'\notin \Phi(x)$. Therefore $x^{-1}(\alpha')\in\Phi^+$. So $x^{-1}(\Delta_{W'})\subseteq \Phi^+$ and therefore $x^{-1}(\Delta_{W'})=\Delta_{x^{-1}W'x}$ by maximality of $W'$. Hence $x^{-1} (\gamma)\notin \Delta_{x^{-1}W'x}$ and $x^{-1} (\alpha) \precinv_w x^{-1} (\gamma)$.
\end{proof}

\begin{cor}\label{cor:Key1} Let $x,w\in W$ and $\beta\in\Phi(w)$.  Write $w'=s_\beta w$ and assume that $xw$ and $xw'$ are reduced products. Then $x(\beta)\in \Phi(xw)$. Furthermore
 $\beta\in \Phi^1(w)$ if and only if $x(\beta)\in \Phi^1(xw)$.
\end{cor}
\begin{proof} Write $\alpha=x(\beta)$. Since $xw$ is a reduced product, we have $\Phi(xw)=\Phi(x)\sqcup x(\Phi(w))$ by Proposition~\ref{prop:InvBasics}. Hence  $\alpha\in \Phi(xw)\setminus\Phi(x)$. By Proposition~\ref{prop:Key1}, we have $\alpha=x(\beta)\in \Phi^1(xw)$ implies $\beta=x^{-1}(\alpha)\in \Phi^1(w)$. Conversely, if $\beta\in\Phi^1(w)$, then $\ell(w')=\ell(s_{\beta}w)=\ell(x^{-1}s_\alpha xw)=\ell(w)-1$. Therefore, since $xw$ and $xw'$ are reduced products, we have:
$$
\ell(xw') = \ell(x)+\ell(w') =\ell(x)+\ell(w)-1=\ell(xw)-1.
$$
Finally, we obtain:
$$
\ell(s_\alpha xw)=\ell(xx^{-1}s_\alpha xw)=\ell(xs_\beta w)=\ell(xw') =\ell(xw)-1,
$$
which means that $\alpha=x(\beta)\in \Phi^1(xw)$.
\end{proof}

\subsection{Some edges in short inversion graphs}\label{ss:Edges1} The key to prove Theorem~\ref{thm:BasisInversion} is to explicitly construct, for $w\in W$ and for each $\beta\in \Phi^1(w)\setminus\Phi^L(w)$, a short inversion $\alpha\in \Phi^1(w)$ such that $\alpha\precinv_w\beta$.  For such a $\beta\in \Phi^1(w)\setminus\Phi^L(w)$, the idea consists in considering $g\in W$ such that $g(\beta)\in\Delta$ and $\ell(g)=\dep(\beta)$, which exists by definition of the depth statistic. The first proposition below handles the case where $D_L(w)\cap D_R(g)$ is not empty.

\begin{prop}\label{prop:First} Let $w\in W$ and $\beta\in \Phi^1(w)$ with $\dep(\beta)>0$. Then $\alpha_s\precinv_w \beta$ for all  $s\in D_L(w)$ such that  $\dep(s(\beta))<\dep(\beta)$. In this case, $\dep(\alpha_s)=0<\dep(\beta)$. 
\end{prop}
\begin{proof} Let $s\in D_L(w)$ such that $\dep(s(\beta))<\dep(\beta)$. In particular, $\alpha_s\in \Phi^L(w)\subseteq \Phi_1(w)$ and $\alpha_s\in \Delta_{\mathcal M_{\alpha_s,\beta}}$, since $S\cap \mathcal M_{\alpha_s,\beta}\subseteq \chi(\mathcal M_{\alpha_s,\beta})$. Since $\dep(s(\beta))<\dep(\beta)$, we have $B(\alpha_s,\beta)>0$ by Eq.~(\ref{eq:Depth}). So $\beta\notin \Delta_{\mathcal M_{\alpha_s,\beta}}$ by Eq~(\ref{eq:BasisAndDihedral}); hence $\alpha_s\precinv_w \beta$.
\end{proof}

\begin{ex}[Continuation of Example~\ref{ex:H3.2}] Here $w=312121321$ in $W$ of type $H_3$. The relation $\alpha_3\precinv_w \alpha_2+\alpha_3$ is obtained by Proposition~\ref{prop:First} with $s=3\in D_L(w)$ and $\beta=\alpha_2+\alpha_3$ since $\dep(3(\alpha_2+\alpha_3))=\dep(\alpha_2)=0<\dep(\alpha_2+\alpha_3)=1$.
\end{ex}

The second case is more subtle: assume now that $sw$ is a reduced product for  $s\in D_R(g)$.   So there is  $s\leq_L x\leq_L g$ such that $xw$ is a reduced product; take such a maximal $x$ with respect to $\leq_L$. We have to deal with two cases: $x=g$ and $x<_L g$. Before dealing with these cases, we need a lemma.

\begin{lem}\label{prop:BruhatKey} Let $w\in W$ and $\beta\in \Phi^1(w)$ and write $w'=s_\beta w$. Let $g\in W$ such that $g(\beta)\in \Delta$ and $\ell(g)=\dep(\beta)$. For 
any $x\leq_L g$ such that $xw$ is a reduced product, the product $xw'$ is also reduced. 
\end{lem}
\begin{proof} We proceed by induction on $\ell(x)$. If $\ell(x)=0$ there is nothing to show. Now let $s\in D_L(x)$ and write $x= sx'$, which is a reduced product.  Notice that $x'<_Lx=sx'\leq_L g$. On one hand, by Corollary~\ref{cor:Palind}, we have:
$$
(\star)\qquad \ell(s_{x(\beta)})=\ell(ss_{x'(\beta)}s) = \ell(s_{x'(\beta)})-2.
$$
On the other hand, since $x'w$ is a suffix of $xw$, $x'w$ is also a reduced product.  Therefore,  $x'w'$ is a reduced product by induction.  Assume by contradiction that $xw'=sx'w'$ is not a reduced product, that is, $sx'w'<x'w'$. We use Lemma~\ref{lem:BruhatKey}: write $z=sx'w'=xw'$ and $t=x's_\beta x'^{-1}=s_{x'(\beta)}$. We have therefore:
$$
z=sx'w'<sz=x'w'<x'w=tx'w'=tsz<xw=stsz .
$$
Since $x'w$ and $x'w'$ are reduced products and $\beta\in \Phi^1(w)$, we have 
$$
\ell(tsz)=\ell(x'w)=\ell(x')+\ell(w) = \ell(x')+\ell(w')+1=\ell(x'w')+1=\ell(sz)+1.
$$
So by Lemma~\ref{lem:BruhatKey} we have $\ell(sts)\not = \ell(t)-2$. But $t=s_{x'(\beta)}$ and $sts=s_{x(\beta)}$, so we obtain a contradiction with $(\star)$. Therefore $xw'$ is a reduced product. 
\end{proof}

The following proposition deals now with the case $x=g$.

\begin{prop}\label{prop:Inter} Let $g,w\in W$ and $\beta\in \Phi^1(w)$ such that $g(\beta)=\alpha_r\in \Delta$ and $\dep(\beta)=\ell(g)$.  Assume that  $gw$ is a reduced product. Let  $t\in D_L(g)$ and consider the maximal dihedral reflection subgroup $W_I$ with $I=\{r,t\}$. Write $g=g_1g_2$ with $g_1\in W_I$ and $g_2\in X_{W_I}$.   
Then   $g_2^{-1}(\alpha_s)\precinv_w \beta$ where $\{s\}=I\setminus D_R(g_1)$. In particular:  $\dep(g_2^{-1}(\alpha_s))<\dep(\beta)$.
\end{prop}
\begin{proof} Observe that $W_I$ is maximal since $r,t\in S$. Also $r\notin D_L(g)$, otherwise $r\in D_L(g_1)$ and so
$g^{-1}_1(\alpha_r) \in \Phi_I^-.$ Since $g_2\in X_{W_I}$ we obtain therefore
$$
\beta = g^{-1}_2( g^{-1}_1(\alpha_r)) \in \Phi^-,
$$
which is a contradiction. In particular $g_1$ is not the longest element in $W_I$, and so $D_R(g_1)\subsetneq I$. Since $t\in D_L(g_1)$, $g_1\not = e$. Therefore  $\{u\}=D_R(g_1)$, $\{s\}= I\setminus D_R(g_1)$; moreover  $I=\{r,t\}=\{s,u\}$ and $\Delta_{W_I}=\{\alpha_s,\alpha_u\}$. Write 
$$
\gamma:=g_2(\beta)=g^{-1}_1(\alpha_r)\in \Phi^+_{W_I}\setminus \Phi(g_2)\textrm{ and }w'=s_\gamma w.
$$  
We first prove the following facts:
\begin{enumerate}[(i)]

\item $\gamma\in\Phi^1(g_2w)$. Indeed $g_2\leq_L g$ and so the suffix $g_2w$ of $gw$ is also a reduced product. By Lemma~\ref{prop:BruhatKey}, the products $gw'$ and $xw'$ are also reduced. Therefore, by Corollary~\ref{cor:Key1}, we have 
$$
g^{-1}_1 (\alpha_r)=g_2(\beta)=\gamma\in \Phi^1(g_2w).
$$

\item $\alpha_s\in \Phi^1(g_2w)$. Indeed,  $g_1g_2w$ is a reduced product since $g=g_1g_2$ and $gw$ are reduced products. Since $u$ is the last letter of the (unique) reduced word for $g_1$, $ug_2$  and $ug_2w$ are reduced  product. Therefore $\alpha_u\in \Phi_I^+\setminus \Phi(g_2w)$. Since $g_1^{-1}(\alpha_r) \in \Phi(g_2w)\cap \Phi_I$, we must have $\alpha_s\in \Phi(g_2w)$ by Proposition~\ref{prop:InvMDRG}; and $\alpha_s\in \Phi^1(g_2w)$ since $\alpha_s$ is a simple root in $\Delta$. 

\item $\gamma\in\Phi^1(g_2w)\setminus \Delta_{W_I}$. Indeed $\gamma=g_2(\beta)\notin \Delta_{W_I}$ since otherwise $g_2(\beta)\in \Delta_{W_I} \subseteq \Delta$ contradicts $\dep(\beta)=\ell(g)>\ell(g_2)$ ($g_1\not = e$). The rest of the statement follows by (i). 

\end{enumerate}

\noindent Now, by definition of the short inversion graph, (ii) and (iii) implies $\alpha_s\precinv_{g_2w} \gamma$. Since  $g_2\in X_{W_I}$, we have $\alpha_s\in \Phi^+\setminus\Phi(g_2)$; and we notice above that $\gamma\in \Phi^+\setminus\Phi(g_2)$.  Therefore, by Proposition~\ref{prop:Key1}, we conclude  that $g_2^{-1}(\alpha_s)\in\Phi^1(w)$ and
$$
g_2^{-1}(\alpha_s)\precinv_{w} g_2^{-1}(\gamma)=\beta.
$$
\end{proof}     

\begin{ex}[Continuation of Example~\ref{ex:H3.2}] Here $w=312121321$ in $W$ of type $H_3$. The relation $\alpha_2+\alpha_3\precinv_w \beta$ is obtained by Proposition~\ref{prop:Inter} as follows. We have $g=231212$ and $g(\beta)=\alpha_1$. The product $gw = 231212.312121321$ is reduced. Take $t=2\in D_L(g)$ and $I=\{2,1\}$, so $g=g_1g_2$ where $g_1=21\in W_I$ and $g_2=3212\in X_{W_I}$. Therefore $D_R(g_1)=\{1\}$ and $s=2$. Hence
$$
g_2^{-1}(\alpha_2)=2123(\alpha_2)=\alpha_2+\alpha_3\precinv_w \beta. 
$$  
\end{ex}

Finally, we deal with the case $x<_L g$. 

\begin{prop}\label{prop:Second} Let $g,w\in W$ and $\beta\in \Phi^1(w)$ such that $g(\beta)\in \Delta$ with $\dep(\beta)=\ell(g)$. Assume that $gw$ is not a reduced product. Take   $x\in W$ and  $t\in S$ such that $x<_L tx\leq_L g$, $xw$ is a reduced product but $txw$ is not reduced. Then $x^{-1}(\alpha_t)\precinv_w\beta$. In particular, $\dep(x^{-1}(\alpha_t))<\dep(\beta)$
\end{prop}
\begin{proof} Since $gw$ is not reduced, there is a suffix $x$ of $g$ and a simple reflection $t\in S$ such that $x<_L tx\leq_L g$ and $xw$ is a reduced product but $txw$ is not reduced. Write $w'=s_\beta w$. By Lemma~\ref{prop:BruhatKey}, the product $xw'$ is also reduced. Therefore, by Corollary~\ref{cor:Key1}, we have $x(\beta)\in \Phi^1(xw)$ since $\beta\in \Phi^1(w)$. Now $t\in D_L(xw)$ since $txw$ is not a reduced product. So $\alpha_t \in \Phi^1(xw)$. Notice that $\alpha_t\not =x(\beta)$ since otherwise  $\dep(\beta)=\ell(x)<\ell(g)$, a contradiction.  Since $\alpha_t\in \Phi(s_{x(\beta)})\setminus\{x(\beta)\}$, by Corollary~\ref{cor:Palind}, then $x(\beta)\notin \Delta_{\mathcal M_{\alpha_t,x(\beta)}}$, by Eq.~(\ref{eq:DRS}).  Hence $\alpha_t\precinv_{xw} x(\beta)$.  Finally, since $tx>_L x$, we have $\alpha_t\notin\Phi(x)$. Also $x(\beta)\notin\Phi(x)$ since $x^{-1}(x(\beta))=\beta\in \Phi^+$. So by Proposition~\ref{prop:Key1}, $x^{-1}(\alpha_t)\precinv_w \beta$.
\end{proof}

\begin{remark} The choice of $x$ is not unique in Proposition~\ref{prop:Second}. However, among all the possible choices of such $x$,  there is a longest one: the join $\bigvee_L G$ (for the left weak order) of the set 
$
G=\{x \in W\mid x\leq_L g,\ xw\textrm{ reduced}\}.
$ 
The join $\bigvee_L G$ is also an element of $G$ by \cite[Corollary~1.6~(b)]{Dy19}.
\end{remark}

\begin{ex}[Continuation of Example~\ref{ex:H3.2}] Here $w=312121321$ in $W$ of type $H_3$. The relation $\alpha_2+\alpha_3\precinv_w \beta$ is obtained by Proposition~\ref{prop:Second} as follows. Consider different $g$ than in the preceding example: take $g=123212$ and here also $g(\beta)=\alpha_1$. The product $gw = 123212.312121321=321 23 1 2 1 2 3 1$ is not reduced, but the product $xw$ is with $x=3212\leq_L g$. Here $t=2$ and $x<_L tx\leq_L g$. We obtain as in the preceding example that
 $
x^{-1}(\alpha_2)=2123(\alpha_2)=\alpha_2+\alpha_3\precinv_w \beta. 
$
\end{ex}

\subsection{Proof of Theorem~\ref{thm:BasisInversion}}

The following statement is a direct corollary of Proposition~\ref{prop:First}, Proposition~\ref{prop:Inter} and Proposition~\ref{prop:Second}.

\begin{cor}\label{cor:Key2} Let $w\in W$ and $\beta\in \Phi^1(w)$. Then either $\beta\in \Delta$ or there is $\alpha\in \Phi^1(w)$ such that $\alpha\precinv_w\beta$.
\end{cor}

\begin{proof}[Proof of Theorem~\ref{thm:BasisInversion}]  By Proposition~\ref{prop:Op}, it is enough to show that for any short inversion $\beta\in \Phi^1(w)$ there is $\alpha\in \Phi^L(w)=\Delta\cap \Phi^1(w)$ such that $\alpha\preceq_w \beta$. We proceed by induction on $\dep(\beta)\in\mathbb N$. If $\dep(\beta)=0$, that is $\beta \in \Delta$, we are done. If $\beta\notin \Delta$, there is  $\beta'\precinv_w \beta$, by Corollary~\ref{cor:Key2}. We conclude by induction that there is $\alpha\in \Phi^L(w)$ such that $\alpha\preceq_w\beta'\precinv_w \beta$.
\end{proof}

\subsection{On finding all edges of a short inversion graph}\label{ss:Edges2}   We could ask the question whether all edges of the short inversion graph of $w\in W$ arise as in Proposition~\ref{prop:First}, Proposition~\ref{prop:Inter} and Proposition~\ref{prop:Second}? Unfortunately it is not true. For instance, consider $(W,S)$ to be the universal Coxeter system of rank $3$, whose Coxeter graph is
\smallskip
\begin{center}
\begin{tikzpicture}
	[scale=2,
	 q/.style={teal,line join=round},
	 racine/.style={blue},
	 racinesimple/.style={blue},
	 racinedih/.style={blue},
	 sommet/.style={inner sep=2pt,circle,draw=black,fill=blue,thick,anchor=base},
	 rotate=0]
 \tikzstyle{every node}=[font=\small]
\def\grosseursimple{0.025}
\coordinate (ancre) at (0,3);

\node[sommet,label=above right:$1$] (a2) at ($(ancre)+(0.25,0.5)$) {};
\node[sommet,label=below right :$2$] (a3) at ($(ancre)+(0.5,0)$) {} edge[thick] node[auto,swap,right] {$\infty$} (a2) ;
\node[sommet,label=below left:$3$] (a4) at (ancre) {} edge[thick] node[auto,swap,below] {$\infty$} (a3) edge[thick] node[auto,swap,left] {$\infty$} (a2);
\end{tikzpicture}
\end{center}
Take $w=123$, then the short inversion graph (left-hand side)  and the Hasse diagram of the short inversion poset (right-hand side) of $w$ are:
\begin{center}
\begin{tikzpicture}
 \tikzstyle{every node}=[font=\small]
\def\grosseursimple{0.025}
\node[draw=blue,circle,fill=blue!55,minimum size=2pt,inner sep=2pt,label=below:$\alpha_1$] (a1) at (0,0) {};
\node[draw=blue,circle,fill=blue!55,minimum size=2pt,inner sep=2pt,label=above right:$1(\alpha_2)$] (a2) at (3,1) {};
\node[draw=blue,circle,fill=blue!55,minimum size=2pt,inner sep=2pt,label=right:$12(\alpha_3)$] (a3) at (2,2) {};
\draw[-latex,line width = 0.7,color=blue] (a1) -- (a2);
\draw[-latex,line width = 0.7,color=blue] (a1) -- (a3);
\draw[-latex,line width = 0.7,color=blue] (a2) -- (a3);

\node[draw=blue,circle,fill=blue!55,minimum size=2pt,inner sep=2pt,label=below:$\alpha_1$] (b1) at (5,0) {};
\node[draw=blue,circle,fill=blue!55,minimum size=2pt,inner sep=2pt,label=above right:$1(\alpha_2)$] (b2) at (8,1) {};
\node[draw=blue,circle,fill=blue!55,minimum size=2pt,inner sep=2pt,label=right:$12(\alpha_3)$] (b3) at (7,2) {};
\draw[-latex,line width = 0.7,color=blue] (b1) -- (b2);
\draw[-latex,line width = 0.7,color=blue] (b2) -- (b3);
\end{tikzpicture}
\end{center}
\smallskip
\noindent The relations $\alpha_1\precinv_w 12(\alpha_3)$ and $\alpha_1\precinv_w 1(\alpha_2)$ arise from Proposition~\ref{prop:First}. Indeed,  $1\in D_L(w)$, $\dep(1.12(\alpha_3))=\dep(2(\alpha_3))=1<\dep(12(\alpha_3))=2$ and $\dep(1.1(\alpha_2))=0<1 =\dep(1(\alpha_2)$. 

However, the relation $1(\alpha_2)\precinv_w 12(\alpha_3)$ does not arise from either Proposition~\ref{prop:First}, Proposition~\ref{prop:Inter} or Proposition~\ref{prop:Second}. Indeed, write $\beta=12(\alpha_3)$ and write $g=21$, which is the unique element of $W$ such that $g(\beta)\in\Delta$ (there is no relation in $(W,S)$). Since $gw$ is not reduced, only Proposition~\ref{prop:Second} could happen, with $x=e\leq_L xt=1\leq_L g=21$. So we only recover the relation $x^{-1}(\alpha_t)=e(\alpha_1)=\alpha_1\precinv_w \beta$. 

We could also ask the question whether all covers of the short inversion poset of $w\in W$ arise as in Proposition~\ref{prop:First}, Proposition~\ref{prop:Inter} and Proposition~\ref{prop:Second}? The example above shows also that it is not true. 

\begin{question} It would be interesting to have more in-depth understanding of the short inversion digraph, the short inversion poset and its cover relations.
\end{question}

\section{$m$-Small roots and $m$-low elements}\label{se:Low}

Let $(W,S)$ be a Coxeter system and $m\in \mathbb N$. In this section, we provide, as a consequence of Theorem~\ref{thm:BasisInversion}, a key characterization of $m$-low elements: an element $w\in W$ is $m$-low if and only if $\Phi^R(w)$ consists of $m$-small roots, see Theorem~\ref{thm:Main} below.

\subsection{Dominance order}  Introduced by Brink and Howlett in \cite{BrHo93}, the {\em dominance order}  is the partial order $\preceq_\dom$ on $\Phi^+$ defined as follows:
$$
\alpha \preceq_\dom \beta \iff (\forall w \in W,\ \beta\in \Phi(w) \implies \alpha\in \Phi(w)).
$$
The dominance order is extended to a partial order on the whole root system $\Phi$ by Fu~\cite{Fu12} as follow: for $\alpha,\beta \in \Phi$, we have
\begin{equation}\label{eq:InftyDepth}
\alpha \preceq_\dom \beta \iff (\forall w \in W,\ w(\beta)\in \Phi^-\implies w(\alpha)\in \Phi^-).
\end{equation}
The following proposition summarizes the principal properties of $(\Phi,\preceq_\dom)$.

\begin{prop}[{\cite[Lemma 3.2]{Fu12}}]\label{prop:DomFu}  Let $\alpha,\beta \in \Phi$.
\begin{enumerate}[(i)]
\item We have $\alpha\preceq_\dom \beta$ if and only if $-\beta\preceq_\dom -\alpha$.
\item If $\alpha,\beta\in \Phi^+$, then $\alpha\preceq_\dom \beta$ if and only if $\dep(\alpha)\leq \dep(\beta)$ and $B(\alpha,\beta)\geq 1$.
\item If $\alpha\in \Phi^-$ and $\beta\in \Phi^+$, then $\alpha\preceq_\dom \beta$ if and only if $B(\alpha,\beta)\geq 1$.
\item The dominance order is $W$-invariant: $\alpha\preceq_\dom \beta$ if and only if $w(\alpha)\preceq_\dom w(\beta)$ for any $w\in W$.
\item We have a dominance relation between $\alpha$ and $\beta$  if and only if $B(\alpha,\beta)\geq 1$.
\end{enumerate}
\end{prop}

As a corollary, we have these two extra properties.

\begin{cor}\label{cor:DomFu} Let $\alpha,\beta \in \Phi$ such that $\alpha\preceq_\dom \beta$. 
 $(vi)$ If $\alpha\in \Phi^+$ then $\beta\in \Phi^+$; $(vii)$ If $\beta\in \Phi^-$ then $\alpha\in \Phi^-$.
\end{cor}
\begin{proof} $(vi)$ is proved directly taking $w=e$ in Eq.~(\ref{eq:InftyDepth}). The proof of $(vii)$ follows from $(vi)$ together with Proposition~\ref{prop:DomFu}~$(i)$.
\end{proof}

\subsection{Dominance-depth and $m$-small roots} Brink and Howlett in \cite{BrHo93} introduced, in relation to the dominance order, another depth-statistic:  the {\em dominance-depth $\dep_\infty:\Phi^+\to\mathbb N$} is defined by
$$
\dep_\infty(\beta) = |\{\alpha\in \Phi^+\setminus\{\beta\}\mid \alpha\prec_\dom \beta\}|.
$$

The following result is an analog to Eq.~(\ref{eq:Depth}) for dominance-depth, see for instance \cite[Proposition 3.3]{DyHo16}:   for $\beta\in \Phi^+$ and $s\in S$, $s\not = s_\beta$, we have:
\begin{equation}\label{eq:DomDepth}
\dep_\infty(s(\beta))=
\left\{
\begin{array}{ll}
\dep_\infty(\beta)-1&\textrm{if }B(\alpha_s,\beta)\geq 1\\
\dep_\infty(\beta)&\textrm{if }B(\alpha_s,\beta)\in ]-1,1[\\
\dep_\infty(\beta)+1&\textrm{if }B(\alpha_s,\beta)\leq -1
\end{array}
\right.
\end{equation}
In particular, $\dep_\infty(\alpha_s)=0$ for all $s\in S$. 

\begin{remark} In~\cite[\S5.1]{DyHo16}, the authors introduced families of {\em depths}, {\em lengths}  and {\em weak orders} on positive root systems. In particular,  the following equivalent definition is provided for the dominance-depth: 
$$
\dep_\infty(\beta) = \frac{|\{\alpha\in \Phi(s_\beta)\mid B(\alpha,\beta)\geq 1\}|-1}{2}.
$$
\end{remark}

Let $m\in \mathbb N$. The set of {\em $m$-small roots} is the set of positive roots that dominate at most $m$ distinct proper positive roots:
$$
\Sigma_m=\{\beta\in \Phi^+\mid \dep_\infty(\beta)\leq m\}.
$$
Obviously $\Phi^+=\bigcup_{m\in \mathbb N} \Sigma_m$. The  $m$-small roots are defined in the introduction in relation with parallelism; this relationship is discussed  in~\S\ref{ss:Parallel}. The following theorem, due to Brink and Howlett~\cite{BrHo93} for $m=0$ and Fu~\cite{Fu12} for all $m$, implies that the sets of $m$-small roots provides a decomposition of the positive roots into finite sets whenever $S$ is finite. 

\begin{thm}[Brink-Howlett, Fu]\label{thm:BrFu} The set $\Sigma_m$ is finite whenever $S$ is finite for any $m\in \mathbb N$. 
\end{thm}

\subsection{$m$-small inversion sets and $m$-low elements}\label{ss:LowElements}  The {\em $m$-small inversion set of $w\in W$} is the set:
$$
\Sigma_m(w)=\Phi(w)\cap\Sigma_m.
$$
The {\em set $L_m$ of $m$-low elements} is, see \cite{DyHo16} for more details:
$$
L_m=\{w\in W\mid \Phi(w)=\cone_\Phi(\Sigma(w)) \}=\{w\in W\mid \Phi^1(w)\subseteq \Sigma_m \}.
$$

\begin{ex} 
\begin{enumerate}

\item If $W$ is finite, then $\Sigma_m=\Sigma_0=\Phi^+$ for all $m\in \mathbb N$. Hence $L_m=L_0=W$.

\item The set $L_0$ (resp. $L_1$) in affine type $\tilde B_2$ consists of the elements in the blue regions in Figure~\ref{fig:Aff1}~(a) (resp. Figure~\ref{fig:Aff2}~(a)). 

\item The set $L_0$ of a non-affine Coxeter arrangement consists of the elements in the blue regions in Figure~\ref{fig:Hyp1}.
\end{enumerate}
\end{ex}

If $S$ is finite, the set $\Sigma_m$ is finite (Brink and Howlett, Fu, see~Theorem~\ref{thm:BrFu}). Therefore the set $L_m$ is also finite. Actually, the set of $m$-low elements is a finite\footnote{If $S$ is finite.} Garside shadow, that is, $L_m$ contains $S$ and is closed under taking suffixes and under taking join in the right weak order. This result in the case $m=0$ was proven in \cite{DyHo16} and is the key to show the existence of a finite Garside family in the corresponding Artin-Tits monoid, see \cite{DDH14} for more information. The result for all $m$ is proven in the case of affine Weyl groups in \cite{DyHo16} and in full generality in \cite[Corollary~1.7]{Dy21}.

The key notion to prove that $L_m$ is a Garside shadow is {\em bipodality}: a set $A\subseteq \Phi^+$ is {\em bipodal} if for any $\beta\in A$ and maximal dihedral reflection subgroup $W'$ such that $\beta\in \Phi_{W'}\setminus \Delta_{W'}$ we have $\Delta_{W'}\subseteq A$; see \cite{DyHo16,Dy21} for more information. 

\begin{thm}[{\cite[Corollary~1.7]{Dy21}}]\label{thm:Dyer} Let $m\in \mathbb N$. The set $\Sigma_m$ is bipodal and the set $L_m$ is a Garside shadow.  
\end{thm}

The above theorem has the following useful consequence on short inversions graphs.

\begin{cor}\label{cor:PosetBasisInversion} Let $w\in W$, $\alpha,\beta\in \Phi^1(w)$ with $\alpha \preceq_w \beta$, then  $\dep_\infty(\alpha)\leq\dep_\infty(\beta)$.
\end{cor}
\begin{proof} Set $p=\dep_\infty(\beta)$. In particular, $\beta\in \Sigma_p$. It is enough to show that if $\alpha\precinv_w\beta$ then $\dep_\infty(\alpha)\leq\dep_\infty(\beta)$, by definition of $\preceq_w$.
Write $W'=\mathcal M_{\alpha,\beta}$, then by definition of $\precinv_w$, $\beta\in \Phi^+_{W'}\setminus\Delta_{W'}$ and $\alpha\in \Delta_{W'}$. Therefore $\alpha\in \Delta_{W'}\subseteq \Sigma_p$ since $\Sigma_p$ is bipodal by Theorem~\ref{thm:Dyer}. Hence  $\dep_\infty(\alpha)\leq p=\dep_\infty(\beta)$.
 \end{proof}

As a direct consequence of Theorem~\ref{thm:BasisInversion} and Corollary~\ref{cor:PosetBasisInversion}, we obtain the following theorem. 

\begin{thm}\label{thm:Main} Let $w\in W$ and set $d_w=\max\{\dep_\infty(\gamma)\mid \gamma\in \Phi^R(w)\}$.
\begin{enumerate}
\item The infinity-depth on $\Phi^1(w)$ reaches its maximum on $\Phi^R(w)$:  
$$
\dep_\infty(\beta)\leq d_w,\ \textrm{for all } \beta \in \Phi^1(w).
$$
\item  The element $w$ is a $d_w$-low element; 
\item For $m\in\mathbb N$, $w\in L_m$ if and only if $m\geq d_w$.
\end{enumerate}
\end{thm}

The following corollary is a proof of \cite[Conjecture~2]{DyHo16}, which is key to proving Theorem~\ref{thm:Main1}.

\begin{cor}\label{cor:Main} Let $m\in \mathbb N$. The map $\lambda_m:L_m\to \Lambda_m=\{\Sigma_m(w)\mid w\in W\}$, defined by $w\mapsto \Sigma_m(w)$, is a bijection. 
\end{cor}
\begin{proof} The map $\lambda_m$ is injective by \cite[Proposition~3.26~(b)]{DyHo16}. Now let $A\in \Lambda_m$. Choose $w\in W$ of minimal length such that $A=\Sigma_m(w)$. Let $r\in D_R(w)$, then by Proposition~\ref{prop:InvBasics} and the definition of small inversion sets we have:
\begin{eqnarray*}
\Sigma_m(w)&=&\Phi(w)\cap \Sigma_m = (\Phi(wr)\cap \Sigma_m) \sqcup (\{-w(\alpha_r) \}\cap \Sigma_m)\\
&=&\Sigma_m(wr)\sqcup (\{-w(\alpha_r) \}\cap \Sigma_m).
\end{eqnarray*}
Since $w$ is of minimal length, $\Sigma_m(wr)\not = A$, therefore $-w(\alpha_r)\in \Sigma_m$. In other words, $\Phi^R(w)\subseteq \Sigma_m$. We conclude that  $w\in L_m$ by Theorem~\ref{thm:Main}.
\end{proof}

The next  proposition is key to proving Theorem~\ref{thm:Main2} and Theorem~\ref{thm:Main3}.

\begin{prop}\label{cor:Main2} Let $m\in \mathbb N$,  $w\in L_m$ and $s\in S$. Then $sw\in L_{m+1}$. Moreover:
\begin{enumerate}
\item $sw\in L_{m+1}\setminus L_m$ if and only if $w<sw$ and  there is $r\in D_R(w)$ such that $\dep_\infty(-sw(\alpha_r))=m+1$.
\item If the conditions above hold, then  $\alpha_s\prec_\dom -sw(\alpha_r)$ for any $r\in D_R(w)$ with $\dep_\infty(-sw(\alpha_r))=m+1$.
 \end{enumerate}
\end{prop}
\begin{proof} In order to show that $sw\in L_{m+1}$, we only need to show that $d_{sw}\leq m+1$ by Theorem~\ref{thm:Main}. Observe also that $sw\in L_{m+1}\setminus L_m$ if and only if $d_{sw}=m+1$.  Assume first that $sw< w$, then $sw$ is a suffix of $w$. Therefore $sw \in L_m$, since $L_m$ is a Garside shadow. In particular, $d_{sw}=m<m+1$ so $(1)-(2)$ cannot happen and $sw\in L_{m+1}$ .

We assume from now on that $sw> w$. Let $r\in D_R(sw)$. Then by the exchange condition we have two cases.
\begin{enumerate}[$(i)$]
\item  $sw=wr$: in this case  we have $s=wrw^{-1}$ and so $w(\alpha_r)=\alpha_s$. Hence $\dep_\infty(-sw(\alpha_r))=\dep_\infty(\alpha_s)=0<m+1$. 

\item  $r\in D_R(w)$: in this case  we have $-w(\alpha_r)\in \Phi^R(w)$. By Theorem~\ref{thm:Main}, we have  $\dep_\infty(-w(\alpha_r))\leq m$ since $w\in L_m$. Therefore  $\dep_\infty(-sw(\alpha_r))= \dep_\infty(-w(\alpha_r))+ 1\leq m+1$ and $B(\alpha_s,-w(\alpha_r))\leq -1$ by Eq.~(\ref{eq:DomDepth}).
\end{enumerate}
 In both cases, we have $d_{sw}\leq m+1$ so $sw\in L_{m+1}$.

\smallskip
\noindent We now prove $(1)$. Assume first that $sw\notin L_m$, we have $\Phi^R(sw)\not\subset\Sigma_m$. Therefore there is $r\in D_R(sw)$ such that $\dep_\infty(-sw(\alpha_r))\geq m+1$. In view of $(i)$ and $(ii)$ above, we deduce that $r\in D_R(w)$ and $\dep_\infty(-sw(\alpha_r))= m+1$. Conversely, assume there is $r\in D_R(w)$ such that $\dep_\infty(-sw(\alpha_r))=m+1$. We already know that $sw\in L_{m+1}$. Since $D_R(w)\subseteq D_R(sw)$, we have $d_{sw}=m+1$ since $-sw(\alpha_r)\in \Phi^R(sw)$. Therefore $sw\in L_{m+1}\setminus L_m$.

\smallskip
\noindent Finally, for $(2)$: let $r\in D_R(w)$ with $\dep_\infty(-sw(\alpha_r))=m+1$.  By $(ii)$ above, we have the inequality $B(\alpha_s,-w(\alpha_r))\leq -1$, which is equivalent to $B(\alpha_s,-sw(\alpha_r))\geq 1$. Therefore we have $\alpha_s\prec_\dom -sw(\alpha_r)$ by Proposition~\ref{prop:DomFu}~{\it (ii)}.
\end{proof}

\section{Chambered sets for Coxeter systems}\label{se:CS}

Let $(W,S)$ be a Coxeter system.   In order to provide the definition of $m$-Shi arrangement ($m\in \mathbb N$), and then proofs for Theorem~\ref{thm:Main1} and Theorem~\ref{thm:Main2}, we need first some space or set $\mathcal U(W,S)$  with a natural action of $W$ as a reflection group, in which the Coxeter arrangement $\mathcal A(W,S)$ is realized and the regions of $\mathcal A(W,S)$ are in bijection with the elements of $W$. 

Examples of possible $\mathcal U(W,S)$ include:  the {\em Tits cone} (see~\cite[\S5.13]{Hu90});  the {\em Coxeter complex} or the {\em Davis complex} (see~\cite[\S3 and \S12.3]{AbBr08}); an affine Euclidean space for finite or affine Coxeter groups (see~\cite[\S10]{AbBr08}), a hyperbolic model for weakly hyperbolic Coxeter systems (see for instance~\cite{AbBr08,HPR20}) or  the {\em imaginary cone} for finit rank, irreducible, indefinite Coxeter systems~\cite{Dy19-1,DyHoRi16}. Other examples involve suitable ``reflection'' actions by $W$  on more general manifolds, or on  the Cayley graph (see  \cite{Da08}) and even the left regular action of  $W$ on itself (by left translation). 

To deal with these and other situations uniformly,  we abstract  the situation  to that of a pair $(\mathcal U,C)$ called a {\em chambered set for $(W,S)$}.  A reader familiar with one of the examples above may proceed directly to \S\ref{se:LowShi} knowing the terminology used for chambered sets (chambers, hyperplanes and walls, etc.) is compatible with the classical ones. 

\subsection{Chambered sets}\label{ss:CS} A pair $(\mathcal U, C)$ is a {\em chambered set\footnote{The notion of chambered set was introduced by the first author, who will study it and variants more systematically in a forthcoming publication.} for $(W,S)$}, or a {\em chambered $W$-set} for short, if the following conditions are satisfied: 
\begin{enumerate}[(i)]
\item $\mathcal U$ is a left $W$-set; 
\item $C$ is  a set of orbit representative for $W$ on $\mathcal U$;
\item for each $a\in C$, the stabilizer of $a$ in $W$ is a standard parabolic subgroup:  there is~$J_{a}\subseteq S$ such that
$
W_{J_a}=\{w\in W\mid wa=a\}.
$ 
\end{enumerate}
The set $C$ is called the {\em fundamental chamber}, and the subsets $C_{w}:=wC$ of $\mathcal U$ are called  {\em chambers}.  For $X\subseteq W$, write $XC=\bigcup_{w\in X}C_w$. Then:
 $$
 \mathcal U=\bigcup_{w\in W}C_{w}=WC.
 $$ 
 A subset $J$ of $S$ is called {\em facetal} if there is some $a\in C$ with stabilizer $W_{J}$, i.e., such that $J_{a}=J$.  A chambered $W$-set $(\mathcal U,C)$  is said to be:
 \begin{enumerate}
 \item[(iv)]  {\em solid} if  $\emptyset$ is a facetal  subset of $S$ i.e. $C$ contains a point $a_0$ with trivial stabilizer.  
 \item[(v)] {\em walled} if for each $s\in S$, $\{s\}$ is a facetal subset of $S$:  there is some $a_{s}\in C$ with stabilizer $W_{\{s\}}=\langle s\rangle=\{e,s\}$. 
 \end{enumerate}
 For any facetal subset $J$ of $S$, the set $C^{(J)}:=\{x\in C\mid J_{x}=J\}$ is called {\em a facet  of $C$}.  The facets of $C$ form a partition of $C$. The {\em facets of $\mathcal U$} are by definition  the subsets $wF$ for $w\in W$ and $F$ a facet of $C$. The  facets of $\mathcal U$ form a partition of $\mathcal U$.

We impose some non-degeneracy conditions subsequently (such as walled and solid) but for now allow the possibility that, for instance,  $\mathcal U=C=\emptyset$.

\begin{prop}\label{prop:CS1} 
\begin{enumerate}
\item Let $x,y\in W$ and $s_{1}\cdots s_{n}$ be a reduced word for $x^{-1}y$. Let $J=\{s_{1},\ldots, s_{n}\}$. Then $C_{x}\cap C_{y}=\bigcap_{u\in W_{J}}C_{xu}$ and $xW_{J}x^{-1}$ stabilizes $C_{x}\cap C_{y}$ pointwise. 
 \item Suppose $(\mathcal U,C)$ is solid. Then for $X,Y\subseteq W$, one has $XC\subseteq YC\iff X\subseteq Y$.
 In particular, the $W$-action on the set of chambers is simply transitive.
\end{enumerate}
  \end{prop}

 \begin{proof}  (1) Translating by $x^{-1}$, we may and do assume that  $x=e$.
 It will suffice to show that $W_{J}$ stabilizes $C\cap C_{y}$ pointwise  for $u\in W_{J}$. For then, if $a\in C\cap C_{y}$ and $u\in W_{J}$, one would have $a=ua\in uC=C_{u}$.
  Let $a\in C\cap C_{y}$, say $a=yb$ where $b\in C$. The definitions give $a=b$ and $y\in W_{J_{a}}$. Since $s_{1}\cdots s_{n}$ is a  reduced expression for $x^{-1}y=y$, it follows that
 $s_{i}\in J_{a}$ for $i=1,\ldots, n$ and thus $J\subseteq J_{a}$.
 Thus, $W_{J}$ stabilizes $a$ as required.
 
 \medskip
 
 \noindent (2)  The implication $\Longleftarrow$ is trivial. 
 Conversely, suppose $XC\subseteq YC$. Let $x\in X$. Then $xa_0\in XC\subseteq YC$, so $xa_0=ya$ for some $a\in C$. Then $y^{-1}xa_0=a$. Therefore $a=a_0$, by Condition~(ii) of the definition of chambered $W$-sets. Therefore $y^{-1}x\in W_{J_{a_0}}=\{e\}$, which implies $x=y\in Y$ as required.
 \end{proof}

We discuss now many of the classical examples mentioned in the introduction of this section but regarded  as chambered $W$-sets, with their standard $W$-action and fundamental chambers.

 \begin{ex}[Tits cone] The Tits cone $\mathcal U$ of the standard geometric representation of a Coxeter group (with linearly independent  simple roots), in the dual space of the linear span of the roots (see e.g. \cite{Hu90}) has a natural structure of chambered $W$-set  with the standard fundamental chamber $C$. Then  every subset of $S$ is facetal. For more general classes  of reflection representation (see for example \cite{Dy19-1}, \cite{Dy94-1}, \cite{Fu18}) allowing linearly dependent simple roots, the set of facetal subsets depends on the facial structure of  the cone spanned by simple roots, and on the ambient quadratic space (or an associated bilinear form). 
\end{ex}
 
\begin{ex} For any chambered $W$-set $(\mathcal U,C)$ and $W$-invariant subset $U_{0}$ of $\mathcal U$, one has a chambered $W$-set $(U_{0},U_{0}\cap C)$. The interior of the Tit's cone
 (for $(W,S)$ of finite rank with reflection representation in a finite-dimensional real vector space) is a natural example.
  \end{ex}
 
  \begin{ex}  For any $J\subseteq S$, one has a chambered $W$-set 
 $W/W_{J},\{1W_{J}\})$ where $W$ acts by left translation on the coset space $W/W_{J}$ and $1=1_{W}$
 The only facetal subset  of $S$ is $J$.     Every chambered $W$-set is a disjoint union (more precisely, coproduct  in a suitable category of chambered $W$-sets) of ones of this type.       
 The regular  chambered $W$-set arises in case $J=\emptyset$ and was mentioned already  above.
  \end{ex}
 
 \begin{ex}[Coxeter complex] As for the Tits cone, the  Coxeter complex may be defined
  in several ways. For finite rank $(W,S)$, it has a realization  as geometric simplicial complex \cite{AbBr08}; the corresponding topological space 
    has an underlying natural structures of chambered $W$-set in which the 
   facetal subsets are the proper subsets of $S$ ($S$ itself corresponds to the empty-simplex). The face poset (including the empty simplex) also has a structure of chambered $W$-set, with all subsets of $S$ facetal. This  extends naturally to arbitrary Coxeter systems (of possibly infinite rank)
   \cite{Ti74}, where the underlying poset is called a chamber complex. 
   
    There are also various other  realizations of  the Coxeter complex as a ''chamber complex,''  which have an underlying $W$-set (perhaps with additional structure) of which certain subsets are identified as chambers. These also  have underlying chambered $W$-sets for which the chambers $C_{w}$  are the chambers for the chamber complex.    In the realization 
    \cite[Ch IV, \S1, Exercise 16]{Bo68}, the facetal subsets are the sets $S\setminus \{s\}$ for $s\in S$ (note in this case, one has $C\subseteq \bigcup_{s\in S}C_{s}$ even though the $W$-action on chambers is simply transitive). In  the realization of the Coxeter complex as a $W$-metric space~\cite{AbBr08} the underlying $W$-set is just the regular left $W$-set, but there is an additional structure of $W$-valued metric; in the corresponding chambered $W$-set, the only facetal subset is $\emptyset$.   
\end{ex}

\begin{ex}[Discrete reflection groups] For a finite Coxeter group acting as Euclidean reflection group on a Euclidean space, the Euclidean space has a structure of chambered $W$-set  with the usual fundamental chamber \cite{Hu90}; then every subset of $S$ is facetal. For an irreducible affine Weyl group, realized as irreducible  affine (and discrete) reflection group \cite{Hu90}, \cite{Bo68},  the affine space on which it acts may be regarded as a chambered $W$-set with  a chosen fundamental alcove as fundamental chamber. The facetal subsets are then the proper subsets of $S$. A similar discussion applies to irreducible compact hyperbolic simplex reflection groups (\cite{Hu90},\cite{Da08}); for discrete reflection groups in a hyperbolic space (\cite{Vi71,Vi93}), the facetal subsets depend on the facial structure of the fundamental chamber, see for more details \cite{DyHoRi16,HPR20}. 
 \end{ex}
 
\begin{ex}[Cayley graph of $(W,S)$] The Cayley graph may be realized as an abstract graph, or  as a one-dimensional (abstract or geometric) simplicial  complex or regular CW-complex \cite{Da08}.  It has a    natural structure as chambered $W$-sets in which the facetal subsets are $\emptyset$ and the singleton subsets of $S$. 
 \end{ex}

\begin{ex}[Davis complex] The Davis complex (see for instance \cite{Da08}) may be regarded as a certain  convex cell complex.  Its poset of non-empty faces identifies with the inclusion ordered subset of all   spherical subsets\footnote{A subset $J$ of $S$ is said to be spherical if $W_{J}$ is finite} of $S$.   It and its face poset (of non-empty faces) have natural structures of chambered $W$-set  with the spherical subsets of $S$ as the facetal subsets.
\end{ex}

\begin{ex}[Imaginary cone] The imaginary cone in the sense of \cite{Dy19-1} has a natural structure of chambered $W$-set, with a certain fundamental domain as fundamental chamber. The fundamental domain is the intersection of the cone spanned by simple roots (in a suitable  realization of $W$ as Coxeter group associated to a suitable root system) with the negative of the fundamental chamber for the corresponding Tits cone. The description of the facetal subsets can be obtained with some effort from results in \cite{Dy19}, but is not simple, especially for $(W,S)$ of infinite rank, and will not be given here. We  make only the following remarks for certain $W$, assumed irreducible for simplicity.  If $W$ is is finite, affine or locally finite (that is, every finite rank standard parabolic subgroup is finite), then $W$ acts trivially on the imaginary cone and the only facetal subset is $S$,
 while if $W$ is finite rank, infinite and  non-affine, the facetal subsets include all spherical subsets of $S$.
 \end{ex}

\subsection{Walls and reflection hyperplanes} From now on fix a chambered $W$-set $(\mathcal U,C)$ and fix also a root system $\Phi$ for $(W,S)$. Recall that the set of reflections of $(W,S)$ is 
$$
T=\{ s_\alpha\mid \beta\in \Phi\}=\{ s_\alpha\mid \beta\in \Phi^+\}.
$$

For a root $\alpha\in \Phi$, we define 
 $X_{\alpha}:=\{w\in W\mid w^{-1}(\alpha)\in \Phi^{+}\}.$ Notice that $W=X_\alpha\sqcup X_{-\alpha}$.   The  {\em closed half-spaces associated to  $\alpha$} are:
$$
H_{\alpha}^{+}:=
X_{\alpha}C \quad\textrm{and}\quad
H_{\alpha}^{-}:=H_{-\alpha}^{+}.
$$
 The {\em hyperplane\footnote{In many situations, the standard terminology   refers to  walls instead of  hyperplanes.} associated to $\alpha$} is defined by:
 $$
H_{\alpha}=H_{-\alpha}:= H_{\alpha}^{+}\cap H_{\alpha}^{-}.
 $$

The result below may be significantly strengthened but for simplicity we prove only the minimum needed in this article. 

\begin{prop}\label{prop:CS2} Let $\alpha\in \Phi$. 
\begin{enumerate}
\item   For all $w\in W$, we have $wH_{\alpha}^{+}=H_{w(\alpha)}^{+}$,  $wH_{\alpha}^{-}=H_{w(\alpha)}^{-}$ and $w H_{\alpha} =H_{w(\alpha)}$.
\item If $\alpha\in \Phi^+$ then:
$$
H_{\alpha}^{+}=\bigcup_{s_\alpha w >w} C_w\quad \textrm{ and } \quad H_{\alpha}^{-}=\bigcup_{s_\alpha w <w} C_w. 
$$
In particular, the fundamental chamber $C$ is contained in  $H_\alpha^+$ if $\alpha\in \Phi^+$. 
\item We have:
$$
H_{\alpha}=\{a\in \mathcal U\mid s_\alpha a=a\}=\bigcup_{w\in W}(C_{w}\cap C_{s_\alpha w}).
$$ 
\end{enumerate}
In particular,   $H_\alpha$ is the set of fixed points for the action of the reflection $s_\alpha$ on $\mathcal U$.
 \end{prop}

 \begin{proof}   (1) follows easily from the definitions. (2) follows from the well-known fact that if $\alpha\in \Phi^+$ then  $s_\alpha w>w$ if and only if $w^{-1}(\alpha) \in \Phi^+$. 
 
 \smallskip
 \noindent (3) Without loss of generality, we may assume $\alpha\in \Phi^+$. Write the equalities to be proven  as $H_\alpha=H_{\alpha}'=H_{\alpha}''$.  First note that by (2) we have:
 $$
 H_{\alpha}''\subseteq (\bigcup_{s_\alpha w >w}C_{w})\cap (\bigcup_{s_\alpha w <w}C_{w})=H_{\alpha}^{+}\cap H_{\alpha}^{-}=H_{\alpha} .
 $$
 We now prove $H_{\alpha}'\subseteq H_{\alpha}''$. Suppose $a\in H_{\alpha}'$, i.e.,  $a\in \mathcal U$ with $s_\alpha a=a$. Let $w\in W$ such that $a\in wC=C_w$.  Then $a=s_\alpha a\in s_\alpha wC$. So $a\in C_{w}\cap C_{s_\alpha w}\subseteq H_{\alpha}''$.  Thus,  $H_{\alpha}'\subseteq H_{\alpha}''$.  Therefore $H_{\alpha}'\subseteq H_{\alpha}''\subseteq H_\alpha$.
 
 Finally, we show $H_{\alpha}\subseteq H_{\alpha}'$.  Let $a\in H_{\alpha}$. By definition of $H_\alpha$ and (2) above, there are $x,y\in W$ such that  $a\in C_{x}$ where $s_\alpha x>x$ and $a\in C_{y}$ where $s_\alpha y<y$.   Choose a reduced expression $s_{1} \ldots s_{n}$ for  $x^{-1}y$. Let $J=\{s_,\dots , s_n\}$. By Proposition~\ref{prop:CS1}, $xW_{J}x^{-1}$ stabilizes $a$  and  we have 
 $a\in C_{xu_{i}}$ for $0=1,\ldots, n$ where $u_{i}:=s_{1}\cdots s_{i}\in W_J$. Ssince $s_\alpha x>x$ and $s_\alpha y<y$, there is some $i$ with $1\leq i\leq n$ such that $s_\alpha xu_{i-1}>xu_{i-1}$ but $s_\alpha xu_{i}<xu_{i}$. This implies that $s_\alpha =xu_{i-1}s_{i}u_{i-1}^{-1}x^{-1}=xu_iu_{i-1}^{-1}x^{-1}\in xW_{J}x^{-1}$ and $s_\alpha a=a$. Hence $a\in H_{\alpha}'$ and $H_{\alpha}\subseteq H_{\alpha}'$, concluding the proof.
 \end{proof}
 
 Nothing guarantees  that the set of hyperplanes $H_\alpha$ is in bijection with the set of reflections $T$ (or $\Phi^+$). The next proposition shows that this happens, for instance, in the case of  walled chambered sets (see Condition~(v) in \S\ref{ss:CS}).
 
 \begin{prop}\label{prop:CS3} Assume $(\mathcal U,C)$ is walled. Then for all $\alpha,\beta\in \Phi^+$:
  $H_\alpha=H_\beta$ if and only if $\alpha=\beta$.
\end{prop}
\begin{proof} The implication $\Longleftarrow$ is trivial. Let $\alpha,\beta\in \Phi^+$ such that $H_\alpha=H_\beta$. Without loss of generality, we assume that $\alpha\in \Delta$, i.e., $s:=s_\alpha\in S$. By proposition~\ref{prop:CS2}, we have:
$$
H_\alpha= \{a\in \mathcal U\mid sa=a\} = \{a\in \mathcal U\mid s_\beta a=a\} =H_\beta.
$$
Thanks to Condition~$(v)$ ($(\mathcal U,C)$ is walled), there is $a_s\in C$ such that $sa_s=a_s=s_\beta a_s$ and the stabilizer of $a_s$ in $W$ is $\{e,s\}$.  By Condition~(iii) of the definition of chambered $W$-set we have  $s_\beta \in \{e,s\}$, which implies $s=s_\beta$. Therefore $\alpha=\beta$, since they are both positive roots. 
\end{proof}

\subsection{Convex subsets in a chambered set} Let $(\mathcal U,C)$ be a solid chambered $W$-set (see Condition~(iv)). 

A subset $D$ of $\mathcal U$ is {\em convex in $\mathcal U$} if there is a subset  $\Gamma\subseteq\Phi$ such that
$$
D=\bigcap_{\alpha\in \Gamma}H_{\alpha}^{\epsilon_\alpha},
$$ 
where $\epsilon_\alpha\in \{ \pm\}$. We allow $\Gamma=\emptyset$, so $U$ itself is convex. A convex subset of $\mathcal U$ is a union of facets of $\mathcal U$.  We shall be  particularly concerned with convex subsets  which are unions of chambers.\footnote{It can be shown that if  a convex subset of $U$ contains a chamber, it is a union of chambers.}   Such a convex set $D$ is of the form $D=XC$ for a uniquely determined (by Proposition \ref{prop:CS1}(2))  subset $X$ of $W$. The subsets $X$ of $W$ arising in  this way from convex subsets $D$ of $\mathcal U$ are what are usually called convex subsets of $W$ (see \cite{Ti74}, \cite{Kr09}).  They are exactly the convex subsets of the left regular  chambered $W$-set $(W,\{e\})$. In particular, they are independent of the choice of solid  chambered $W$-set  $(U,C)$. 

The {\em convex closure $\conv(A)$} of a  subset $A$ of $\mathcal U$ is the 
intersection of all convex subsets of $\mathcal U$ which contain $A$. It is the inclusion-minimum convex subset of $\mathcal U$ which contains $A$. 

\subsection{Infinity-depth and parallelism in chambered $W$-sets}\label{ss:Parallel} Let $(\mathcal U,C)$ be a solid chambered $W$-set.  Let $\alpha,\beta\in \Phi^+$ be distinct. It is well-known that the dihedral reflection subgroup $W'$ generated by $s_\alpha$ and $s_\beta$ is infinite if and only if $|B(\alpha,\beta)|\geq 1$, see for instance \cite[Proposition 1.5]{HoLaRi14}. We also know that in this case:
\begin{itemize}
\item $\Delta_{W'}=\{\alpha,\beta\}$ if $B(\alpha,\beta)\leq -1$;
\item either $\alpha$ or $\beta$ is in $\Delta_{W'}$ if $B(\alpha,\beta)\geq 1$. Then, by Proposition~\ref{prop:DiDep}, we have either $\dep(\alpha)< \dep(\beta)$ or $\dep(\alpha)> \dep(\beta)$. 
\end{itemize}
We say that {\em $H_\alpha$ and $H_\beta$ are parallel} if $|B(\alpha,\beta)|\geq 1$.  If $ |B(\alpha,\beta)|> 1$, $H_\alpha$ and $H_\beta$ are often said to be {\em ultra-parallel}. 
The following  proposition gives an interpretation of the dominance order in $(\mathcal U,C)$. 

\begin{prop}\label{prop:DomHyp} Assume $(\mathcal U,C)$ is solid.  Let $\alpha,\beta\in \Phi^+$. Then:
\begin{enumerate}
\item  $B(\alpha,\beta)\geq 1$ if and only if $H_\alpha^+\subseteq H_\beta^+$ or $H_\beta^+\subseteq H_\alpha^+$;
\item $\alpha\preceq_\dom \beta$ if and only if $H_\alpha^+\subseteq H_\beta^+$;
\end{enumerate}
\end{prop}
\begin{proof}  By definition of the dominance order (Eq.~(\ref{eq:InftyDepth})) and Proposition~\ref{prop:CS1}(2),  we get  
  \[\alpha\preceq_\dom \beta\iff X_{\alpha}\subseteq X_{\beta}  \iff X_{\alpha}C\subseteq X_{\beta}C\iff 
  H_{\alpha}^{+}\subseteq H_{\beta}^{+}.\] This proves (2). In particular, $H_\alpha^+\subseteq H_\beta^+$ implies $B(\alpha,\beta)\geq 1$ by Proposition~\ref{prop:DomFu}~$(v)$, which proves (1) since the role of $\alpha$ and $\beta$ are interchangeable.  
\end{proof}

Let $\alpha,\beta\in \Phi^+$. We say that a hyperplane {\em $H_\alpha$ separates $H_\beta$ from $C$} if  $H_\alpha^+\subset H_\beta^+$. Note that, in this case, $H_\alpha$ and $H_\beta$ are parallel by Proposition~\ref{prop:DomHyp}. We obtain therefore an useful way of computing the infinity-depth: for $\beta\in \Phi^+$ we have
\begin{equation}\label{eq:DepH}
\dep_\infty(\beta)=|\{\alpha\in \Phi^+ \mid H_\alpha \textrm{ separates }C\textrm{ and }H_\beta\} |.
\end{equation}
For instance, in Figure~\ref{fig:Hyp2}(a), the black hyperplane (geodesic) bounding $23$ has infinite-depth equal to $2$, since it is separated from $C=e$ by the two parallel geodesics.

For an affine Weyl group, two roots are comparable in dominance order if and only if the corresponding reflecting hyperplanes (in the realization of the group as an affine reflection groups) are parallel (or equal).  Since parallelism is an equivalence relation on affine hyperplanes, the dominance order is a disjoint union of chains; see \cite{DyHo16} for an explicit description. The next corollary follows immediately.  

\begin{cor}\label{cor:Para} Assume $(W,S)$ is an affine Coxeter system and $\mathcal U$ is Euclidean space. Let $\alpha,\beta,\gamma\in \Phi^+$ such that $\alpha\not = \beta$, $\alpha\prec_\dom \gamma$ and $\beta\prec_\dom \gamma$. Then either $\alpha\prec_\dom \beta\prec_\dom  \gamma$ or $\beta\prec_\dom \alpha\prec_\dom  \gamma$. 
\end{cor}

\subsection{Hyperplane arrangements in a  chambered $W$-set}\label{ss:HypCS} Let $(\mathcal U,C)$ be a solid and walled chambered $W$-set. 

 On  one hand, being solid implies, by Proposition~\ref{prop:CS1}, that the $W$-action on the set of chambers of $(\mathcal U,C)$ is simply transitive. In particular $W$ is in bijection with the set of chambers of $(\mathcal U,C)$.  On the other hand, being walled  implies, by Proposition~\ref{prop:CS3},  that the set of hyperplanes $H_\alpha$ is in bijection with the set of reflections $T$ (or $\Phi^+$).

\begin{defi} The {\em Coxeter arrangement of $(W,S)$ in $(\mathcal U,C)$} is the collection of hyperplanes 
$$
\mathcal A:= \mathcal A(W,S)=\{H_\alpha\mid \alpha \in \Phi^+\}.
$$
More generally, for any $\Gamma\subseteq \Phi^+$, we consider the {\em $\Gamma$-arrangement of hyperplanes}  $\mathcal A_\Gamma:=\{H_\alpha\mid \alpha \in \Gamma\}$; notice that $\mathcal A_{\Phi^+}=\mathcal A$. 
\end{defi}
By Proposition~\ref{prop:CS3},  $\mathcal A_\Gamma$  is in bijection with $\Gamma\subseteq \Phi^+$. In particular,  the Coxeter arrangement $\mathcal A$  is in bijection with $ \Phi^+$ (and therefore the set of reflections~$T$). 

\begin{ex} A Coxeter arrangement in the affine case of type $\tilde B_2$, realized as an affine hyperplane arrangement in the affine Euclidean space $\mathbb R^2$, is given in Figures~\ref{fig:Aff1}  and  \ref{fig:Aff2}. A non-affine Coxeter arrangement, realized in the hyperbolic plane $\mathbb H^2$, is given in Figure~\ref{fig:Hyp1}.
\end{ex}

Let $\Gamma\subseteq \Phi^+$. The arrangement $\mathcal{A}_{\Gamma}$ determines certain subsets of $\mathcal U$ as follows. First, the open regions of $\mathcal{A}_{\Gamma}$ are defined to be the subsets of $\mathcal U$ which arise as  a non-empty intersection of the form $\bigcap_{\alpha\in \Gamma} (H_{\alpha}^{\epsilon_{\alpha}}\setminus H_{\alpha})$ where for each $\alpha\in \Gamma$, $\epsilon_{\alpha}\in \{\pm\}$. The closed regions of $\mathcal A_\Gamma$ are the non-empty intersections $\bigcap_{\alpha\in \Gamma} H_{\alpha}^{\epsilon_{\alpha}}$.
In many (but not all)  of the examples in \S\ref{ss:CS} in which $(W,S)$ is of finite rank and  $\mathcal U$ has a natural structure of topological space,  $\{H_{\alpha}\mid \alpha\in \Gamma\}$   is a locally  finite\footnote{This is automatic if $\Gamma$ is finite.} set of closed subspaces of $\mathcal U$, the open regions are the (open) connected components of $\mathcal U\setminus \bigcup_{\alpha\in \Gamma} H_{\alpha}$, and the closed regions are the topological closures of the open regions.     

In general, both open and closed regions are unions of facets. Henceforward, by a {\em closed region for the arrangement  $\mathcal A_\Gamma$}, we shall always mean a closed region which is a union of  chambers. Such closed regions are the  (convex) subsets $XC$ of $\mathcal U$  such that $X$ is a non-empty (convex) subset $X= \bigcap_{\alpha\in \Gamma}X_{\epsilon_\alpha \alpha}$ of $W$. The subsets $X$ of $W$ such that $XW$ is a closed region for  $\mathcal A_\Gamma$ form  a partition of $W$. In particular, there is an equivalence relation $\sim_{\Gamma}$ of $W$ such that $x\sim_{\Gamma}y$ means that $C_{x}$ and $C_{y}$ are contained in the same closed region for  $\mathcal A_\Gamma$. For the Coxeter arrangement,  the closed regions are the chambers.

\section{Extended Shi arrangements and low elements}\label{se:LowShi}  

Let $(W,S)$ be a Coxeter system and $m\in \mathbb N$.    Let $(\mathcal U,C)$ be a solid and walled chambered $W$-set. As mentioned in the beginning of \S\ref{se:CS}, the reader may think of $\mathcal U$ as the Tits cone, the Coxeter complex, or any other well-known example of a  space or set $\mathcal U$  with a natural action of $W$ as a reflection group, in which the Coxeter arrangement $\mathcal A(W,S)$ is realized and the regions of $\mathcal A(W,S)$ are in bijection with the elements of $W$. 

We remark that the results we give below on $m$-Shi arrangements  are essentially combinatorial in nature, and in particular, they are independent of the choice of  a solid and walled chambered $W$-set. In fact, the result can be extended to solid (but not necessarily walled) chambered sets by reformulating them in terms of arrangements of pairs of opposite closed half-spaces $(\{H_\alpha^+,H_\alpha^- \})_{\alpha \in \Gamma}$ instead of arrangements $(H_\alpha)_{\alpha\in \Gamma}$ of hyperplanes.

\smallskip

In this section we first introduce extended Shi arrangements and prove Theorem~\ref{thm:Main1}. Then we prove Theorem~\ref{thm:Main2}, settling~\cite[Conjecture~3]{HoNaWi16}. Finally,  we obtain as a byproduct a direct proof of Thiel's Theorem~\ref{thm:Main3}, that is, if $(W,S)$ is an affine Coxeter system with underlying finite Weyl group $W_0$,  the union of the chambers $C_{w^{-1}}$ for $w\in L_m$ is a convex subset of  the Euclidean space. 
Then we  provide in a counterexample of the convexity of the inverses of $L_m$ if $m>0$ and $(W,S)$ is neither affine nor finite (i.e. indefinite) and we end by discussing enumeration of $m$-low elements (and therefore of $m$-Shi regions).

\subsection{Extended Shi arrangements and proof of Theorem~\ref{thm:Main1}} Let $m\in \mathbb N$. The {\em (extended) $m$-Shi arrangement}  $\shi_m(W,S)$ in $\mathcal U$ is the arrangement of {\em $m$-small hyperplanes}:
$$
\shi_m(W,S)=\{H_\beta\mid \beta \in \Sigma_m\},
$$
which, by Eq.~(\ref{eq:DepH}), consists of the hyperplanes in $\mathcal A$ that are separated from $C$ by at most $m$ parallel hyperplanes.

The closed regions in $(U,C)$ for $\shi_m(W,S)$ are called the  {\em $m$-Shi regions}. The corresponding equivalence relation $\sim_{\Sigma_m}$ on W  (see \S\ref{ss:HypCS})  is abbreviated  $\sim_m$  in this case. We have:
$$
u\sim_m v \iff C_u\textrm{ and } C_v\textrm{ are contained in the same $m$-Shi region.}
$$

\begin{ex} See Figures~\ref{fig:Aff1}~(a), \ref{fig:Aff2}~(a) and \ref{fig:Hyp1}~(a) where the blue chambers correspond to the  $m$-low elements and are the unique minimal chamber of their corresponding $m$-Shi region. For $m=0$, observe that the small hyperplanes (thick blue lines) do not have any other hyperplanes between them and $C$.  In Figure~\ref{fig:Aff2}, the $1$-small hyperplanes are constituted of the small hyperplanes plus hyperplanes that have exactly one hyperplane between them and $C$. 
\end{ex}

\begin{prop}\label{prop:ShiEq} Let $m\in\mathbb N$ and $u,v\in W$, then: 
$$
u\sim_m v\iff \Sigma_m(u)=\Sigma_m(v).
$$ 
In other words, two chambers $C_u$ and $C_u$ are in the same $m$-Shi region if and only if  $u$ and $v$ have the same $m$-small inversion set.
\end{prop}
\begin{proof} By definition $u\sim_m v$ if and only for each $\beta\in \Sigma_m$ we have $C_u\subseteq H_\beta^+$ if and only if $C_v\subseteq H_\beta^+ $. So for all $\beta \in \Sigma_m$, $H_\beta$ separates $C_u$ from $C$ if and only if $H_\beta$ separates $C_v$ from $C$. In other words, for all $\beta \in \Sigma_m$, $\beta\in \Phi(u)$ if and only if $\beta\in \Phi(v)$. The statement follows. 
\end{proof}

\begin{remark} In affine Weyl group and in the case $m=0$, the map $w\mapsto \Sigma_0(w)$ from $W$ to $\Lambda_0$ is the generalization of Shi's admissible sign type map~\cite{Shi88}; see also~\cite{ChHo22}. 
\end{remark}

The following theorem proves in particular Theorem~\ref{thm:Main1}.

\begin{thm}\label{thm:MainA} Let $m\in\mathbb N$.  For any $w\in W$, there is a unique $m$-low element $u\in L_m$  such that $u\sim_m w$. Moreover $u\leq_R w$. 
In particular, each region of  $\shi_m(W,S)$ contains a unique element of minimal length, which is a low element. 
\end{thm}
\begin{proof} Let $w\in W$. By Corollary~\ref{cor:Main}, there is a unique low element $u\in L_m$ such that $\Sigma_m(u)=\Sigma_m(w)$. By Proposition~\ref{prop:ShiEq}, $u\sim_m w$. Moreover, by definition of the $m$-low elements, we have:
$$
\Phi(u)=\cone_\Phi(\Sigma_m(u))=\cone_\Phi(\Sigma_m(w))\subseteq \cone_\Phi(\Phi(w))=\Phi(w).
$$
By Proposition~\ref{prop:InvBasics} we conclude that $u\leq_R w$ and so $u$ is the unique minimal element in the $m$-Shi region containing $C_w$. 
\end{proof}

\begin{remark} In the terminology of \cite{PaYa21}, the above theorem means that any $m$-Shi arrangement is {\em gated} and that $L_m$ is the set of {\em gates} of $\shi_m(W,S)$.
\end{remark}

\subsection{$m$-Shi polyhedron} We now set up the proof of Theorem~\ref{thm:Main2} and Theorem~\ref{thm:Main3}. Let $m\in \mathbb N$. Consider the set 
$$
\mathcal B_m=\{x^{-1}(\alpha_s)\mid x\in L_m,\ s\in S,\ sx\notin L_m \}.
$$ 
Since the set $L_m$ is a Garside shadow (Theorem~\ref{thm:Dyer}), it is stable under taking suffixes, so $s\in S\setminus D_L(x)$ in the definition above. Observe also that  $\mathcal B_m\subseteq \Phi^+$ and that $|\mathcal B_m|\leq |L_m|\times |S|$.  Therefore the set $\mathcal B_m$ is finite (provided that $S$ is finite).

\begin{defi} We define the {\em $m$-Shi polyhedron} to be the convex set:
$$
\mathscr S_m=\bigcap_{\beta\in \mathcal B_m} H_\beta^+.
$$ 
\end{defi}

In the case of irreducible affine Weyl groups, Shi proved in~\cite[\S8]{Shi88} that $\mathscr S_0$ is a simplex with $|S|$ half-spaces in the above definition. 

\begin{question} The set of supporting hyperplanes of the facets of $\mathscr S_m$ is not the set of hyperplanes $H_\beta$ for $\beta\in \mathcal B_m$  in general.  Describe a set of non-redundant half-spaces for $\mathscr S_m$.
\end{question}

\begin{ex} See Figures~\ref{fig:Aff1}~(b), \ref{fig:Aff2}~(b) and \ref{fig:Hyp1}~(b) where the shaded regions correspond to the corresponding $m$-Shi polyhedron.
\end{ex}

\begin{prop}\label{prop:MainB} Let $m\in \mathbb N$, then the $m$-Shi polyhedron satisfies:
$$
\mathscr S_m\subseteq \bigcup_{w \in L_m} C_{w^{-1}}.
$$
\end{prop}

The proof of this proposition is a consequence of the following lemma.

\begin{lem}\label{lem:Conv1} Let $m\in \mathbb N$ and $w\in W$ such that $\Phi(w^{-1})\cap \mathscr B_m=\emptyset$. Then $w\in L_m$.  In other words:
$
 L_m\supseteq \{w\in W\mid  \Phi(w^{-1})\cap \mathscr B_m=\emptyset\}.
$
\end{lem}
\begin{proof} Assume by contradiction that there is $w\in W\setminus L_m$ such that  $\Phi(w^{-1})\cap \mathscr B_m=\emptyset$; take $w$ of minimal length for this property. Since $e\in L_m$, $\ell(w)>0$. Take $s\in D_L(w)$, so $w=sg$ reduced; in particular with $\ell(g)<\ell(w)$. By Proposition~\ref{prop:InvBasics}, we have 
$$
\Phi(w^{-1})=\Phi(wg^{-1})\sqcup\{g^{-1}(\alpha_s)\}.
$$
Therefore $\Phi(g^{-1})\cap \mathscr B_m=\emptyset$. Hence, by minimality of $w$, $g\in L_m$ and $w=sg\notin L_m$. We conclude that $g^{-1}(\alpha_s)\in \mathscr B_m \cap \Phi(w^{-1}) =\emptyset$, a contradiction.
\end{proof}

\begin{proof}[Proof of Proposition~\ref{prop:MainB}] Assume that $C_{w^{-1}}\subseteq \mathscr S_m$, then $C_{w^{-1}}$ and $C$ are not separated by any $H_\beta$ for any $\beta\in \mathcal B_m$. Therefore $w\in L_m$ by Lemma~\ref{lem:Conv1}.
\end{proof}

The following corollary gives a useful criterion the equality to be satisfied in Proposition~\ref{prop:MainB}. 

\begin{cor}\label{cor:MainB} Let $m\in \mathbb N$, then the $m$-Shi polyhedron satisfies:
\begin{eqnarray*}
\mathscr S_m =  \bigcup_{w \in L_m} C_{w^{-1}}& \iff& L_m =\{w\in W\mid  \Phi(w^{-1})\cap \mathscr B_m=\emptyset\}\\
& \iff& L_m \subseteq \{w\in W\mid  \Phi(w^{-1})\cap \mathscr B_m=\emptyset\}.
\end{eqnarray*}
\end{cor}
\begin{proof}  Assume first that $L_m \subseteq \{w\in W\mid  \Phi(w^{-1})\cap \mathscr B_m=\emptyset\}$.  If $w\in L_m$, then $\Phi(w^{-1})\cap \mathscr B_m=\emptyset$. So $C$ and $C_{w^{-1}}$ are not separated by any $H_\beta$ for any $\beta\in \mathcal B_m$. In particular $C_{w^{-1}}\subseteq \mathscr S_m$. We conclude by Proposition~\ref{prop:MainB}. Now, if there is equality in Proposition~\ref{prop:MainB}, then $C_{w^{-1}}\subseteq \mathscr S_m$ for any $w\in L_m$. Therefore, by definition, $\Phi(w^{-1})\cap \mathscr B_m=\emptyset$ for any $w\in L_m$, which concludes the proof.
\end{proof}

\subsection{Proof of Theorem~\ref{thm:Main2} and Theorem~\ref{thm:Main3}}  \label{se:Convex}

The following two theorems show Theorem~\ref{thm:Main2} and Theorem~\ref{thm:Main3}.

\begin{thm}\label{thm:MainB1} The $0$-Shi polyhedron is:
$$
\mathscr S_0=\bigcup_{w \in L_0} C_{w^{-1}}.
$$
\end{thm}

\begin{thm}\label{thm:MainB2} Assume that $(W,S)$ is an affine Coxeter system and let $m\in \mathbb N$, then the $m$-Shi polyhedron is:
$$
\mathscr S_m=\bigcup_{w \in L_m} C_{w^{-1}}.
$$
\end{thm}

The proofs of both theorems are based on the following discussion. Let $w\in W$ such that $\Phi(w^{-1})\cap \mathscr B_m\not = \emptyset$. 
Then there is $u\in L_m$ and $s\in S$ such that $su\notin L_m$ and 
$$
\beta=u^{-1}(\alpha_s)\in \Phi(u^{-1}s)\cap \Phi(w^{-1}).
$$
By Corollary~\ref{cor:MainB}, we need  to prove, in order to prove Theorems~\ref{thm:MainB1}~and~\ref{thm:MainB2}, that:
\begin{itemize}
\item  $w\notin L_0$ in general;
\item  $w\notin L_m$ if $(W,S)$ is an affine Coxeter system. 
\end{itemize}
Since $L_m$ is closed under taking suffixes (it is a Garside Shadow), we may assume without loss of generality that $w$ of minimal length for $\beta\in \Phi(u^{-1}s)\cap \Phi(w^{-1})$. Therefore there is $t\in D_L(w)$ such that:
\begin{equation}\label{eq:beta}
\beta=u^{-1}(\alpha_s)=-w^{-1}(\alpha_t)\in \Phi^R(w^{-1}).
\end{equation}

Now, since $u\in L_m$ and $su\notin L_m$, there is $r\in D_R(su)$ such that $\dep_\infty(-su(\alpha_r))= m+1$ and $\alpha_s\prec_\dom -su(\alpha_r)$ by Proposition~\ref{cor:Main2}. Therefore, since $\Phi(su)=\{\alpha_s\}\sqcup s(\Phi(u))$, there are $m+1$ distinct positive roots $\beta_0=\alpha_s$, $s(\beta_1), \dots,s(\beta_m) \in s(\Phi(u))$ such that $s(\beta_i) \prec_\dom -su(\alpha_r)$ for all $0\leq i \leq m$.

By Proposition~\ref{prop:DomFu}, we obtain, by multiplying by $wu^{-1}s$ and by Eq.~(\ref{eq:beta}):
$$
wu^{-1}s(\alpha_s)=w(-u^{-1}(\alpha_s))=w(w^{-1}(\alpha_t))=\alpha_t \prec_\dom wu^{-1}s(-su(\alpha_r)) =-w(\alpha_r)
$$
and $wu^{-1}(\beta_i) \prec_\dom -w(\alpha_r)$ for all $1\leq i \leq m$. By Corollary~\ref{cor:DomFu}, we obtain that $-w(\alpha_r)\in \Phi^+$ and so $-w(\alpha_r)\in \Phi^R(w)$. Since $\alpha_t \prec_\dom -w(\alpha_r)$, it means that $w\notin L_0$.

\begin{proof}[Proof of Theorem~\ref{thm:MainB1}] The discussion above shows that if $w\in W$ is such that $\Phi(w^{-1})\cap \mathscr B_0\not = \emptyset$, then $w\notin L_0$. Therefore $L_0 \subseteq \{w\in W\mid  \Phi(w^{-1})\cap \mathscr B_0=\emptyset\}$. We conclude by Corollary~\ref{cor:MainB}.
\end{proof}

\begin{proof}[Proof of Theorem~\ref{thm:MainB2}] Corollary~\ref{cor:MainB}, it remains to show that if $w\in W$ such that $\Phi(w^{-1})\cap \mathscr B_m\not =\emptyset$, then $w\notin L_m$. 

By Corollary~\ref{cor:Para}, we can choose  the $\beta_i's$ in the discussion above such that:
$$
\alpha_s\prec_\dom s(\beta_1) \prec_\dom \cdots \prec_\dom s(\beta_m) \prec_\dom -su(\alpha_r).
$$
Now by multiplying by $wu^{-1}s$ we obtain:
$$
\alpha_t\prec_\dom wu^{-1}(\beta_1) \prec_\dom \cdots \prec_\dom wu^{-1}(\beta_m) \prec_\dom -w(\alpha_r).
$$
By Corollary~\ref{cor:DomFu}, there are all positive roots (since $\alpha_t$ is), hence $\dep_\infty(-w(\alpha_r))\geq m+1$. Since $-w(\alpha_r)\in \Phi^R(w)$, $w\notin L_m$ by Theorem~\ref{thm:Main}. 
\end{proof}

\begin{remark} In the proof of Theorem~\ref{thm:MainB2}, we need the following property (Corollary~\ref{cor:Para}): if $\alpha,\beta,\gamma\in \Phi^+$ are such that $\alpha\preceq_\dom \gamma$ and $\beta\preceq_\dom \gamma$, then either $\alpha\preceq_\dom \beta\preceq_\dom  \gamma$ or $\beta\preceq_\dom \alpha\preceq_\dom  \gamma$. This property arises from the transitivity of the parallelism relation in Euclidean geometry. Unfortunately, it is not true in non-Euclidean space.
\end{remark}

\subsection{Convexity and extended Shi arrangement in indefinite Coxeter system}

An analogous result of Theorem~\ref{thm:MainB2} for indefinite Coxeter systems, i.e., $(W,S)$ is neither finite nor affine: in general the $m$-Shi polyhedron is not equal to the union of the chambers of the inverses of the $m$-low elements. For instance, consider $(W,S)$ be the indefinite Coxeter system whose Coxeter graph is given in Figure~\ref{fig:Hyp2}. We have:
$$
L_0=\{e,1,2,3,13=31\}\textrm{ and } L_1=L_0\cup\{12,21,23,32,213\},
$$
see Figure~\ref{fig:Hyp2}~(a) where $L_0$ is in dark blue (together with the grey chamber which is $C$) and for $L_1$ add the light blue chambers. The $1$-Shi polyhedron is:
\begin{eqnarray*}
\mathscr S_1 &=& 2(H^+_{\alpha_1})\cap 2(H^+_{\alpha_3}) \cap 32(H^+_{\alpha_1})\cap 32(H^+_{\alpha_3})\\
                        &&   \cap 12(H^+_{\alpha_1})\cap 12(H^+_{\alpha_3}) \cap 132(H^+_{\alpha_1})\cap 132(H^+_{\alpha_3}),
\end{eqnarray*}
For instance, $12\in L_1$ but $132\notin L_1$.  Indeed $13(\alpha_2)\in \Phi^R(132)$ but the hyperplane $H_{13(\alpha_2)}=13(H_{\alpha_2})$  is parallel to $H_{\alpha_1}$ and $H_{\alpha_2}$, so $\dep_\infty(13(\alpha_2))=2$. Therefore $23(\alpha_1)=2(\alpha_1)\in \mathcal B_1$ and $\mathscr S_1\subseteq H^+_{2(\alpha_1)}$.

However, the chambers $C_{21}$ and $C_{23}$ are not in $\mathscr S_1$ but corresponds to $12,32\in L_1$, see Figure~\ref{fig:Hyp2}~(b).  Moreover:
$$
\mathscr S_1 = C\cup C_1\cup C_2\cup C_3 \cup C_{12}\cup C_{32}\cup C_{13}\subsetneq \bigcup_{w\in L_1} C_{w^{-1}}, 
$$
and the right-hand side is not convex because $C_{213}$ (in red see Figure~\ref{fig:Hyp2}~(b))  is not included in this union, since $132\notin L_1$.

\begin{question}  Classify the Coxeter graphs for which $\bigcup_{w\in L_m} C_{w^{-1}}$ is convex for all $m\geq 1$. (Or for some $m\geq 1$?)
\end{question}

\begin{question} 
Is the $m$-Shi polyhedron $\mathscr S_m$ the largest convex subset of $\bigcup_{w\in L_m} C_{w^{-1}}$?
\end{question}


\begin{figure}
\subfigure[The $0$ and $1$-Shi arrangements]{
\resizebox{0.7\hsize}{!}{
\begin{tikzpicture}
	[scale=1,
	 q/.style={teal,line join=round},
	 racine/.style={blue},
	 racinesimple/.style={blue},
	 racinedih/.style={blue},
	 sommet/.style={inner sep=2pt,circle,draw=black,fill=blue,thick,anchor=base},
	 rotate=0]
 \tikzstyle{every node}=[font=\small]
\def\grosseursimple{0.025}
\coordinate (ancre) at (5.3,3);
\node[anchor=center,inner sep=0pt] at (ancre) {\includegraphics[width=12cm]{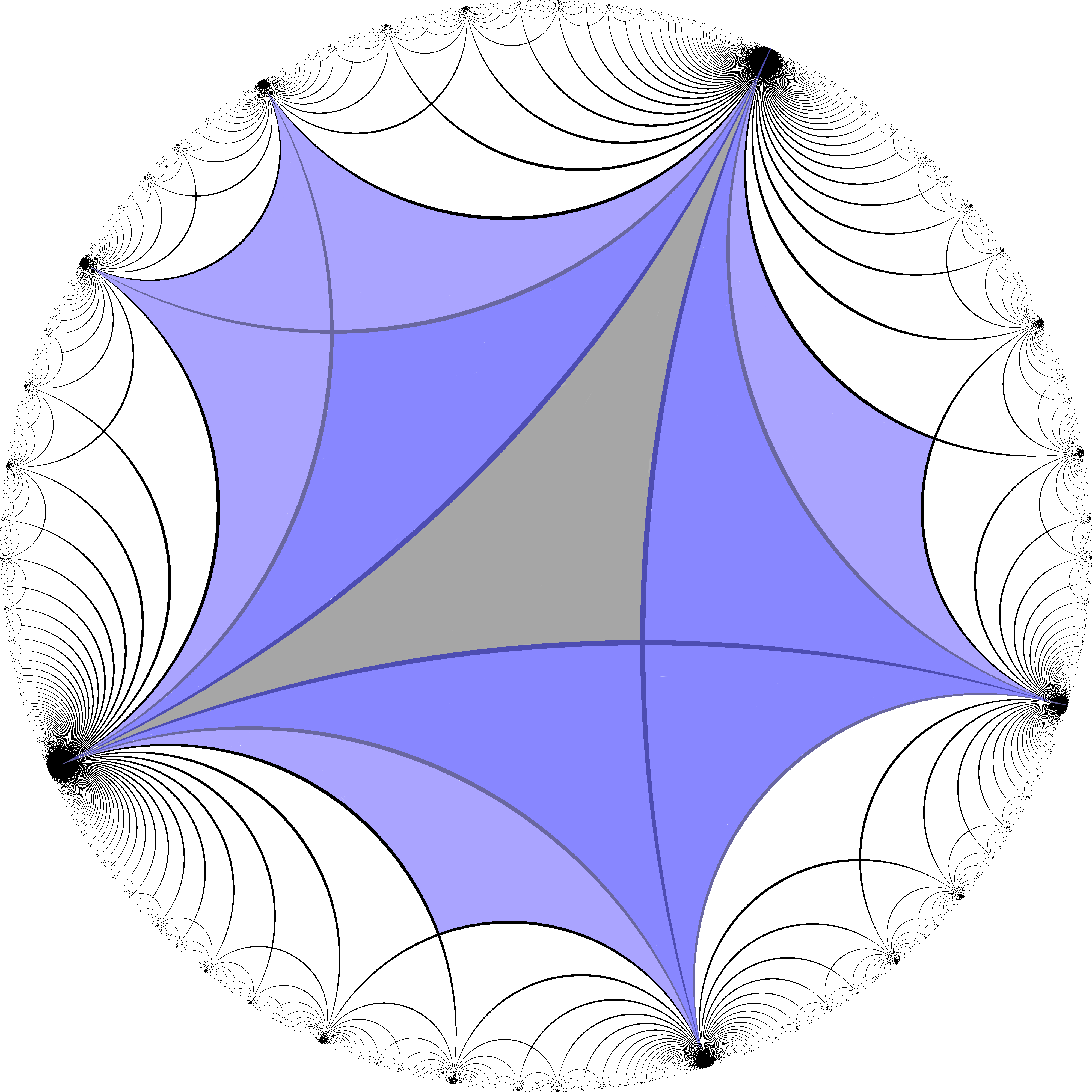}};

\node[sommet,label=above:$2$] (gamma) at ($(ancre)+(5.2,6)$) {};
\node[sommet,label=below right :$1$] (beta) at ($(ancre)+(6.4,5.4)$) {} edge[thick]  node[auto,swap] {$\infty$} (gamma) ;
\node[sommet,label=below left:$3$] (alpha) at ($(ancre)+(4.95,4.6)$) {}
 edge[thick] node[left,swap] {$\infty$} (gamma)  ;

\node[scale=1] at ($(ancre)+(0.2,0)$) {$e$};
\node[scale=1]  at ($(ancre)+(2,0)$) {$1$};
\node[scale=1]  at ($(ancre)+(-0.8,1.6)$) {$2$};
\node[scale=1] at ($(ancre)+(0.2,-2)$) {$3$};
\node[scale=1] at ($(ancre)+(2,-2)$) {$13$};
\node[scale=1]   at ($(ancre)+(-3,1.3)$) {$23$};
\node[scale=1] at ($(ancre)+(-0.6,3.1)$) {$21$};
\node[scale=1] at ($(ancre)+(-3.1,3)$) {$213$};
\node[scale=1]   at ($(ancre)+(3.4,1)$) {$12$};
\node[scale=1]   at ($(ancre)+(-0.8,-3.4)$) {$32$};
\node[scale=1]   at ($(ancre)+(2.9,-3.2)$) {$132$};


\end{tikzpicture}}}
\subfigure[The convexity counterexample for $m=1$: the chambers corresponding to inverses of the $1$-low elements are the shaded chambers except the chamber labelled by $213$ (in red). The $1$-Shi polyhedron is bounded by the red lines.]{ 
\resizebox{0.75\hsize}{!}{
\begin{tikzpicture}
	[scale=1,
	 q/.style={teal,line join=round},
	 racine/.style={blue},
	 racinesimple/.style={blue},
	 racinedih/.style={blue},
	 sommet/.style={inner sep=2pt,circle,draw=black,fill=blue,thick,anchor=base},
	 rotate=0]
 \tikzstyle{every node}=[font=\small]
\def\grosseursimple{0.025}
\coordinate (ancre) at (5.3,3);
\node[anchor=center,inner sep=0pt] at (ancre) {\includegraphics[width=12cm]{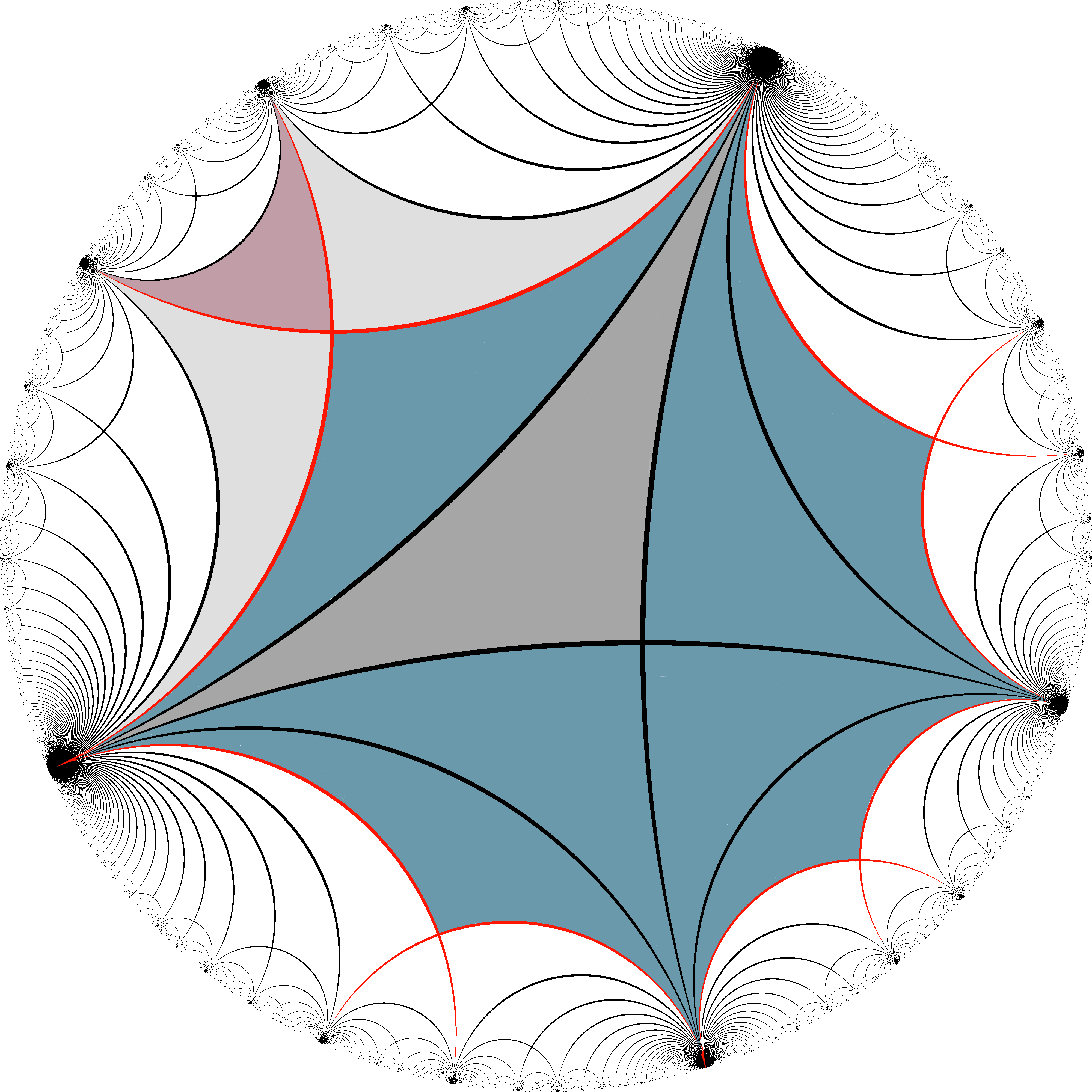}};

\node[scale=1] at ($(ancre)+(0.2,0)$) {$e$};
\node[scale=1]  at ($(ancre)+(2,0)$) {$1$};
\node[scale=1]  at ($(ancre)+(-0.8,1.6)$) {$2$};
\node[scale=1] at ($(ancre)+(0.2,-2)$) {$3$};
\node[scale=1] at ($(ancre)+(2,-2)$) {$13$};
\node[scale=1]   at ($(ancre)+(-3,1.3)$) {$23$};
\node[scale=1] at ($(ancre)+(-0.6,3.1)$) {$21$};
\node[scale=1] at ($(ancre)+(-3.1,3)$) {$213$};
\node[scale=1]   at ($(ancre)+(3.4,1)$) {$12$};
\node[scale=1]   at ($(ancre)+(-0.8,-3.4)$) {$32$};
\node[scale=1]   at ($(ancre)+(2.9,-3.2)$) {$132$};
\node[scale=1,color=red!60!black] at ($(ancre)+(-3,5.5)$) {$2(H_{\alpha_3})$};
\node[scale=1,color=red!60!black] at ($(ancre)+(-5.4,3.5)$) {$2(H_{\alpha_1})$};
\node[scale=1,color=red!60!black] at ($(ancre)+(6.7,1)$) {$12(H_{\alpha_1})$};
\node[scale=1,color=red!60!black] at ($(ancre)+(6.2,2.7)$) {$12(H_{\alpha_3})$};
\node[scale=1,color=red!60!black] at ($(ancre)+(-2.6,-5.8)$) {$32(H_{\alpha_1})$};
\node[scale=1,color=red!60!black] at ($(ancre)+(-0.8,-6.4)$) {$32(H_{\alpha_3})$};
\node[scale=1,color=red!60!black] at ($(ancre)+(5.6,-3.8)$) {$132(H_{\alpha_1})$};
\node[scale=1,color=red!60!black] at ($(ancre)+(4.2,-5.2)$) {$132(H_{\alpha_3})$};

\end{tikzpicture}}}
\caption{The $0$ and $1$-Shi arrangements and a counterexample of convexity for the indefinite Coxeter system whose Coxeter graph is on top righthand side of the picture.}
\label{fig:Hyp2}
\end{figure}
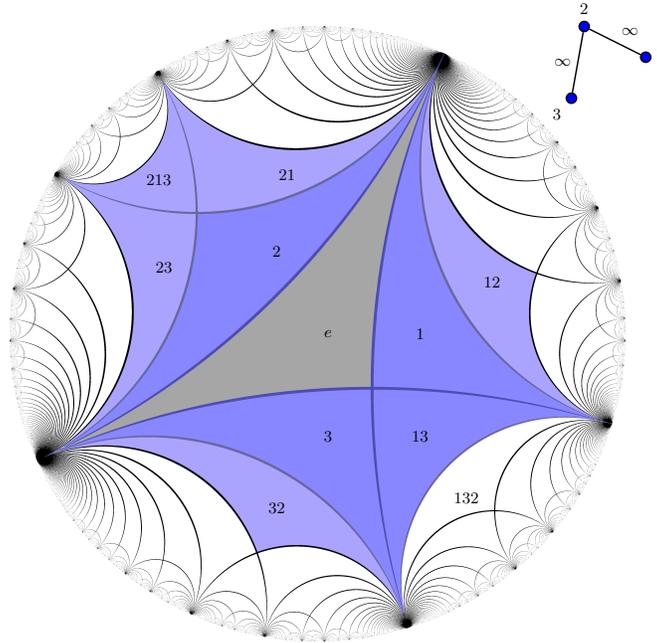
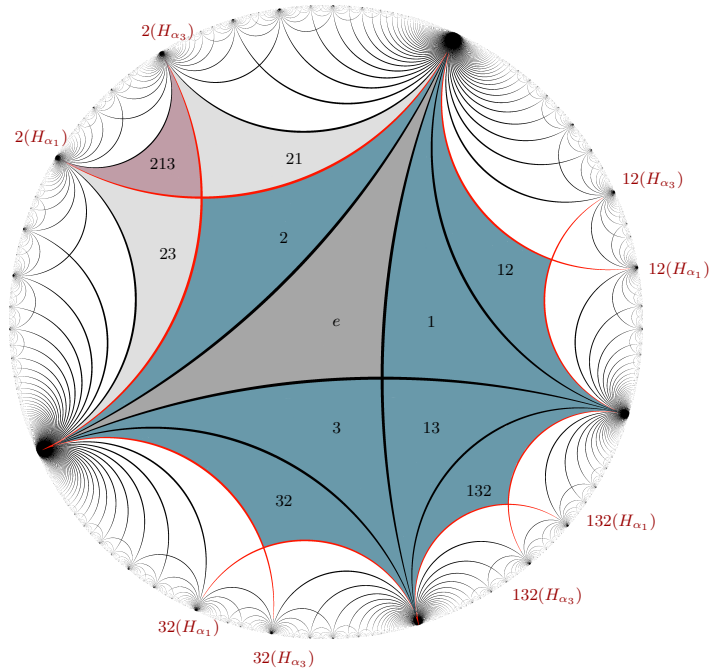

Following up the two above questions, we may also ask if the definition of $m$-small roots might be modified to obtain subarrangements of the Coxeter arrangement for which the results in Theorem~\ref{thm:Main1} holds and such that the union of the inverses of the minimal elements form a convex set. We propose hereafter such a conjectural family of subsarrangements. Fix $m\in \mathbb{N}$. 
We define a positive root $\alpha\in \Phi^{+}$ to be {\em $m$-medium} if there is no chain $\alpha_{0}\prec_{\dom}\alpha_{1}\prec_{\dom} \ldots \prec_{\dom} \alpha_{m+1}=\alpha$  of positive roots in dominance order.  Note that $0$-medium roots are the same as small (that is,  $0$-small) roots and that all $m$-small roots are $m$-medium.  Also, in affine type, dominance order is a coproduct (disjoint union) of chains, so $m$-medium roots are the same as  $m$-small roots.
 For finite rank Coxeter systems, it  can be deduced from the finiteness of the number of $m$-small roots for all $m$ that there are only finitely many $m$-medium roots  for all $m$  (see \cite{Dy19-1}[\S13.3-\S13.6] for a discussion, without proof, of stronger facts). 
  
  In analogy to the definition (see~\ref{ss:LowElements}) of $m$-low elements  of $W$, one may define an element $w$ of $W$ to be {\em $m$-medium-low} if $\Phi^{1}(w)$ consists of $m$-medium roots.  Similarly,  one may define the $m$-medium Shi arrangement as the arrangement of hyperplanes $H_{\beta}$ corresponding to $m$-medium roots $\beta$.

\begin{question}  Which of the main results  proved in \cite{Dy21,DyHo16} and this article for $m$-small roots, $m$-low elements and $m$-Shi arrangements also apply to their
 analogues defined above using $m$-medium roots instead of $m$-small roots.  In  particular, we ask whether the analogue of Theorem~\ref{thm:Main3} with $L_{m}$ replaced by the set $L_{m}'$ of medium-low elements holds for arbitrary Coxeter systems. 
\end{question}

The example below discusses the situation for  the (non-affine) group in Figure~\ref{fig:Hyp2}, for which the statement  corresponding to Theorem~\ref{thm:Main3}  is false for $L_{1}$ but true for $L_{1}'$. 

\begin{ex} Consider the $1$-medium Shi arrangement in Figure~\ref{fig:Hyp2}(a). This contains the hyperplanes of the $1$-Shi arrangement together with one additional hyperplane
separating the chamber labelled $13$ from that labelled $132$; the corresponding positive root $\beta$ dominates the two mutually orthogonal simple roots, so $\beta$ is not $1$-small, but those two simple roots are not comparable in dominance order, so $\beta$ is $1$-medium. 

Each region of the $1$-medium Shi arrangement in Figure~\ref{fig:Hyp2}(a)  contains a minimum chamber. The set of  elements of $W$ indexing these chambers is the set of $1$-medium-low elements, which is a Garside shadow. In Figure~\ref{fig:Hyp2}(b), the chambers corresponding to the inverses of the $1$-medium low elements are those corresponding to the inverses of the  $1$-low elements, together with the red chamber labeled as $213$, which make the union of these chambers convex. 
\end{ex}

\subsection{Enumeration of $m$-low elements}   By Theorem~\ref{thm:Main1}, we know that for any $m\in\mathbb N$  the enumeration of the set  $L_m$ of $m$-low elements is the same as the enumeration of the regions of the $m$-Shi arrangement $\shi_m(W,S)$. In particular the following are known:

\begin{itemize}
\item In the case of a finite Coxeter system $(W,S)$:  $L_m=W$ for all $m\in\mathbb N$. 

\item  For an affine Coxeter system $(W,S)$ with underlying finite Coxeter system $(W_0,S_0)$, denote $h$ to be the Coxeter number of $W_0$ and $n=|S_0|$.  Then for any $m\in\mathbb N$ we have:
\begin{equation}\label{eq:AffCard}
|L_m| = \big ((m+1)h+1\big)^n.
\end{equation}
In the case of the classical Shi arrangement, i.e. $m=0$, this formula was proven by Shi in~\cite{Shi88}. For an arbitrary extended Shi arrangement, this formula was first proven for classical affine Coxeter systems by Athanasiadis~\cite[Theorem~4.3]{At00} (see also \S7.2 of Athanasiadis' thesis) and for arbitrary affine Coxeter systems as a consequence of a result of Yoshinaga~\cite{Yo04}, see also~\cite[Theorem 5.1.16]{Arm09} for more details. In 2015,  Thiel~\cite{Th15,Th16} recovered this results by applying Shi's original method.

\item For indefinite types, we do not have closed formulas; see Table~\ref{tab:ex} for some examples. 
\end{itemize}

In the case of an affine Weyl group $(W,S)$ with underlying finite Weyl group $W_0$, Shi introduced the so-called Shi arrangement in~\cite{Shi88} to prove that the number of regions in the $0$-Shi arrangement  is $(h+1)^{n}$, where $h$ is the Coxeter number of $W_0$ and $n$ is the rank of $W_0$. Shi's proof in \cite{Shi88} is based on the following steps:
\begin{enumerate}
\item $\shi_0(W,S)$ is introduced and its regions are labelled by sign-types;
\item Any Shi-region contains a unique element of minimal length, which turns out to be a low element thanks to \cite{ChHo22}, or the more general Theorem~\ref{thm:MainA} above.
\item The union of the chambers $C_{w^{-1}}$ for $w\in L_0$ is a convex subset of $\mathcal U$, which is the Euclidean space in that case. 
\item Shi proves the formula for $m=0$ in the case of an affine Coxeter system by counting the number of chambers inside the Shi polyhedron $\mathscr S_0$.
\end{enumerate}
 The key result is that, in this case, the Shi polyhedron $\mathscr S_0$ is a dilatation of the fundamental chamber $C$ by a factor of $(h+1)$; see \cite[Lemma 8.5 and Corollary 8.6]{Shi88}.  The same method is used by Thiel in his thesis~\cite[Chapter 5]{Th15}  to recover the formula in Eq.~(\ref{eq:AffCard}): write $\mathscr S_m$ in term of the ribbons used in  \cite[Lemma 8.5]{Shi88} with $k_\alpha = -(m+1)$ and the `$m$' in Shi's article replaced by $(m+1)h+1$; this  implies  that $\mathscr S_m$ is a dilation of $C$ by a factor of $(m+1)h+1$.

Unfortunately, for indefinite Coxeter systems, this method cannot be applied since $\mathscr S_m$ is not a dilation of the fundamental chamber $C$; see for instance Figures~\ref{fig:Hyp1}~and~\ref{fig:Hyp2}. Actually, the number of $m$-small roots is not known even in hyperbolic Coxeter systems.

\begin{question} For an indefinite Coxeter system, enumerate $\Sigma_m$ and $L_m$ for all $m\in\mathbb N$.
\end{question}

\def\arraystretch{1.25}
\begin{center} 
\begin{tabular}{|c|cc|cc|cc|}
\hline
Coxeter graph of $(W,S)$ & $|\Sigma_0|$ & $|L_0|$  & $|\Sigma_1|$ & $|L_1|$ &$|\Sigma_2|$ & $|L_2|$  \\
  \hline \hline
  \begin{tikzpicture}
	[scale=1,
	 sommet/.style={inner sep=2pt,circle,draw=black,fill=blue,thick,anchor=base},
	 rotate=0,
	 baseline = 0]
 \tikzstyle{every node}=[font=\small]
\coordinate (ancre) at (0,-0.0);
\node  at ($(ancre)+(0,0.6)$) {};
\node[sommet] (a2) at ($(ancre)+(0.25,0.5)$) {};
\node[sommet] (a3) at ($(ancre)+(0.5,0)$) {} edge[thick] node[auto,swap,right] {$\infty$} (a2) ;
\node[sommet] (a4) at (ancre) {} edge[thick] node[auto,swap,below] {$\infty$} (a3) edge[thick] node[auto,swap,left] {$\infty$} (a2);
\end{tikzpicture}
&3&$4$&$9$&$10$&$21$&$22$\\
  \hline
    \begin{tikzpicture}
	[scale=1,
	 sommet/.style={inner sep=2pt,circle,draw=black,fill=blue,thick,anchor=base},
	 rotate=0,
	 baseline = 0]
 \tikzstyle{every node}=[font=\small]
\coordinate (ancre) at (0,-0.0);
\node  at ($(ancre)+(0,0.6)$) {};
\node[sommet] (a2) at ($(ancre)+(0.25,0.5)$) {};
\node[sommet] (a3) at ($(ancre)+(0.5,0)$) {} edge[thick] node[auto,swap,right] {$\infty$} (a2) ;
\node[sommet] (a4) at (ancre) {}  edge[thick] node[auto,swap,left] {$\infty$} (a2);
\end{tikzpicture}
&$3$&$5$&$7$&$10$&$14$&$19$\\
  \hline
    \begin{tikzpicture}
	[scale=1,
	 sommet/.style={inner sep=2pt,circle,draw=black,fill=blue,thick,anchor=base},
	 rotate=0,
	 baseline = 0]
 \tikzstyle{every node}=[font=\small]
\coordinate (ancre) at (0,-0.0);
\node  at ($(ancre)+(0,0.6)$) {};
\node[sommet] (a2) at ($(ancre)+(0.25,0.5)$) {};
\node[sommet] (a3) at ($(ancre)+(0.5,0)$) {} edge[thick] node[auto,swap,right] {} (a2) ;
\node[sommet] (a4) at (ancre) {} edge[thick] node[auto,swap,below] {$4$} (a3) edge[thick] node[auto,swap,left] {} (a2);
\end{tikzpicture}
&$7$&$18$&$13$&$40$&$20$&$70$\\
  \hline
   \begin{tikzpicture}
	[scale=1,
	 sommet/.style={inner sep=2pt,circle,draw=black,fill=blue,thick,anchor=base},
	 rotate=0,
	 baseline = 0]
 \tikzstyle{every node}=[font=\small]
\coordinate (ancre) at (0,-0.0);
\node  at ($(ancre)+(0,0.6)$) {};
\node[sommet] (a2) at ($(ancre)+(0.25,0.5)$) {};
\node[sommet] (a3) at ($(ancre)+(0.5,0)$) {} edge[thick] node[auto,swap,right] {} (a2) ;
\node[sommet] (a4) at (ancre) {}  edge[thick] node[auto,swap,left] {$7$} (a2);
\end{tikzpicture}
&$12$&$40$&$18$&$72$&$24$&$110$\\
  \hline
   \begin{tikzpicture}
	[scale=1,
	 sommet/.style={inner sep=2pt,circle,draw=black,fill=blue,thick,anchor=base},
	 rotate=0,
	 baseline = 0]
 \tikzstyle{every node}=[font=\small]
\coordinate (ancre) at (0,-0.0);
\node  at ($(ancre)+(0,0.6)$) {};
\node[sommet] (a1) at ($(ancre)+(0,0.5)$) {};
\node[sommet] (a2) at ($(ancre)+(0.5,0.5)$) {} edge[thick] node[auto,swap,right] {} (a1) ;
\node[sommet] (a3) at ($(ancre)+(0.5,0)$) {} edge[thick] node[auto,swap,right] {$4$} (a2) ;
\node[sommet] (a4) at (ancre) {} edge[thick] node[auto,swap,below] {$4$} (a3) edge[thick] node[auto,swap,left] {$4$} (a1);
\end{tikzpicture}
&$19$&$134$&$43$&$387$&$94$&$997$\\
  \hline
\end{tabular}
  \captionof{table}{Some examples  for $|\Sigma_m|$ and $|L_m|$ ($m=0,1,2$)}\label{tab:ex}
\end{center}


\bibliographystyle{plain}

\begin{thebibliography}{10}

\bibitem{AbBr08}
Peter Abramenko and Kenneth~S. Brown.
\newblock {\em Buildings}, volume 248 of {\em Graduate Texts in Mathematics}.
\newblock Springer, New York, 2008.
\newblock Theory and applications.

\bibitem{Arm09}
Drew Armstrong.
\newblock Generalized noncrossing partitions and combinatorics of {C}oxeter
  groups.
\newblock {\em Mem. Amer. Math. Soc.}, 202(949):x+159, 2009.

\bibitem{ArRh11}
Drew Armstrong and Brendon Rhoades.
\newblock The {S}hi arrangement and the {I}sh arrangement.
\newblock {\em Trans. Amer. Math. Soc.}, 364(3):1509--1528, 2012.

\bibitem{At00}
Christos~A. Athanasiadis.
\newblock Deformations of {C}oxeter hyperplane arrangements and their
  characteristic polynomials.
\newblock {\em Advances Studies in Pure Mathemartics}, 27:1--26, 2000.

\bibitem{At04}
Christos~A. Athanasiadis.
\newblock Generalized {C}atalan numbers, {W}eyl groups and arrangements of
  hyprplanes.
\newblock {\em Bull. London Math. Soc.}, 36:294--302, 2004.

\bibitem{At05}
Christos~A Athanasiadis.
\newblock On a refinement of the generalized {C}atalan numbers for {W}eyl
  groups.
\newblock {\em Transactions of the American Mathematical Society},
  357:179--196, 2005.

\bibitem{Gu15}
Misha Belolipetsky and Paul~E. Gunnells.
\newblock Kazhdan {L}usztig cells in infinite {C}oxeter groups.
\newblock {\em J. Gen. Lie Theory Appl.}, 9(S1):Art. ID S1--002, 4, 2015.

\bibitem{BjBr05}
Anders Bj\"{o}rner and Francesco Brenti.
\newblock {\em Combinatorics of Coxeter Groups}, volume 231 of {\em Graduate
  Texts in Mathematics}.
\newblock Springer, New York, 2005.

\bibitem{Bo68}
N.~Bourbaki.
\newblock {\em {G}roupes et alg\`ebres de {L}ie, {C}hapitres {IV--VI}}.
\newblock Actualit\'es Scientifiques et Industrielles, No. 1337. Hermann,
  Paris, 1968.

\bibitem{BrHo93}
Brigitte Brink and Robert~B. Howlett.
\newblock A finiteness property and an automatic structure for {C}oxeter
  groups.
\newblock {\em Math. Ann.}, 296:179--190, 1993.

\bibitem{CePa02}
Paola Cellini and Paolo Papi.
\newblock {a}d-nilpotent ideals of a {B}orel subalgebra {II}.
\newblock {\em Journal of Algebra}, 258(1):112--121, 2002.

\bibitem{ChHo22}
Nathan Chapelier-Laget and Christophe Hohlweg.
\newblock Shi arrangements and low elements in affine {C}oxeter groups.
\newblock to appear in {\em Journal of Canadian Math. Soc.},  2024.

\bibitem{Chb20}
Balthazar Charles.
\newblock Low elements and small inversion sets are in bijection in rank 3
  {C}oxeter groups.
\newblock {\em S\'{e}m. Lothar. Combin.}, 85B, Art. 51, 12, 2021


\bibitem{Da08}
M.~W. Davis.
\newblock {\em The Geometry and Topology of Coxeter Groups}, volume~32.
\newblock London Mathematical Society Monographs, 2008.

\bibitem{DDH14}
Patrick Dehornoy, Matthew Dyer, and Christophe Hohlweg.
\newblock Garside families in {A}rtin-{T}its monoids and low elements in
  {C}oxeter groups.
\newblock {\em Comptes Rendus Mathematique}, 353:403--408., 2015.

\bibitem{Dy90}
Matthew~J. Dyer.
\newblock Reflection subgroups of {C}oxeter systems.
\newblock {\em J. Algebra}, 135(1):57--73, 1990.

\bibitem{Dy91}
Matthew~J. Dyer.
\newblock On the ``{B}ruhat graph'' of a {C}oxeter system.
\newblock {\em Compositio Math.}, 78(2):185--191, 1991.

\bibitem{Dy92}
Matthew~J. Dyer.
\newblock Hecke algebras and shellings of {B}ruhat intervals. {II}. {T}wisted
  {B}ruhat orders.
\newblock In {\em Kazhdan-Lusztig theory and related topics (Chicago, IL,
  1989)}, volume 139 of {\em Contemp. Math.}, pages 141--165. Amer. Math. Soc.,
  Providence, RI, 1992.

\bibitem{Dy94-1}
Matthew~J. Dyer.
\newblock Bruhat intervals, polyhedral cones and {K}azhdan-{L}usztig-{S}tanley
  polynomials.
\newblock {\em Math. Z.}, 215(2):223--236, 1994.

\bibitem{Dy10}
Matthew~J. Dyer.
\newblock On conjugacy of abstract bases of root systems of {C}oxeter groups.
\newblock {\em Rocky Mountain J. Math.}, 48(7):2223--2287, 2018.

\bibitem{Dy19-1}
Matthew~J. Dyer.
\newblock {Imaginary cone and reflection subgroups of Coxeter groups}.
\newblock {\em Dissertationes Mathematicae 545}, 545:1--117, 2019.

\bibitem{Dy19}
Matthew~J. Dyer.
\newblock On the weak order of {C}oxeter groups.
\newblock {\em Canadian Journal of Mathematics}, 71(2):299--336, 2019.


\bibitem{Dy21}
Matthew~J. Dyer.
\newblock {$n$-low elements and maximal rank $k$ reflection subgroups of
  Coxeter groups}.
\newblock {\em Journal of Algebra}, 607:139--180, 2022.

\bibitem{DyHo16}
Matthew~J. Dyer and Christophe Hohlweg.
\newblock {Small roots, low elements, and the weak order in Coxeter groups}.
\newblock {\em Advances in Mathematics}, 301:739--784, 2016.

\bibitem{DyHoRi16}
Matthew~J. Dyer, Christophe Hohlweg, and Vivien Ripoll.
\newblock {Imaginary cones and limit roots of infinite Coxeter groups}.
\newblock {\em Mathematische Zeitschrift}, 284(3):715--780, 2016.

\bibitem{Ed09}
Tom Edgar.
\newblock Dominance and regularity in {C}oxeter groups.
\newblock {\em PhD Thesis, University of Notre Dame}, 2009.
\newblock http://etd.nd.edu/ETD-db/.

\bibitem{Fi19}
Susanna Fishel.
\newblock A survey of the {S}hi arrangement.
\newblock In {\em Recent trends in algebraic combinatorics}, volume~16 of {\em
  Assoc. Women Math. Ser.}, pages 75--113. Springer, Cham, 2019.

\bibitem{Fu12}
Xiang Fu.
\newblock The dominance hierarchy in root systems of {C}oxeter groups.
\newblock {\em J. Algebra}, 366:187--204, 2012.

\bibitem{Fu18}
Xiang Fu.
\newblock On paired root systems of {C}oxeter groups.
\newblock {\em J. Aust. Math. Soc.}, 105(3):347--365, 2018.

\bibitem{Gu10}
Paul~E. Gunnells.
\newblock Automata and cells in affine {W}eyl groups.
\newblock {\em Represent. Theory}, 14:627--644, 2010.

\bibitem{HoLaRi14}
C.~Hohlweg, J.-P. Labb{\'e}, and V.~Ripoll.
\newblock Asymptotical behaviour of roots of infinite {C}oxeter groups.
\newblock {\em Canad. J. Math.}, 66:323--353, 2014.

\bibitem{HoLa16}
Christophe Hohlweg and Jean-Philippe Labb\'e.
\newblock On {I}nversion {S}ets and the {W}eak {O}rder in {C}oxeter {G}roups.
\newblock {\em European Journal of Combinatorics}, 55:1--19, 2016.

\bibitem{HoNaWi16}
Christophe Hohlweg, Philippe Nadeau, and Nathan Williams.
\newblock Automata, reduced words and {G}arside shadows in {C}oxeter groups.
\newblock {\em J. Algebra}, 457:431--456, 2016.

\bibitem{HPR20}
Christophe {Hohlweg}, Jean-Philippe {Pr{\'e}aux}, and Vivien Ripoll.
\newblock {On the Limit Set of Root Systems of Coxeter Groups acting on
  Lorentzian spaces}.
\newblock {\em Communication in Algebra}, 48(3):1281--1304, April 2020.

\bibitem{Hu90}
James~E. Humphreys.
\newblock {\em Reflection groups and Coxeter groups}, volume~29.
\newblock Cambridge University Press, Cambridge, 1990.

\bibitem{Kr09}
D.~Krammer.
\newblock The conjugacy problem for {C}oxeter groups.
\newblock {\em Groups Geom. Dyn.}, 3(1):71--171, 2009.
\newblock \url{http://dx.doi.org/10.4171/GGD/52}.

\bibitem{OsPr22}
Damian Osajda and Piotr Przytycki.
\newblock Coxeter groups are biautomatic.
\newblock {\em arXiv:2206.07804}, 2022.

\bibitem{PaYa21}
James Parkinson and Yeeka Yau.
\newblock Cone types, automata, and regular partitions in {C}oxeter groups.
\newblock {\em To appear in Advances in Mathematics}, 2021.

\bibitem{PrYa23}
Piotr Przytycki and Yeeka Yau.
\newblock A Pair of Garside Shadows.
\newblock {\em 	arXiv:2310.06267}, 2023.

\bibitem{Shi86}
Jian~Yi Shi.
\newblock {\em The {K}azhdan-{L}usztig cells in certain affine {W}eyl groups},
  volume 1179 of {\em Lecture Notes in Mathematics}.
\newblock Springer-Verlag, Berlin, 1986.

\bibitem{Shi88}
Jian~Yi Shi.
\newblock Sign types corresponding to an affine {W}eyl group.
\newblock {\em J. London Math. Soc. (2)}, 35(1):56--74, 1987.

\bibitem{Sp82}
T.~A. Springer.
\newblock Some remarks on involutions in {C}oxeter groups.
\newblock {\em Comm. Algebra}, 10(6):631--636, 1982.

\bibitem{sage}
William~A. Stein et~al.
\newblock {\em {S}age {M}athematics {S}oftware ({V}ersion 9.5)}.
\newblock The Sage Development Team, 2011.
\newblock \url{http://www.sagemath.org}.

\bibitem{Th15}
Marko Thiel.
\newblock {\em Catalan combinatorics of crystallographic root systems}.
\newblock PhD thesis, Universit{\"a}t {W}ien,
  http://user.math.uzh.ch/thiel/files/Thesis.pdf, 2015.

\bibitem{Th16}
Marko Thiel.
\newblock From {A}nderson to zeta.
\newblock {\em Adv. in Appl. Math.}, 81:156--201, 2016.

\bibitem{Ti74}
Jacques Tits.
\newblock {\em Buildings of spherical type and finite {BN}-pairs}.
\newblock Lecture Notes in Mathematics, Vol. 386. Springer-Verlag, Berlin-New
  York, 1974.

\bibitem{Vi71}
\`E.~B. Vinberg.
\newblock Discrete linear groups that are generated by reflections.
\newblock {\em Izv. Akad. Nauk SSSR Ser. Mat.}, 35:1072--1112, 1971.
\newblock translation by P. Flor, IOP Science.

\bibitem{Vi93}
{\`E}.~B. Vinberg and O.~V. Shvartsman.
\newblock Discrete groups of motions of spaces of constant curvature.
\newblock In {\em Geometry, {II}}, volume~29 of {\em Encyclopaedia Math. Sci.},
  pages 139--248. Springer, Berlin, 1993.

\bibitem{Yo04}
Masahiko Yoshinaga.
\newblock Characterization of a free arrangement and conjecture of {E}delman
  and {R}einer.
\newblock {\em Invent. Math.}, 157(2):449--454, 2004.

\end{thebibliography}

\end{document}